\documentclass[12pt]{amsart}
\usepackage{latexsym,amsmath, amscd, amssymb, amsthm, euscript, mathrsfs, url}
\usepackage{enumerate}
\usepackage[all]{xy}
\newtheorem{thm}[subsection]{Theorem}

\newtheorem{thm/def}[subsection]{Theorem/Definition}
\newtheorem{cor}[subsection]{Corollary}
\newtheorem{conjecture}[subsection]{Conjecture}
\newtheorem{lem}[subsection]{Lemma}

\newtheorem{prop}[subsection]{Proposition}

\theoremstyle{definition}

\newtheorem{defn}[subsection]{Definition}
\newtheorem{assumption}[subsection]{Assumption}

\theoremstyle{definition}

\theoremstyle{definition}

\newtheorem{rem}[subsection]{Remark}

\newtheorem{example}[subsection]{Example}
\numberwithin{equation}{subsection}

\newtheorem{pg}[subsection]{}

\newcommand{\divCZX}{\mathrm{Div}_{\mathscr{C} \cup \overline{\mathcal{Z}}/\overline{Z}}^0(\overline{\mathfrak{X}})}
\newcommand{\dlim}{\mathop{\varinjlim}\limits}
\newcommand{\ilim}{\mathop{\varprojlim}\limits}

\begin{document}

\title{Albanese and Picard 1-Motives in Positive Characteristic}

\author{Peter Mannisto}

\begin{abstract}
We investigate how to define 1-motives $M_{(c)}^1(X)$, $M_{(c)}^{2d-1}(X)$ for $X$ a variety over a perfect field $k$ of positive characteristic, such that the $\ell$-adic realizations of these 1-motives are canonically isomorphic to $H_{(c)}^1(X_{\overline{k}},\mathbb{Z}_\ell(1))$ and $H_{(c)}^{2d-1}(X_{\overline{k}},\mathbb{Z}_\ell(d))$ respectively.  This is the analogue in positive characteristic of previous results of Barbieri-Viale and Srinivas, except that we only consider the $\ell$-adic realization but also consider compactly supported cohomology.  The 1-motives $M^1(X)$ and $M_c^1(X)$ can be defined by standard techniques, and indeed this case is probably well known.  But the 1-motives $M^{2d-1}(X)$ and $M_c^{2d-1}(X)$ require stronger tools, namely a strong version of de Jong's alterations theorem and some cycle class theory on smooth Deligne-Mumford stacks which may be of independent interest.  Unfortunately, we only succeed in defining $M^{2d-1}(X)$ and $M_c^{2d-1}(X)$ when $X$ is a variety over an algebraically closed field, and only up to isogeny.  Nevertheless, as a corollary to our definition of these 1-motives, for a variety $X$ over a finite field $k$ we deduce independence of $\ell$ for the cohomology groups $H_{(c)}^{i}(X_{\overline{k}},\mathbb{Q}_\ell)$ for $i = 1,2d-1$.
\end{abstract}

\maketitle

\section{Introduction}
\begin{pg}
Let $X$ be a separated scheme of finite type over a field $k$ of characteristic $p \geq 0$, and fix a separable algebraic closure $k \hookrightarrow k^s$.  Then for any prime $\ell \not= p$ we can associate to $X$ its $\ell$-adic \'{e}tale cohomology groups
$H^i(X_{k^s},\mathbb{Q}_\ell)$ and  $H_c^i(X_{k^s},\mathbb{Q}_\ell)$, where the subscript $c$ indicates compactly supported cohomology.   The philosophy of mixed motives states that there should exist a category $\mathscr{MM}(k)$ of mixed motives over $k$ which is universal with respect to these different cohomology theories (and also for the mixed Hodge structure and de Rham theory if $k \subseteq \mathbb{C}$, and crystalline cohomology if $k$ is perfect of positive characteristic).  For the \'{e}tale cohomology theories, this would take the form of realization functors $$V_\ell: \mathscr{MM}(k) \rightarrow \mathrm{Rep}_{Gal(k^s/k)}(\mathbb{Q}_\ell)$$ to the category of $\mathbb{Q}_\ell$-representations of $Gal(k^s/k)$, such that there exist functors
\begin{equation*}
M^i(-), \ M_c^i(-): (Sch/k)^{op} \longrightarrow \mathscr{MM}(k),
\end{equation*}
with the property that
\begin{align*}
V_\ell M^i(X) &\cong H^i(X_{k^s},\mathbb{Q}_\ell) \ \ \ \mathrm{and} \\
V_\ell M_c^i(X) &\cong H_c^i(X_{k^s},\mathbb{Q}_\ell)
\end{align*}
functorially in $X$.
\end{pg}
\begin{pg}
While the theory of mixed motives remains largely conjectural, the category of 1-motives (mixed motives of level $\leq 1$) does a good job of explaining cohomological phenomena in dimension and codimension 1, at least over a perfect field.
\begin{defn} \cite[D\'{e}finition 10.1.1]{hodge3}
Let $k$ be a field.  The category of (free) 1-motives over $k$, denoted $\mathscr{M}^1(k)$, is the category of 2-term complexes
\begin{equation*}
[L \rightarrow G]
\end{equation*}
of commutative group schemes over $k$, where $L$ is an \'{e}tale-locally constant sheaf such that $L(k^s)$ is a free finitely generated abelian group, and $G$ is a semi-abelian variety over $k$.  The morphisms in $\mathscr{M}^1(k)$ are the morphisms of complexes of sheaves.
\end{defn}
In \cite[Sect. 10]{hodge3}, Deligne constructs realization functors $T_\ell: \mathscr{M}^1(k) \rightarrow \mathrm{Rep}_{Gal(k^s/k)}(\mathbb{Z}_\ell)$ (for $\ell \not= p$), and if $k \subseteq \mathbb{C}$, a realization functor from 1-motives to mixed Hodge structures.  Deligne conjectured in \cite[10.4.1]{hodge3} that certain mixed Hodge structures associated to a separated finite type scheme over $\mathbb{C}$ arise naturally from 1-motives; in particular, he conjectured that the mixed Hodge structures $H^1(X,\mathbb{Z}(1))$ and $H^{2d-1}(X,\mathbb{Z}(d))/\mathrm{torsion}$ occur as the Hodge realizations of 1-motives $M^1(X)$ and $M^{2d-1}(X)$, respectively, defined purely algebraically.  This special case of Deligne's conjecture was solved in \cite{albpic}, and much more general cases were handled in  \cite{BRS} and \cite{BVK}. 
\end{pg}

For \'{e}tale cohomology, one can make the following conjecture, which is an $\ell$-adic analogue of this special case of Deligne's conjectures on 1-motives.  Note that we restrict ourselves to the case where $k$ is perfect; as noted in \cite[p. 3]{ramachandran}, it is not clear that this conjecture should be true for non-perfect fields.
\begin{conjecture}\label{mainconj}
Fix a perfect field $k$ of characteristic $p \geq 0$, and let $Sch/k$ denote the category of separated finite type $k$-schemes.  Then there exist functors
\begin{equation*}
M^{1}(-), \ M_c^1(-): (Sch/k)^{op} \longrightarrow \mathscr{M}^1(k), 
\end{equation*}
with the property that for all primes $\ell \not= p$,
\begin{align*}
T_\ell M^1(X) &\cong H^1(X_{\overline{k}},\mathbb{Z}_\ell(1)) \ \ \mathrm{and} \\
T_\ell M_c^1(X) &\cong H_c^1(X_{\overline{k}},\mathbb{Z}_\ell(1))
\end{align*}
functorially in $X$.

In addition, let $Sch_d/k \subset Sch/k$ be the full subcategory of $d$-dimensional separated finite type $k$-schemes.  Then there exist functors
\begin{equation*}
M^{2d-1}(-), \ M_c^{2d-1}(-): (Sch_d/k)^{op} \longrightarrow \mathscr{M}^1(k),
\end{equation*}
with the property that for all primes $\ell \not= p$,
\begin{align*}
T_\ell M^{2d-1}(X) &\cong H^{2d-1}(X_{\overline{k}},\mathbb{Z}_\ell(d))/\mathrm{torsion} \ \ \mathrm{and} \\
T_\ell M_c^{2d-1}(X) &\cong H_c^{2d-1}(X_{\overline{k}},\mathbb{Z}_\ell(d))/\mathrm{torsion}
\end{align*}
functorially in $X$. 
\end{conjecture}

\begin{pg}
In \cite{albpic}, Barbieri-Viale and Srinivas solve this conjecture in the case where $\mathrm{char \ }k = 0$, for non-compactly supported cohomology (the compactly supported case is no harder, and can be done by similar methods).  They call the 1-motive $M^1(X)$ the \emph{Picard} 1-motive of $X$, and $M^{2d-1}(X)$ the \emph{Albanese} 1-motive of $X$, as these 1-motives generalize the classical theory of the Picard and Albanese variety.  In addition, the papers \cite{crys} and \cite{ramachandran} provide (independently) definitions of $M^1(X)$ for $k$ perfect of positive characteristic.  But the full conjecture above, particularly defining the 1-motives $M_{(c)}^{2d-1}(X)$ in positive characteristic, has not been dealt with to our knowledge.
\end{pg}
\begin{pg}
In this paper we investigate how to define the 1-motives $M_{(c)}^1(X)$, $M_{(c)}^{2d-1}(X)$ over a perfect field of positive characteristic.  In Section 6 we provide definitions of the Picard 1-motives $M^1(X)$ and $M_c^1(X)$.  As indicated above, the definition of $M^1(X)$ has previously appeared in \cite{crys} and \cite{ramachandran}.  We generalize this by defining a 1-motive $M_{D,E}^1(X)$ associated to any triple $(\overline{X},D,E)$ consisting of a proper scheme $\overline{X}$ and two disjoint closed subschemes $D,E \subset \overline{X}$.  The 1-motive $M_{D,E}^1(\overline{X})$ realizes the relative cohomology group $H^1(\overline{X}_{\overline{k}} - E_{\overline{k}},D_{\overline{k}},\mathbb{Z}_\ell(1))$, which we also write as $H_{D,E}^1(X,\mathbb{Z}_\ell(1))$.  Then to define $M^1(X)$ and $M_c^1(X)$ we choose a compactification $X \hookrightarrow \overline{X}$ with closed complement $D$, and define $M^1(X) := M_{\emptyset,D}^1(\overline{X})$ and $M_c^1(X) = M_{D,\emptyset}^1(\overline{X})$.  Of course, we show that these definitions are independent of the choice of compactification.
\end{pg}
\begin{rem}
Because we do not deal with $p$-adic realizations in this paper, we only show that $M_{D,E}^1(X)$ is well-defined as an object of $\mathscr{M}^1(k)[1/p]$, the category of 1-motives up to $p$-isogeny.  The $p$-adic realization of $M^1(X)$ is discussed in \cite{crys}; we hope to deal with the $p$-adic realization of $M_{D,E}^1(\overline{X})$ in future work.
\end{rem}
\begin{pg}
Defining the Albanese 1-motives $M^{2d-1}(X)$ and $M_c^{2d-1}(X)$ turns out to be harder, and most of Sections 2-4 are devoted to preliminary facts on Picard groups and divisors for Deligne-Mumford stacks that we will use in our construction.  Even so, our constructions are only valid for an \emph{algebraically closed} field of positive characteristic, and even then these 1-motives will only be well-defined up to \emph{isogeny} of 1-motives; the difficulty is due to lack of resolution of singularities as will be discussed below.  In particular, we only succeed in proving Conjecture \ref{mainconj} (up to isogeny) in case $k$ is algebraically closed.  In summary, the following is our main result:
\begin{thm}
Let $k$ be a perfect field of characteristic $p > 0$, and let $Sch/k$ be the category of separated finite type $k$-schemes.  Then there exist functors
\begin{equation*}
M^1(-), M_c^1(-) : (Sch/k)^{op} \longrightarrow \mathscr{M}^1(k)[1/p]
\end{equation*}
with the property that 
\begin{align*}
T_\ell M^1(X) &\cong H^1(X_{\overline{k}},\mathbb{Z}_\ell(1)) \ \ \mathrm{and} \\
T_\ell M_c^1(X) &\cong H_c^1(X_{\overline{k}},\mathbb{Z}_\ell(1))
\end{align*}
functorially in $X$, for $\ell \not= p$.

Next assume $k$ is algebraically closed, and let $Sch_d/k$ be the full subcategory of $d$-dimensional separated finite type $k$-schemes.  Then there exist functors
\begin{align*}
M^{2d-1}(-), M_c^{2d-1}(-) : (Sch_d/k)^{op} \longrightarrow \mathscr{M}^1(k) \otimes \mathbb{Q}
\end{align*}
such that 
\begin{align*}
T_\ell M^{2d-1}(X) \otimes \mathbb{Q} &\cong H^{2d-1}(X,\mathbb{Q}_\ell(d)) \ \ \mathrm{and} \\
T_\ell M_c^{2d-1}(X) \otimes \mathbb{Q} &\cong H_c^{2d-1}(X,\mathbb{Q}_\ell(d)).
\end{align*}
functorially in $X$, for $\ell \not= p$.
\end{thm}
\end{pg}
\subsection{Description of the 1-motive $M_{D,E}^1(\overline{X})$}
We informally describe the 1-motives constructed in this paper; more details are to be found in Sections 6-8.  

Fix a perfect field $k$, and let $\overline{X}$ be a proper reduced $k$-scheme with disjoint reduced closed subschemes $D$ and $E$.  The results of \cite{dejong} (as explained in \cite[Thm. 4.7]{desc}) imply that there exists a proper hypercover
\begin{equation*}
\pi_{\bullet}: \overline{X}_\bullet \rightarrow \overline{X}
\end{equation*}
with each $\overline{X}_i$ proper and smooth.  Let $D_\bullet = \pi_\bullet^{-1}(D)_{red}$ and $E_\bullet = \pi_\bullet^{-1}(E)_{red}$; we may assume that $D_\bullet$ and $E_\bullet$ are strict normal crossings divisors.

Let $p_\bullet: \overline{X}_\bullet \rightarrow \mathrm{Spec \ }k$ be the structure morphism, and $i_\bullet: D_\bullet \hookrightarrow \overline{X}_\bullet$ the inclusion.  Then consider the sheaf
\begin{equation*}
\mathbf{Pic}_{\overline{X}_\bullet,D_\bullet} := R^1(p_\bullet)_*(\mathrm{ker}(\mathbb{G}_{m,\overline{X}_\bullet} \rightarrow (i_\bullet)_*\mathbb{G}_{m,D_\bullet}))
\end{equation*}
on $(Sch/k)_{fppf}$.  Informally, $\mathbf{Pic}_{\overline{X}_\bullet,D_\bullet}$ classifies isomorphism classes of pairs $(\mathscr{L}^{\bullet},\sigma)$, where $\mathscr{L}^\bullet$ is a line bundle on $\overline{X}_\bullet$ and $\sigma$ is an isomorphism of $\mathscr{L}^{\bullet}\vert_{D_\bullet}$ with $\mathcal{O}_{D_\bullet}$.  A straightforward reduction (\ref{simprelpic}) from well-known representability results shows that $\mathbf{Pic}_{\overline{X}_\bullet,D_\bullet}$ is representable by a locally finite type commutative $k$-group scheme.  Moreover, let $\mathbf{Pic}_{\overline{X}_\bullet,D_\bullet}^{0,red}$ denote the reduction of the connected component of the identity of $\mathbf{Pic}_{\overline{X}_\bullet,D_\bullet}$.  Then (loc. cit.) $\mathbf{Pic}_{\overline{X}_\bullet,D_\bullet}^{0,red}$ is a semi-abelian variety.  

Next, for any closed subscheme $C$ of a scheme $X$, let $\mathrm{Div}_{C}(X)$ be the lattice of (Weil) divisors on $X$ with support contained in $C$.  Then for the simplicial closed subscheme $E_\bullet$ of $\overline{X}_\bullet$, we define
\begin{equation*}
\mathrm{Div}_{E_\bullet}(\overline{X}_\bullet) := \mathrm{Ker}(p_1^* - p_2^*: \mathrm{Div}_{E_0}(\overline{X}_0) \rightarrow \mathrm{Div}_{E_1}(\overline{X}_1)),
\end{equation*}
where $p_1,p_2: \overline{X}_1 \rightarrow \overline{X}_0$ are the simplicial structure maps from $\overline{X}_1$ to $\overline{X}_0$.  Because $E_\bullet$ and $D_\bullet$ are disjoint, there is a natural map $\mathrm{Div}_{E_\bullet}(\overline{X}_\bullet) \rightarrow \mathrm{Pic}(\overline{X}_\bullet,D_\bullet)$ (defined in Paragraph \ref{pg65}).  We define $\mathrm{Div}_{E_\bullet}^0(\overline{X}_\bullet)$ to be the subgroup mapping into $\mathrm{Pic}^0(\overline{X}_\bullet,D_\bullet)$.  More generally, we can define an \'{e}tale $k$-group scheme $\mathbf{Div}_{E_\bullet}^0(\overline{X}_\bullet)$ with $\overline{k}$-points $\mathrm{Div}^0_{E_{\overline{k},\bullet}}(\overline{X}_{\overline{k},\bullet})$  and $Gal(\overline{k}/k)$ acting in the obvious way.  Then there is a natural map of group schemes
\begin{equation*}
\mathbf{Div}_{E_\bullet}^0(\overline{X}_\bullet) \rightarrow \mathbf{Pic}^{0,red}_{\overline{X}_\bullet,D_\bullet}.
\end{equation*}
\begin{defn}
With notation as above, we define the 1-motive $M_{D,E}^1(\overline{X})$ to be 
\begin{equation*}
M_{D,E}^1(\overline{X}) := [\mathbf{Div}^0_{E_\bullet}(\overline{X}_\bullet) \rightarrow \mathbf{Pic}_{\overline{X}_\bullet,D_\bullet}^{0,red}].
\end{equation*}
\end{defn}
In Section 6 we show that this definition is independent of the choice of hypercover $\overline{X}_\bullet$ and functorial in the triple $(\overline{X},D,E)$.  

\begin{pg}
As special cases of the above construction, consider a separated finite type $k$-scheme $X$.  Choose a compactification $X \hookrightarrow \overline{X}$, with closed complement $C \subset \overline{X}$.  Then we define
\begin{equation*}
M^1(X) := M_{\emptyset,C}^1(\overline{X}) = [\mathbf{Div}^0_{C_\bullet}(\overline{X}_\bullet) \rightarrow \mathbf{Pic}^{0,red}_{\overline{X}_\bullet}]
\end{equation*}
and
\begin{equation*}
M_c^1(X) := M_{C,\emptyset}^1(\overline{X}) = [0 \rightarrow \mathbf{Pic}^{0,red}_{\overline{X}_\bullet,C_\bullet}].
\end{equation*}
We show in Section 6 that $M^1(X)$ is contravariantly functorial for arbitrary morphisms, while $M_c^1(X)$ is contravariantly functorial for proper morphisms.
\end{pg}
\subsection{Definition of the 1-motives $M^{2d-1}(X)$ and $M_c^{2d-1}(X)$}
We next define the 1-motives $M^{2d-1}(X)$ and $M_c^{2d-1}(X)$.  As noted before, we will assume that our field $k$ is algebraically closed of positive characteristic.  Choose a compactification $X \hookrightarrow \overline{X}$.  Then by \cite[7.3]{dejong} there exists a sequence of maps
\begin{equation*}
\overline{\mathfrak{X}} \stackrel{p}{\longrightarrow} \overline{X}'' \stackrel{q}{\longrightarrow} \overline{X}' \stackrel{r}{\longrightarrow} \overline{X},
\end{equation*}
satisfying the following conditions:
\begin{enumerate}
\item $r$ is purely inseparable and surjective, therefore a universal homeomorphism;
\item $q$ is proper and birational;
\item $\overline{\mathfrak{X}}$ is a smooth proper Deligne-Mumford stack (in fact a global quotient $[U/G]$ of a smooth proper $k$-scheme $U$ by a finite group $G$) and $p$ identifies $\overline{X}''$ with the coarse moduli space of $\overline{\mathfrak{X}}$.
\end{enumerate}
We let $\overline{\pi}: \overline{\mathfrak{X}} \rightarrow \overline{X}$ be the composition $r \circ q \circ p$.  Let $C = \overline{X} - X$ and $\mathcal{C} = \overline{\mathfrak{X}} - \mathfrak{X}$.  The key property of $\overline{\pi}$ we will use is that there exists an open dense subset $U \subset X$ with the property that $\pi^{-1}(U) \rightarrow U$ induces an isomorphism $\mathbb{Q}_{\ell,U} \stackrel{\sim}{\longrightarrow} R\pi_* \mathbb{Q}_{\ell,\pi^{-1}(U)}$ in $D_c^b(U)$; see Proposition \ref{ratcohthesame}.  Choose such an open subset $U$, and let $Z = X - U$, $\mathcal{Z} : = \mathfrak{X} - \pi^{-1}(U)$.  Finally, let $\overline{Z}$ (resp. $\overline{\mathcal{Z}}$) be the closure of $Z$ in $\overline{X}$ (resp. of $\mathcal{Z}$ in $\overline{\mathfrak{X}}$).  In summary, we have diagrams
\begin{equation*}
\xymatrix{
{\mathfrak{X}} \ar@{^{(}->}[r]^{\alpha'} \ar[d]^{\pi} & \overline{\mathfrak{X}} \ar[d]^{\overline{\pi}} & \mathcal{C} \ar@{_{(}->}[l]_{\beta'} \ar[d]^{\pi_C} \\
X \ar@{^{(}->}[r]^{\alpha} & \overline{X} & C \ar@{_{(}->}[l]_{\beta}}
\end{equation*}
and
\begin{equation*}
\xymatrix{
\mathcal{U} \ar@{^{(}->}[r]^{j'} \ar[d]^{\pi_U} & \mathfrak{X} \ar[d]^{\pi} & \mathcal{Z} \ar@{_{(}->}[l]_{i'} \ar[d]^{\pi_Z} \\
U \ar@{^{(}->}[r]^{j} & X & Z. \ar@{_{(}->}[l]_{i}}
\end{equation*}

We can define sheaves $\mathbf{Pic}_{\overline{\mathfrak{X}}}$ and $\mathbf{Pic}_{\overline{\mathfrak{X}},\mathcal{C}}$ by the same formulas as in the case of schemes, so that $\mathbf{Pic}_{\overline{\mathfrak{X}}}$ classifies isomorphism classes of line bundles on $\overline{\mathfrak{X}}$ and $\mathbf{Pic}_{\overline{\mathfrak{X}},\mathcal{C}}$ classifies isomorphism classes of pairs $(\mathscr{L},\sigma)$ where $\mathscr{L}$ is a line bundle on $\overline{\mathfrak{X}}$ and $\sigma$ is an isomorphism of $\mathscr{L}\vert_{\mathcal{C}}$ with $\mathcal{O}_{\mathcal{C}}$.  The sheaves $\mathbf{Pic}_{\overline{\mathfrak{X}}}$ and $\mathbf{Pic}_{\overline{\mathfrak{X}},\mathcal{C}}$ are both representable by locally finite type commutative $k$-group schemes (\ref{stackrelpic}).  Moreover, $\mathbf{Pic}_{\overline{\mathfrak{X}}}^{0,red}$ and $\mathbf{Pic}_{\overline{\mathfrak{X}},\mathcal{C}}^{0,red}$ are both semi-abelian varieties.

\begin{pg}
To define the 1-motive $M_c^{2d-1}(X)$, first consider a divisor $D \in \mathrm{Div}_{\mathcal{C} \cup \overline{\mathcal{Z}}}(\overline{\mathfrak{X}})$.  Such a divisor can be uniquely written as $D = D_1 + D_2$, with $D_1$ supported on $\mathcal{C}$ and $D_2$ supported on $\overline{\mathcal{Z}}$.  Then we let $\mathrm{Div}_{\mathcal{C} \cup \overline{\mathcal{Z}}/\overline{Z}}^0(\overline{\mathfrak{X}})$ be the group of divisors $D = D_1 + D_2$ supported on $\mathcal{C} \cup \overline{\mathcal{Z}}$ such that 
\begin{enumerate}
\item $D$ maps to 0 in $NS(\overline{\mathfrak{X}})$, and
\item $D_1$ maps to 0 under the proper pushforward $\overline{\pi}_*: \mathrm{Div}_{\overline{\mathcal{Z}}}(\overline{\mathfrak{X}}) \rightarrow \mathrm{Div}_{\overline{Z}}(\overline{X})$.
\end{enumerate}
Let $\mathbf{Div}^0_{\mathcal{C} \cup \overline{\mathcal{Z}}/\overline{Z}}(\overline{\mathfrak{X}})$ be the $k$-group scheme with $k$-points $\mathrm{Div}_{\mathcal{C} \cup \overline{\mathcal{Z}}/\overline{Z}}^0(\overline{\mathfrak{X}})$.  We then define
\begin{equation*}
M_c^{2d-1}(X) := [\mathbf{Div}^0_{\mathcal{C} \cup \overline{\mathcal{Z}}/\overline{Z}}(\overline{\mathfrak{X}}) \rightarrow \mathbf{Pic}_{\overline{\mathfrak{X}}}^{0,red}]^\vee
\end{equation*}
where the superscript $^\vee$ indicates taking the Cartier dual (see Section 5 for background on Cartier duality for 1-motives).
\end{pg}
\begin{pg}
To define the 1-motive $M^{2d-1}(X)$, let $\mathrm{Div}_{\mathcal{Z}}(\overline{\mathfrak{X}})$ be the group of divisors on $\overline{\mathfrak{X}}$ supported on $\mathcal{Z}$ (note that these are equal to the divisors supported on $\overline{\mathcal{Z}}$ which are disjoint from $\mathcal{C}$).  Since these divisors are disjoint from $\mathcal{C}$, there is a natural map $\mathrm{Div}_{\mathcal{Z}}(\overline{\mathfrak{X}}) \rightarrow \mathbf{Pic}_{\overline{\mathfrak{X}},\mathcal{C}}^{0,red}$.  We let $\mathrm{Div}_{\mathcal{Z}/Z}^0(\overline{\mathfrak{X}})$ be the subgroup of $\mathrm{Div}_{\mathcal{Z}}(\overline{\mathfrak{X}})$ of divisors $D$ such that
\begin{enumerate}
\item $D$ maps to 0 in $NS(\overline{\mathfrak{X}},\mathcal{C})$, and
\item $D$ maps to 0 under the proper pushforward $\pi_*: \mathrm{Div}_{\mathcal{Z}}(\overline{\mathfrak{X}}) \rightarrow \mathrm{Div}_{Z}(\overline{X})$.
\end{enumerate}
If $\mathbf{Div}_{\mathcal{Z}/Z}^0(\overline{\mathfrak{X}})$ is the associated group scheme, we define
\begin{equation*}
M^{2d-1}(X) = [\mathbf{Div}_{\mathcal{Z}/Z}^0(\overline{\mathfrak{X}}) \rightarrow \mathbf{Pic}_{\overline{\mathfrak{X}},\mathcal{C}}^{0,red}]^\vee.
\end{equation*}

Secions 7 and 8 go into  more detail on the construction of $M^{2d-1}(X)$ and $M_c^{2d-1}(X)$, showing that the preceding definitions yield functors $M^{2d-1}(-), \ M_c^{2d-1}(-): (Sch_d/k)^{op} \rightarrow \mathscr{M}^1(k) \otimes \mathbb{Q}$ compatible with the $\ell$-adic realization functors.
\end{pg}
\subsection{Other results}
In order to define $M^{2d-1}(X)$ and $M_c^{2d-1}(X)$ we need some results on cycle classes for Deligne-Mumford stacks which aren't in the literature.  Some of these results may be of independent interest and are highlighted below.

In Section 2 we extend the theory of 1-motivic sheaves \cite[App. C]{BVK} to show that the Picard sheaf of a smooth Artin stack is 1-motivic.  See Section 2 for the definition of a 1-motivic sheaf; beyond Section 2 we will only use the following corollary:
\begin{cor} (of Proposition \ref{onemotivic}) Let $\mathfrak{X}$ be a smooth Artin stack of finite type over an algebraically closed field $k$, and with quasi-compact separated diagonal.  Then there exists a divisible group $\mathrm{Pic}^0(\mathfrak{X})$, a finitely generated group $NS(\mathfrak{X})$, and a sequence
\begin{equation*}
0 \rightarrow \mathrm{Pic}^0(\mathfrak{X}) \rightarrow \mathrm{Pic}(\mathfrak{X}) \rightarrow NS(\mathfrak{X}) \rightarrow 0
\end{equation*}
which becomes exact after inverting $p := \mathrm{char \ }k$.
\end{cor}
Presumably this is true even without inverting $\mathrm{char \ }k$, although we don't know how to prove it.

In Section 3 we review the theory of Weil and Cartier divisors on a Deligne-Mumford stack $\mathfrak{X}$.  In addition, we prove the following:
\begin{prop} (Proposition \ref{cartpic} in text) Let $\mathfrak{X}$ be a geometrically reduced, separated finite type Deligne-Mumford stack over a field $k$.  Then the quotient $\mathrm{Pic}(\mathfrak{X})/\mathrm{CaCl}(\mathfrak{X})$ is a finite group.
\end{prop}
This has the following corollary:
\begin{cor} (Corollary \ref{pic0} in text) Let $\mathfrak{X}$ be a smooth proper Deligne-Mumford stack over an algebraically closed field $k$.  Then every element of $\mathrm{Pic}^0(\mathfrak{X})$ is represented by a Weil divisor.
\end{cor}
Section 4 develops the theory of cycle class maps $CH^d(\mathfrak{X}) \rightarrow H^{2d}(\mathfrak{X},\mathbb{Z}_\ell(d))$ for a smooth separated Deligne-Mumford stack $\mathfrak{X}$.  We proceed in the same manner as the article \cite[Cycle]{SGA4h}, and the results of that article are used repeatedly in Section 4.

In Section 5 we review the necessary background on 1-motives.  All of the results in this section can be found in several other sources, for example \cite[App. C]{BVK}.  In Sections 6 through 8 we define the 1-motives $M_{(c)}^{1}(X)$ and $M_{(c)}^{2d-1}(X)$ as discussed above.

We conclude the paper with the short Section 9, which gives the following easy consequence of our work on 1-motives:
\begin{prop}(Proposition 9.1 in text)\label{independenceofl}
Let $X_0$ be a $d$-dimensional separated scheme of finite type over a finite field $\mathbb{F}_q$, and let $X = X_0 \times_{\mathbb{F}_q} \overline{\mathbb{F}}_q$.  Let $f: X \rightarrow X$ be an endomorphism, and for $i$ between 0 and $2d$, define the polynomials
\begin{align*}
P_\ell^i(f,t) &:= \mathrm{det}(1 - tf\mid H^i(X,\mathbb{Q}_\ell)) \ \ \mathrm{and} \\
P_{\ell,c}^i(f,t) &:= \mathrm{det}(1 - tf \mid H_c^i(X,\mathbb{Q}_\ell)).
\end{align*}
(for $P_{\ell,c}^i(f,t)$ we assume $f$ is proper).  Then for $i = 0, 1, 2d-1, 2d$, these polynomials have integer coefficients independent of $\ell$. 
\end{prop}
Probably the only new case here is $i = 2d-1$, although the other cases aren't clearly stated in the literature.  Of course, when $X_0$ is smooth and proper, this corollary is a special case of the main result of \cite{WeilII}.

Using known results on trace formulas (\cite[5.4.5]{fuj} and \cite[Thm 1.1]{chern}), we get the following corollary in the case of surfaces:
\begin{cor}(Corollary \ref{surfaces} in text)\label{surfaces2}
Let $X_0$ be a 2-dimensional separated finite type $\mathbb{F}_q$-scheme, and let $X = X_0 \times_{\mathbb{F}_q} \overline{\mathbb{F}}_q$.  If $f: X \rightarrow X$ is any proper endomorphism, then for all values of $i$, the polynomial $P_{\ell,c}^i(f,t)$ has rational coefficients independent of $\ell$.  If $f: X \rightarrow X$ is any quasi-finite endomorphism, then the polynomial $P_{\ell}^i(f,t)$ has rational coefficients independent of $\ell$ for all $i$.
\end{cor}
In particular, if $f = F$ is the Frobenius endomorphism, then $P_{\ell}^i(F,t)$ and $P_{\ell,c}^i(F,t)$ have rational coefficients independent of $\ell$ (and hence integer coefficients since the roots of these polynomials are algebraic integers \cite[4.2]{illusie}).

\section{Preliminaries on Picard Functors of Smooth Stacks}
\begin{pg}
Let $\mathfrak{X}$ be an Artin stack of finite type over a field $k$, with quasi-compact and separated diagonal.  Let $\pi: \mathfrak{X} \rightarrow \mathrm{Spec \ }k$ be the structure morphism, and let $\mathbb{G}_{m,\mathfrak{X}}$ be the sheaf on $(Sch/\mathfrak{X})_{fppf}$ sending $T$ to $\Gamma(T,\mathcal{O}_{T}^\times)$.  Recall that the Picard functor $\mathbf{Pic}_{\mathfrak{X}/k} \in Sh(Sch/k)_{fppf}$ is defined to be the sheaf $R^1\pi_*(\mathbb{G}_{m,\mathfrak{X}})$, or equivalently, the fppf-sheafification of the functor
\begin{equation*}
Y \mapsto \mathrm{Pic}(\mathfrak{X} \times_{k} Y).
\end{equation*}
We then have the following representability results in case $\mathfrak{X}$ is proper, due to Brochard \cite{repres}, \cite{finite}:
\end{pg}
\begin{thm}\label{stackrepres}
Let $\mathfrak{X}$ be a proper Artin stack over the field $k$.  Then the following hold:
\begin{enumerate}
\item The sheaf $\mathbf{Pic}_{\mathfrak{X}/k}$ is representable by a locally finite type commutative group scheme over $k$ \cite[2.3.7]{finite}.
\item If $\mathbf{Pic}^0_{\mathfrak{X}/k}$ denotes the connected component containing the identity of $\mathbf{Pic}_{\mathfrak{X}/k}$, then we have an exact sequence of group schemes
\begin{equation*}
0 \rightarrow \mathbf{Pic}^0_{\mathfrak{X}/k} \rightarrow \mathbf{Pic}_{\mathfrak{X}/k} \rightarrow \mathbf{NS}_{\mathfrak{X}/k} \rightarrow 0,
\end{equation*}
defining the group scheme $\mathbf{NS}_{\mathfrak{X}/k}$. Furthermore, $\mathbf{NS}_{\mathfrak{X}/k}$ is an \'{e}tale group scheme, and $\mathbf{NS}_{\mathfrak{X}/k}(\overline{k})$ is a finitely generated abelian group for any algebraic closure $k \hookrightarrow \overline{k}$ \cite[3.4.1]{finite}.
\item Assume in addition that $\mathfrak{X}$ is smooth, and $k$ is a perfect field.  Then the reduced group scheme $\mathbf{Pic}^{0,red}_{\mathfrak{X}/k}$ is an abelian variety \cite[4.2.2]{repres}.
\end{enumerate}
\end{thm}
\begin{pg}
We write $\mathrm{Pic}^0(\mathfrak{X})$ and $NS(\mathfrak{X})$ for the $k$-points of $\mathbf{Pic}^0_{\mathfrak{X}/k}$ and $\mathbf{NS}_{\mathfrak{X}/k}$, respectively (note however that if $k$ is not algebraically closed then an element of $\mathrm{Pic}^0(\mathfrak{X})$ might not be given by a line bundle on $\mathfrak{X}$).   For our application to 1-motives, we need to make sense of the groups $\mathrm{Pic}^0(\mathfrak{X})$ and $NS(\mathfrak{X})$ when $\mathfrak{X}$ is smooth, but not necessarily proper.  In this case we cannot expect the Picard sheaf $\mathbf{Pic}_{\mathfrak{X}/k}$ to be representable, but it satisfies a weak form of representability which is sufficient for our purposes.  This can be summarized in the statement that $\mathbf{Pic}_{\mathfrak{X}/k}$ is a \emph{1-motivic sheaf}, at least when $k$ is perfect and after inverting the characteristic $p := \mathrm{char \ k}$ in Hom-groups.  We will prove that $\mathbf{Pic}_{\mathfrak{X}/k}$ is 1-motivic for any smooth Artin stack $\mathfrak{X}$ over $k$ with quasi-compact separated diagaonal.  The argument is a straightforward generalization of \cite[3.4]{BVK}; we include the details for the convenience of the reader.  Recall the following definitions, following \cite[Sect. 3]{BVK}:
\end{pg}

\begin{defn}
Let $k$ be a perfect field, and $(Sm/k)_{et}$ the category of smooth separated $k$-schemes, with the \'{e}tale topology.  We denote by $Sh(Sm/k)_{et}[1/p]$ the category of sheaves of abelian groups on this site, with $\mathrm{Hom}(\mathscr{F},\mathscr{G})$ replaced by $\mathrm{Hom}(\mathscr{F},\mathscr{G})[1/p]$.  We say that $\mathscr{F} \in Sh(Sm/k)_{et}[1/p]$ is \emph{discrete} if it is locally constant for the \'{e}tale topology, and $\mathscr{F}(\overline{k})$ is a finitely generated abelian group.
\end{defn}

\begin{defn}\cite[3.2.1]{BVK}
Let $\mathscr{F} \in Sh(Sm/k)_{et}[1/p]$ be a sheaf of abelian groups as above.  We say that $\mathscr{F}$ is a \emph{1-motivic sheaf} if there exists a semi-abelian variety $G$ and a morphism $f: G \rightarrow \mathscr{F}$ of sheaves on $(Sm/k)_{et}$, such that $\mathrm{ker \ }f$ and $\mathrm{coker \ }f$ are discrete sheaves.  We denote by $Shv_1(k)[1/p] \subset Sh(Sm/k)_{et}[1/p]$ the full subcategory of 1-motivic sheaves.
\end{defn}
\begin{rem}
In the theory of this section we will invert $p$ in all Hom-groups.  Most of our propositions are false as stated if we do not do this.  If we wish to avoid inverting $p$ in Hom-groups, we probably must use a finer topology like the fppf topology (see Remark \ref{mustinvert}).  But this seemingly forces us to replace $Sm/k$ with the bigger category $Sch/k$, which breaks down the proof that $\mathbf{Pic}_{\mathfrak{X}/k}$ is 1-motivic.  See \cite{bert} for some discussion of this point.
\end{rem}
\begin{example}\label{examp1}
Let $X$ be a smooth variety over a field $k$, and suppose that $X$ embeds into a smooth proper variety $\overline{X}$ with complement $D:= \overline{X} - X$.  Then for every smooth scheme $U/k$, we have an exact sequence
\begin{equation*}
\mathrm{Div}_{D \times U}(\overline{X} \times U) \rightarrow \mathrm{Pic}(\overline{X} \times U) \rightarrow \mathrm{Pic}(X \times U) \rightarrow 0.
\end{equation*}
Here $\mathrm{Div}_{D \times U}(\overline{X} \times U)$ denotes the free abelian group of Weil divisors on $\overline{X} \times U$ supported on $D \times U$.  Let $\mathbf{Pic}_{X/k}$ be the lisse-\'{e}tale Picard sheaf of $X/k$, that is, the sheafification in $(Sm/k)_{et}$ of the functor
\begin{equation*}
U \mapsto \mathrm{Pic}(X \times U).
\end{equation*}
Then the above exact sequence shows that we have an exact sequence in $(Sm/k)_{et}$
\begin{equation*}
\mathbf{Div}_{D}(\overline{X}) \rightarrow \mathbf{Pic}_{\overline{X}/k} \rightarrow \mathbf{Pic}_{X/k} \rightarrow 0,
\end{equation*}
where $\mathbf{Div}_D(\overline{X})$ is the \'{e}tale-locally constant sheaf of divisors on $\overline{X}$ supported on $D$.  By Proposition \ref{picrepres}(c) below, we conclude that $\mathbf{Pic}_{X/k}$ is 1-motivic. 
\end{example}
We have the following key facts about 1-motivic sheaves:
\begin{prop}\label{picrepres} \mbox{} 
\begin{enumerate}[(a)]
\item Given a 1-motivic sheaf $\mathscr{F}$, there exists a unique (up to isomorphism) semi-abelian variety $G$ together with a map $b: G \rightarrow \mathscr{F}$ such that $\mathrm{ker \ }b$ is torsion-free.  We say that such a pair $(G,b)$ is \emph{normalized}.
\item Given 1-motivic sheaves $\mathscr{F}_1$, $\mathscr{F}_2$ and normalized morphisms $b_i: G_i \rightarrow \mathscr{F}_i$ for $i = 1,2$.  Then for any morphism of sheaves $f: \mathscr{F}_1 \rightarrow \mathscr{F}_2$, there exists a unique morphism  of group schemes $\varphi_f: G_1 \rightarrow G_2$ making the diagram
\begin{equation*}
\begin{CD}
G_1 @> b_1 >> \mathscr{F}_1 \\
@V \varphi_f VV @V f VV \\
G_2 @> b_2 >> \mathscr{F}_2
\end{CD}
\end{equation*}
commute.
\item The full subcategory $Sh_1(k)[1/p] \subset Sh(Sm/k)_{et}[1/p]$ is stable under kernels,  cokernels, and extensions.
\end{enumerate}
\end{prop}
\begin{proof}
This is \cite[3.2.3]{BVK} and \cite[3.3.1]{BVK}.
\end{proof}
\begin{rem}\label{mustinvert}
The above proposition is false if we do not invert $p$.  For example, over a field $k$ of characteristic $p > 0$ consider the $p$th power morphism $F: \mathbb{G}_m \rightarrow \mathbb{G}_m$, where we view $\mathbb{G}_m$ as a sheaf on $(Sm/k)_{et}$.  Then $\mathrm{coker \ }F$ is not 1-motivic.  To see this, first note that $\mathrm{coker \ }F$ is non-zero: if it were zero, then $F$ would have zero kernel and cokernel as a morphism of sheaves on $(Sm/k)_{et}$, implying that $F$ is an isomorphism by the Yoneda lemma.  So $\mathrm{coker \ }F$ is non-zero; on the other hand, $(\mathrm{coker \ }F)(\overline{k}) = 0$ since $F$ is an epimorphism of fppf sheaves.  This is impossible for a 1-motivic sheaf.
\end{rem}
We now state and prove our main fact for this section:
\begin{prop}\label{onemotivic}
Let $\mathfrak{X}$ be a smooth Artin stack over the perfect field $k$, with quasi-compact separated diagonal.  Then the restriction of the sheaf $\mathbf{Pic}_{\mathfrak{X}/k}$ to $(Sm/k)_{et}$ (which we also denote by $\mathbf{Pic}_{\mathfrak{X}/k}$) is 1-motivic.
\end{prop}
\begin{rem}
In the case where $X$ is a scheme, this is \cite[3.4.1]{BVK}.  For the reader's convenience, we include the proof of this special case in our argument below.
\end{rem}
\begin{proof} (of Proposition \ref{onemotivic})
We prove Proposition \ref{onemotivic} by an increasingly general sequence of lemmas.
\begin{lem}\label{lem1}
Let $X$ be a proper smooth scheme over a field $k$, and let $\pi: X \rightarrow \mathrm{Spec \ }k$ be the structure morphism.  Let $\mathbb{G}_{m,X}$ denote the representable sheaf on $(Sch/X)_{fppf}$ sending $Y$ to $\Gamma(Y,\mathcal{O}_Y)^\times$.  Then the sheaf $R^0\pi_*\mathbb{G}_{m,X}$ on $(Sch/k)_{fppf}$  is representable by a torus.  
\end{lem}
\begin{proof} 
Because the statement is \'{e}tale-local on $k$, we may assume that $\pi$ has a section $s: \mathrm{Spec \ }k \rightarrow X$.  By \cite[7.7.6]{ega3}, the sheaf $R^0\pi_*\mathbb{G}_a$ sending $T$ to $\Gamma(X_T,\mathcal{O}_{X_T})$ is representable by $\mathrm{Spec \ }V$, where $V$ is a finite-dimensional $k$-vector space.  More specifically, we have a functorial bijection
\begin{equation*}
\Gamma(X_T,\mathcal{O}_{X_T}) \longrightarrow \mathrm{Hom}_k(V,\Gamma(T,\mathcal{O}_T)).
\end{equation*}
Using the fact that $\pi$ has a section, applying $T = k$ in the above bijection yields $$V \cong \Gamma(X,\mathcal{O}_X)^\vee \cong k^n,$$  where $n = |\pi_0(X)|$.  Then $R^0\pi_*\mathbb{G}_m$ is representable by $\mathbb{G}_m^n$.
\end{proof}
\begin{lem}\label{lem2}
Let $X$ be a proper smooth scheme over a field $k$, and let $\pi: X \rightarrow \mathrm{Spec \ }k$ be the structure morphism.  Then the sheaves $R^0\pi_*\mathbb{G}_{m,X}$ and $R^1\pi_*\mathbb{G}_{m,X}$ are 1-motivic.
\end{lem}
\begin{proof}
This is immediate from Lemma \ref{lem1} and Theorem \ref{stackrepres}.
\end{proof}
\begin{lem}\label{compactifiable}
Let $X$ be a smooth separated scheme of finite type over $k$, such that there exists an open immersion $j: X \hookrightarrow \overline{X}$ with $\overline{X}$ smooth and proper. Then (letting $\pi: X \rightarrow \mathrm{Spec \ }k$ be the structure morphism) the sheaves $R^0\pi_*\mathbb{G}_{m,X}$ and $R^1\pi_*\mathbb{G}_{m,X}$ are 1-motivic.
\end{lem}
\begin{proof}
Let $i: D \hookrightarrow \overline{X}$ be the complement of $U$ in $\overline{X}$.  Then we have an exact sequence of sheaves
\begin{equation*}
0 \longrightarrow R^0\overline{\pi}_*\mathbb{G}_{m,{\overline{X}}} \longrightarrow R^0\pi_*\mathbb{G}_{m,X} \longrightarrow \mathbf{Div}_{D}(\overline{X}) \longrightarrow \mathbf{Pic}_{\overline{X}/k} \longrightarrow \mathbf{Pic}_{X/k} \rightarrow 0,
\end{equation*}
using the fact that $\overline{X}$ is smooth.  Here $\overline{\pi}: \overline{X} \rightarrow \mathrm{Spec} \ k$ is the structure morphism, and $\mathbf{Div}_D(\overline{X})$ is the locally constant sheaf of Weil divisors on $\overline{X}$ supported on $D$.  By Proposition \ref{picrepres}(c),  $R^0\pi_*\mathbb{G}_{m,X}$ is 1-motivic.
\end{proof}
\begin{lem}
Let $X$ be a smooth separated algebraic space of finite type over $k$.  Then $R^0\pi_*\mathbb{G}_{m,X}$ and $R^1\pi_*\mathbb{G}_{m,X}$ are 1-motivic sheaves.
\end{lem}
\begin{proof}
By \cite{nagata}, we can choose a compactification $X \hookrightarrow \overline{X}$ and then by \cite[3.1]{dejong}, we can find an alteration $\overline{X}_0 \rightarrow \overline{X}$ which is generically \'{e}tale, with $\overline{X}_0$ smooth proper.  That is, we have a commutative diagram
\begin{equation*}
\xymatrix{
U_0 \ar@{^{(}->}[r] \ar[d] & X_0 \ar@{^{(}->}[r] \ar[d] & \overline{X}_0 \ar[d] \\
U \ar@{^{(}->}[r] & X \ar@{^{(}->}[r] & \overline{X},
}
\end{equation*}
where $U_0 \rightarrow U$ is finite \'{e}tale.  We can then extend $\overline{X}_0 \rightarrow \overline{X}$ to a simplicial scheme $\overline{X}_\bullet \rightarrow \overline{X}$, and by possibly further restricting $U$, we can arrange that the restriction  $U_\bullet \rightarrow U$ of this simplicial scheme to $U$ has the property that $U_i \rightarrow U$ is finite \'{e}tale for $i \in \{0,1,2\}$.  Let $\pi_p: U_p \rightarrow U$ be the projection.  Then in the spectral sequence
\begin{equation}\label{spectralsequence}
R^q{\pi_p}_* \pi_p^*\mathbb{G}_{m,U} \Rightarrow R^{p+q}\pi_*\mathbb{G}_{m,U},
\end{equation}
we have $\pi_p^*\mathbb{G}_{m,U} = \mathbb{G}_{m,U_p}$ for $p = 0,1,2$.  For all $q$, let $\mathscr{K}^{q \bullet}$ be the complex with terms
\begin{equation*}
\mathscr{K}^{qp} = R^q\pi_{p*}\mathbb{G}_{m,U_p}.
\end{equation*}
By Lemma \ref{compactifiable}, $\mathscr{K}^{qp}$ is 1-motivic for $q = 0,1$ and all $p$.  Moreover, the spectral sequence \ref{spectralsequence} yields
\begin{equation*}
R^0\pi_*\mathbb{G}_{m,U} \cong \mathscr{H}^0(\mathscr{K}^{0\bullet})
\end{equation*}
and an exact sequence
\begin{equation*}
0 \longrightarrow \mathscr{H}^1(\mathscr{K}^{0\bullet}) \longrightarrow R^1\pi_*\mathbb{G}_{m,U} \longrightarrow \mathscr{H}^0(\mathscr{K}^{1\bullet}) \longrightarrow \mathscr{H}^2(\mathscr{K}^{0\bullet}).
\end{equation*}
By Proposition \ref{picrepres}(c), the homology sheaves $\mathscr{H}^p(\mathscr{K}^{q,\cdot})$ are 1-motivic for $p$ arbitrary and $q = 0,1$.  Another application of Proposition \ref{picrepres}(c) shows that $R^1\pi_*\mathbb{G}_{m,U}$ is 1-motivic.

Now let $i: D \hookrightarrow X$ be the inclusion of the complement $D = X - U$, and let $\pi_X, \pi_U, \pi_D$ be the structure morphisms to $\mathrm{Spec \ }k$.  Then we have an exact sequence of sheaves on $(Sm/k)_{et}$
\begin{equation*}
0 \rightarrow R^0{\pi_X}_*\mathbb{G}_{m,X} \rightarrow R^0{\pi_U}_*\mathbb{G}_{m,U} \rightarrow \mathbf{Div}_{D}(X) \rightarrow R^1{\pi_X}_*\mathbb{G}_{m,X} \rightarrow R^1{\pi_U}_*\mathbb{G}_{m,U} \rightarrow 0,
\end{equation*}
where $\mathbf{Div}_{D}(X)$ is the lattice of Weil divisors on $X$ supported on $D$.  Here the exactness on the right is because $X$ and $U$ are smooth.  Therefore Proposition \ref{picrepres} shows that $R^i{\pi_X}_*\mathbb{G}_{m,X}$ is 1-motivic for $i = 0,1$, as was to be shown.
\end{proof}
We can now complete the proof of Proposition \ref{onemotivic}.  Let $\mathfrak{X}$ be a smooth Artin stack of finite type over $k$, with separated quasi-compact diagonal.  Choose a smooth cover $U \rightarrow \mathfrak{X}$, and take the corresponding Cech simplicial cover $U_\bullet \rightarrow \mathfrak{X}$.  We then have a spectral sequence
\begin{equation*}
R^q{\pi_p}_*\mathbb{G}_{m,U_p} \Rightarrow R^{p+q}\pi_*\mathbb{G}_{m,\mathfrak{X}}.
\end{equation*}
By the previous lemma, $R^q{\pi_p}_*\mathbb{G}_{m,U_p}$ is 1-motivic for $q = 0,1$ and $p$ arbitrary.  Therefore Proposition \ref{picrepres} shows that $R^i\pi_*\mathbb{G}_{m,\mathfrak{X}}$ is 1-motivic for $i= 0 ,1$.  In particular, $\mathbf{Pic}_{\mathfrak{X}/k}$ is 1-motivic.
\end{proof}
\begin{pg}\label{ns}
We can now use Propositions \ref{picrepres} and \ref{onemotivic} to define sheaves $\mathbf{Pic}^0_{\mathfrak{X}/k}$ and $\mathbf{NS}_{\mathfrak{X}/k}$ for any smooth Artin stack $\mathfrak{X}/k$.  Namely, by Proposition \ref{picrepres} there exists a semiabelian variety mapping to $\mathbf{Pic}_{\mathfrak{X}/k}$ with discrete cokernel.  We let $\mathbf{Pic}^0_{\mathfrak{X}/k}$ be the image of this mapping, and $\mathbf{NS}_{\mathfrak{X}/k}$ be its cokernel.  
\begin{rem}
If $\mathfrak{X}$ is proper and smooth over a field of characteristic 0, then these definitions agree with the definitions in Proposition \ref{stackrepres}.  This is not quite true in positive characteristic, because $\mathbf{Pic}^0_{\mathfrak{X}/k}$ does not see the non-reduced part of the Picard scheme of $\mathfrak{X}$: when we view $\mathbf{Pic}^0_{\mathfrak{X}/k}$ as a 1-motivic sheaf, we are restricting to the category $(Sm/k)_{et}$.  For any group scheme $G$, $\mathrm{Hom}(Y,G) = \mathrm{Hom}(Y,G_{red})$ for any smooth scheme $Y$, so the sheaves on $(Sm/k)_{et}$ defined by $G$ and by $G^{red}$ agree.
\end{rem}
\end{pg}
\begin{pg}
There is another subtlety in positive characteristic $p$, namely that we inverted $p$ in all Hom-groups in our study of 1-motivic sheaves.  Therefore the abelian groups $\mathrm{Pic}^0(\mathfrak{X}) := \mathbf{Pic}^0_{\mathfrak{X}/k}(k)$ and $NS(\mathfrak{X}) := \mathbf{NS}_{\mathfrak{X}/k}(k)$ are only well-defined up to a kernel and cokernel annihilated by a power of $p$.  We will only deal with these groups through their $\ell$-adic avatars $T_\ell \mathrm{Pic}^0(\mathfrak{X})$ and ${\ilim}_n \mathrm{NS}(\mathfrak{X})/\ell^nNS(\mathfrak{X})$, however (for $\ell \not= p$), and the following simple proposition shows that these are well-defined:
\end{pg}
\begin{prop}
Let $A$ and $B$ be abelian groups, and assume that we have an exact sequence
\begin{equation*}
0 \rightarrow K \rightarrow A \rightarrow B \rightarrow C \rightarrow 0
\end{equation*}
with $K$ and $C$ annihilated by some power $p^r$.  Then for any $\ell \not= p$ and integer $n$, $A[\ell^n] \cong B[\ell^n]$ and $A/\ell^n A \cong B/\ell^n B$.
\end{prop}
\begin{proof}
Let $I$ be the image of $A$ in $B$, so that we have exact sequences
\begin{align*}
&0 \rightarrow K \rightarrow A \rightarrow I \rightarrow 0 \ \ \mathrm{and} \\
&0 \rightarrow I \rightarrow B \rightarrow C \rightarrow 0.
\end{align*}
Because $K[\ell^n] = K/\ell^n K = 0$ (and similarly for $C$), applying the functor of $\ell^n$-torsion to both sequences yields $A[\ell^n] \cong I[\ell^n] \cong B[\ell^n]$ and $A/\ell^n A \cong I/\ell^n I \cong B/\ell^n B$.
\end{proof}
Therefore the groups $\mathrm{Pic}^0(\mathfrak{X})[\ell^n]$ and $NS(\mathfrak{X})/\ell^nNS(\mathfrak{X})$ are well-defined.

\section{Weil Divisors and Cartier Divisors on Deligne-Mumford Stacks}
\begin{pg}
Let $\mathfrak{X}$ be a $d$-dimensional separated Deligne-Mumford stack of finite type over a field $k$, and assume that the connected components of $\mathfrak{X}$ are equidimensional of dimension $d$.  We recall the definition of the Chow groups $A^*(\mathfrak{X})$ \cite[3.4]{vistoli}.  For each $i$ between 0 and $d$, define a presheaf $\mathscr{Z}^i$ on $\mathfrak{X}_{et}$ by setting $$\mathscr{Z}^i(X \rightarrow \mathfrak{X}) := Z^i(X),$$ where $Z^i(X)$ is the free abelian group on the integral closed subschemes of $X$ of codimension $i$, and the transition maps are induced by flat pullback of cycles.  The presheaf $\mathscr{Z}^i$ is in fact a sheaf \cite[4.2]{gillet}.  Also define a sheaf $\mathscr{W}^i$ on $\mathfrak{X}_{et}$ by setting $$\mathscr{W}^i(X \rightarrow \mathfrak{X}) := \bigoplus_{V} k(V)^\times,$$  where the direct sum is over all subvarieties $V$ of $X$ of codimension $i - 1$ and $k(V)^\times$ is the group of invertible rational functions on $V$.   Since rational equivalence is preserved under flat pullback, we have a morphism of sheaves $\mathscr{W}^i \rightarrow \mathscr{Z}^i$ by taking the associated divisor of a rational function.  We define
\begin{align*}
Z^i(\mathfrak{X}) &:= \Gamma(\mathfrak{X},\mathscr{Z}^i), \\
W^i(\mathfrak{X}) &:= \Gamma(\mathfrak{X},\mathscr{W}^i), \ \ \mathrm{and} \\
A^i(\mathfrak{X}) &:= Z^i(\mathfrak{X})/W^i(\mathfrak{X}).
\end{align*}
We define the group $\mathrm{Div \ }\mathfrak{X}$ of Weil divisors on $\mathfrak{X}$ to be the group $Z^1(\mathfrak{X})$.  We define the Weil class group $\mathrm{Cl \ }\mathfrak{X}$ to be $A^1(\mathfrak{X})$. 
\end{pg}
\begin{pg}
Next we define the notion of Cartier divisor on the stack $\mathfrak{X}$.  For simplicity we will assume that the stack $\mathfrak{X}$ is geometrically reduced over $k$.  Then we can define the sheaf $\mathscr{K}$ of rational maps on $\mathfrak{X}_{et}$ by setting $$\mathscr{K}(X \rightarrow \mathfrak{X}) := \dlim_{U \subseteq X \ \mathrm{open \ dense}} \mathrm{Hom}(U,\mathbb{A}_k^1).$$  We define $\mathscr{K}^*$ to be the subsheaf of invertible elements of $\mathscr{K}$ under multiplication.  We then define $$\mathrm{Ca \ }\mathfrak{X} := \Gamma(\mathfrak{X},\mathscr{K}^*/\mathbb{G}_m),$$ the group of Cartier divisors on $\mathfrak{X}$, where $\mathbb{G}_m$ is the usual sheaf of invertible sections on $\mathfrak{X}_{et}$.  Then define the Cartier class group
\begin{equation*}
\mathrm{CaCl \ }\mathfrak{X} := \Gamma(\mathfrak{X},\mathscr{K}^*/\mathbb{G}_m)/\Gamma(\mathfrak{X},\mathscr{K}^*).
\end{equation*}
\end{pg}
\begin{prop}
Let $\mathfrak{X}$ be a smooth separated Deligne-Mumford stack of finite type over $k$.  Then there is a canonical isomorphism $\mathrm{Ca} \ \mathfrak{X} \cong \mathrm{Div \ }\mathfrak{X}$.  Moreover, this isomorphism induces an isomorphism $\mathrm{CaCl \ }\mathfrak{X} \cong \mathrm{Cl \ }\mathfrak{X}$.
\end{prop}
\begin{proof}
We first note that we have a commuting diagram 
\begin{equation}\label{diag1}
\begin{CD}
\mathscr{K}^* @>>> \mathscr{K}^*/\mathcal{O}^* \\
@V \cong VV @V \cong VV \\
\mathscr{W}^1 @>>> \mathscr{Z}^1
\end{CD}
\end{equation}
of sheaves on $\mathfrak{X}_{et}$, where the vertical maps are isomorphisms.  To see that we have such a diagram, notice that any etale $X/\mathfrak{X}$ is smooth, and so taking sections of the above square at $X$, we require a commuting diagram
\begin{equation*}
\begin{CD}
k(X)^\times @>>> \mathrm{Ca \ }X \\
@V \cong VV  @V \cong VV \\
k(X)^\times @>>> \mathrm{Div \ }X.
\end{CD}
\end{equation*}
The fact that this diagram is commutative (and that the right-hand vertical arrow is an isomorphism) is a standard consequence of $X$ being smooth.  Therefore we have the commuting diagram (\ref{diag1}).  Taking global sections of the right hand map in (\ref{diag1}) gives us an isomorphism $\mathrm{Ca \ }\mathfrak{X} \cong \mathrm{Div \ }\mathfrak{X}$, while taking cokernels of the maps on global sections induced by the horizontal arrows gives us $\mathrm{CaCl \ }\mathfrak{X} \cong \mathrm{Cl \ }\mathfrak{X}$.
\end{proof}
\begin{pg}
From the exact sequence 
\begin{equation*}
0 \rightarrow \mathbb{G}_m \rightarrow \mathscr{K}_{\mathfrak{X}}^* \rightarrow \mathscr{K}_{\mathfrak{X}}^*/\mathbb{G}_m \rightarrow 0
\end{equation*}
of sheaves on $\mathfrak{X}_{et}$, we get an injection $\mathrm{CaCl \ }\mathfrak{X} \hookrightarrow \mathrm{Pic \ }\mathfrak{X}$.  Unlike the case of schemes, however, we cannot expect this map to be an isomorphism, even when $\mathfrak{X}$ is smooth.  For example, if $\mathfrak{X} = BG$ where $G$ is a finite group, then $\mathrm{CaCl \ }{BG} = 0$ while $$\mathrm{Pic \ }BG = \mathrm{Hom}(G,\mathbb{G}_m) \not= 0.$$The best we can do is the following:
\end{pg}
\begin{prop}\label{cartpic}
Let $\mathfrak{X}$ be a geometrically reduced, separated Deligne-Mumford stack of finite type over a field $k$.  Then the quotient
\begin{equation*}
H := \mathrm{Pic \ }\mathfrak{X}/\mathrm{CaCl \ }\mathfrak{X}
\end{equation*}
is a finite group.
\end{prop}
\begin{proof}
 Taking cohomology in the exact sequence
\begin{equation*}
0 \rightarrow \mathbb{G}_{m,\mathfrak{X}} \rightarrow \mathscr{K}_{\mathfrak{X}}^* \rightarrow \mathscr{K}_{\mathfrak{X}}^*/\mathbb{G}_{m,\mathfrak{X}} \rightarrow 0,
\end{equation*}
we have an exact sequence of groups
\begin{equation*}
0 \rightarrow \mathrm{CaCl \ }\mathfrak{X} \rightarrow \mathrm{Pic \ }\mathfrak{X} \rightarrow H^1(\mathfrak{X},\mathscr{K}^*),
\end{equation*}
so it suffices to show that $H^1(\mathfrak{X},\mathscr{K}^*)$ is finite.  Let $X$ be the coarse moduli space of $\mathfrak{X}$, let $\xi_1,...,\xi_n$ be the generic points of $X$, and let $\xi = \xi_1 \coprod ... \coprod \xi_n$ be their disjoint union.  Finally, let $\mathscr{G}_i = \xi_i \times_{X} \mathfrak{X}$ and $\mathscr{G} = \xi \times_{X} \mathfrak{X}$ be the fiber products, and $\iota_i: \mathscr{G}_i \hookrightarrow \mathfrak{X}$, $\iota: \mathscr{G} \hookrightarrow \mathfrak{X}$ the resulting maps.  

\begin{lem}
With notation as above, we have an equality
\begin{equation*}
\iota_*\mathbb{G}_{m} = \mathscr{K}^*
\end{equation*}
of sheaves on $\mathfrak{X}_{et}$.
\end{lem}
\begin{proof}(of lemma)
Let $V \rightarrow \mathfrak{X}$ be \'{e}tale.  Then 
\begin{equation*}
\mathscr{K}^*(V \rightarrow \mathfrak{X}) = \dlim_{U \subset V \\ \mathrm{dense}} \mathrm{Hom}(U,\mathbb{G}_m).
\end{equation*}
On the other hand,
\begin{align*}
(\iota_*\mathbb{G}_m)(V \rightarrow \mathfrak{X}) &= \mathrm{Hom}(V \times_{\mathfrak{X}} \mathscr{G},\mathbb{G}_m) \\
&= \mathrm{Hom}(V \times_X \xi, \mathbb{G}_m) \\
&= \mathscr{K}^*(V \rightarrow \mathfrak{X}).
\end{align*} 
The last equality is because $V$ is quasi-finite over $X$ with open image, and hence the fiber over $\xi$ consists of the generic points of $V$.
\end{proof}
Continuing with the proof of Proposition \ref{cartpic}, from the spectral sequence
\begin{equation*}
H^p(\mathfrak{X},R^q\iota_*\mathbb{G}_m) \Rightarrow H^{p+q}(\mathscr{G},\mathbb{G}_m),
\end{equation*}
we get an inclusion $H^1(\mathfrak{X},\iota_*\mathbb{G}_m) \hookrightarrow \mathrm{Pic}(\mathscr{G})$.  This reduces us to showing that $\mathrm{Pic}(\mathscr{G})$ is finite.  Moreover, since $\mathscr{G} = \amalg_i \mathscr{G}_i$, it suffices to show that $\mathrm{Pic}(\mathscr{G}_i)$ is finite for each $i$.  Set $\mathscr{H} := \mathscr{G}_i$ for any $i$, and $\zeta := \xi_i$.  Then $\mathscr{H} \rightarrow \zeta$ is an fppf-gerbe: Sinec $\zeta \rightarrow X$ is flat, $\mathscr{H} \rightarrow \zeta$ is the coarse moduli space of $\mathscr{H}$, so the topological space of $\mathscr{H}$ has just one point.  By \cite[Thm 11.5]{lmb}, $\mathscr{H}$ must be an fppf-gerbe over $\zeta$.  Moreover, $\mathscr{H} \rightarrow \zeta$ is banded by a finite group $G$ since $\mathscr{H}$ is a Deligne-Mumford stack.  Therefore the following lemma will complete the proof of Proposition \ref{cartpic}. 
\end{proof}
\begin{lem}
Let $\mathscr{H}$ be a Deligne-Mumford stack such that $\mathscr{H} \rightarrow \zeta$ is an fppf-gerbe, with $\zeta$ the spectrum of a field.  Then $\mathrm{Pic}(\mathscr{H})$ is finite.
\end{lem}
\begin{proof}
Let $\zeta = \mathrm{Spec \ }F$, and let $F \hookrightarrow L$ be a finite field extension such that $\mathscr{H} \times_{\zeta} \mathrm{Spec \ }L$ is isomorphic to $BG$, where $G$ is a finite group.  Then $\mathrm{Pic}(\mathscr{H} \times_{\zeta} \mathrm{Spec \ }L) = \mathrm{Pic}(BG) = \mathrm{Hom}(G,\mathbb{G}_m)$ is a finite group.  Now consider the Picard functor $\mathbf{Pic}_{\mathscr{H}/\zeta}$.  The spectral sequence
\begin{equation*}
H_{fppf}^p(\zeta,R^q\pi_*\mathbb{G}_{m,\mathscr{H}}) \Rightarrow H^{p+q}(\mathscr{H},\mathbb{G}_m)
\end{equation*}
yields an exact sequence
\begin{equation*}
0 \rightarrow H_{fppf}^1(\zeta,\pi_*\mathbb{G}_{m,\mathscr{H}}) \rightarrow \mathrm{Pic}(\mathscr{H}) \rightarrow \mathbf{Pic}_{\mathscr{H}/\zeta}(\zeta).
\end{equation*}
Note that we have $\pi_*\mathbb{G}_{m,\mathscr{H}} = \mathbb{G}_{m,\zeta}$ as sheaves in $\zeta_{fppf}$.  This is because for any $Y \rightarrow \zeta$, the stack $Y \times_{\zeta} \mathscr{H}$ has $Y$ as coarse moduli space (since $Y \rightarrow \zeta$ is flat) and therefore $\mathrm{Hom}(X \times_{\zeta} \mathscr{H},\mathbb{G}_m) = \mathrm{Hom}(X,\mathbb{G}_m)$.  It is well-known that $H_{fppf}^1(\zeta,\mathbb{G}_m) = \mathrm{Pic}(\zeta) = 0$, and hence $\mathrm{Pic}(\mathscr{H})$ injects into $\mathbf{Pic}_{\mathscr{H}/\zeta}(\zeta)$.  Finally, by faithfully flat descent, $\mathbf{Pic}_{\mathscr{H}/\zeta}(\zeta)$ injects into $\mathbf{Pic}_{\mathscr{H}/\zeta}(\mathrm{Spec \ }L)$ which is finite since it equals $\mathrm{Pic}(\mathscr{H} \times_{\zeta} \mathrm{Spec \ }L) = \mathrm{Pic}(BG)$.  
\end{proof}

\begin{cor}\label{pic0}
Let $\mathfrak{X}$ be a smooth proper Deligne-Mumford stack over an algebraically closed field $k$.  Then every element of $\mathrm{Pic}^0(\mathfrak{X})$ is represented by a Weil divisor.
\end{cor}
\begin{proof}
Temporarily, let $\mathrm{Cl}^0(\mathfrak{X}) \subset \mathrm{Pic}^0(\mathfrak{X})$ denote the subgroup of the Weil class group which maps to 0 in $NS(\mathfrak{X})$.  Then we have an injection
\begin{equation*}
\mathrm{Pic}^0(\mathfrak{X})/\mathrm{Cl}^0(\mathfrak{X}) \hookrightarrow \mathrm{Pic}(\mathfrak{X})/\mathrm{Cl}(\mathfrak{X}),
\end{equation*}
so $\mathrm{Pic}^0(\mathfrak{X})/\mathrm{Cl}^0(\mathfrak{X})$ is finite.  On the other hand, it is a quotient of the divisible group $\mathrm{Pic}^0(\mathfrak{X})$, hence divisible.  Therefore it is trivial.
\end{proof}

\section{Cycle class map on smooth Deligne-Mumford stacks}
\begin{pg}
Let $\mathfrak{X}$ be a smooth Deligne-Mumford stack of pure dimension $N$ over a field $k$.  In this section we review the definition of the cycle class map 
\begin{equation*}
cl: A^d(\mathfrak{X}) \rightarrow H^{2d}(\mathfrak{X}_{\overline{k}},\mathbb{Q}_\ell(d)),
\end{equation*}
where $\ell$ is different from $p = \mathrm{char}(k)$.  For the definition of a cycle class map for singular Deligne-Mumford stacks and more general coefficient rings, see \cite[Sect. 3]{chern}.
\end{pg}
\begin{pg}
For ease of notation we assume $k = \overline{k}$; it will be clear from our construction that the cycle class map we produce will be invariant under Galois action.  Let $\mathfrak{X}$ be a purely $N$-dimensional Deligne-Mumford stack over $k$, let $D \in A^d(\mathfrak{X})$ be a cycle, and let $e = N - d$ be the dimension of $D$.  Write $D = \sum a_i D_i$ as a sum of integral cycles $D_i$, and let $U_i \subset D_i$ be the smooth locus.  
\begin{lem}\label{tracemap}
For each $D_i$, there is a canonical trace map $$Tr_i: H_c^{2e}(D_i,\mathbb{Q}_\ell(e)) \stackrel{\sim}{\longrightarrow} \mathbb{Q}_\ell$$ 
\end{lem}
Before starting the proof of Lemma \ref{tracemap} we note the following fact, which will be used repeatedly in this paper:
\begin{lem}\label{cms}
Let $\mathfrak{X}$ be a separated finite type Deligne-Mumford stack, and let $\pi: \mathfrak{X} \rightarrow X$ be its coarse moduli space.  Then the natural map $\mathbb{Q}_{\ell,X} \rightarrow R\pi_*\mathbb{Q}_{\ell,\mathfrak{X}}$ is an isomorphism in $D_c^b(X,\mathbb{Q}_\ell)$.
\end{lem}
\begin{proof}
Combining Theorem 5.1 and Corollary 5.8 of \cite{fujiwara}, we have that $R\pi_*\mathbb{Q}_\ell$ is acyclic in non-zero degrees, so we only need to show that $R^0\pi_*\mathbb{Q}_{\ell,\mathfrak{X}} = \mathbb{Q}_{\ell,X}$ which follows easily from the fact that the topological spaces of $\mathfrak{X}$ and $X$ are homeomorphic.
\end{proof}
We now return to Lemma \ref{tracemap}:
\begin{proof} (of Lemma \ref{tracemap})  Let $U_i$ be the non-empty smooth locus of $D_i$, and $j: U_i \hookrightarrow D_i$ the inclusion, and $k: Z_i \hookrightarrow D_i$ the inclusion of $Z_i := D_i - U_i$ into $D_i$.  Then from the short exact sequence 
\begin{equation*}
0 \rightarrow j_!\mathbb{Q}_{\ell,U_i} \rightarrow \mathbb{Q}_{\ell,D_i} \rightarrow i_*\mathbb{Q}_{\ell,Z_i} \rightarrow 0,
\end{equation*}
we get a long exact sequence
\begin{equation*}
... \rightarrow H_c^{2e-1}(Z_i,\mathbb{Q}_\ell(e)) \rightarrow H_c^{2e}(U_i,\mathbb{Q}_\ell(e)) \rightarrow H_c^{2e}(D_i,\mathbb{Q}_\ell(e)) \rightarrow H_c^{2e}(Z_i,\mathbb{Q}_\ell(e)) \rightarrow ...
\end{equation*}
But $H_c^{2e-1}(Z_i,\mathbb{Q}_\ell(e)) = H_c^{2e}(Z_i,\mathbb{Q}_\ell(e)) = 0$ because $\mathrm{dim}(Z_i) < e$ (use Lemma \ref{cms} and the fact that the statement is true for algebraic spaces.)  Now Poincar\'{e} duality on the smooth stack $U_i$ \cite[4.4.1]{LO} gives the required map $Tr_i$.
\end{proof}
An easy extension of the above lemma shows that the trace maps $Tr_i$ induce a canonical isomorphism
\begin{equation}\label{trace}
\mathrm{Hom}(H_c^{2e}(D,\mathbb{Q}_\ell(e)),\mathbb{Q}_\ell) \stackrel{\sim}{\longrightarrow} \mathrm{Hom}(\oplus_i H_c^{2e}(U_i,\mathbb{Q}_\ell(e)),\mathbb{Q}_\ell) \stackrel{\sim}{\longrightarrow} \mathbb{Q}_\ell^{I(D)},
\end{equation}
where $I(D) = \{D_1,...,D_r\}$ is the set of irreducible components of $D$.   Let $\alpha: D \hookrightarrow \mathfrak{X}$ be the inclusion.  Let $\mathscr{D}_{\mathfrak{X}}$ and $\mathscr{D}_k$ be the Verdier dualities on $\mathfrak{X}$ and $k$, respectively.  Then we have
\begin{equation*}
\mathscr{D}_k R\Gamma_c \alpha^* \mathbb{Q}_\ell(e) = R\Gamma R\alpha^! \mathscr{D}_{\mathfrak{X}}(\mathbb{Q}_\ell(e)) = R\Gamma R\alpha^! \mathbb{Q}_\ell[2N](N-e),
\end{equation*}
where we used the fact that $\mathfrak{X}$ is smooth in the right-hand equality.  This induces a canonical isomorphism
\begin{equation}\label{verdier}
\mathrm{Hom}(H_c^{2e}(D,\mathbb{Q}_\ell(e)),\mathbb{Q}_\ell) \stackrel{\sim}{\longrightarrow} H_D^{2d}(\mathfrak{X},\mathbb{Q}_\ell(d))
\end{equation}
(recall that $d = N - e$).
\end{pg}
\begin{defn}\label{cycleclassdefn}
In the above notation, we obtain from \ref{trace} a canonical element $[D] \in \mathrm{Hom}(H_c^{2e}(D,\mathbb{Q}_\ell(e)),\mathbb{Q}_\ell)$ corresponding to $\sum a_i D_i \in \mathbb{Q}_\ell^{I(D)}$.  Then we set the localized cycle class of $D$, denoted $cl(D)$, to be the image of $[D]$ in $H_D^{2d}(\mathfrak{X},\mathbb{Q}_\ell(d))$ under \ref{verdier}.  
\end{defn}
\begin{pg}
We also write $cl(D)$ for the image of the localized cycle class under the map $H_D^{2d}(\mathfrak{X},\mathbb{Q}_\ell(d)) \rightarrow H^{2d}(\mathfrak{X},\mathbb{Q}_\ell(d))$ (we call this the global cycle class of $D$).  When we need to distinguish between the local cycle class and global cycle class of $D$, we will denote these by $cl_{loc}(D)$ and $cl_{gl}(D)$, respectively.  It is standard in the case of schemes that the resulting map
\begin{equation*}
cl: Z^d(\mathfrak{X}) \rightarrow H^{2d}(\mathfrak{X},\mathbb{Q}_\ell(d))
\end{equation*}
passes to the Chow group $A^d(\mathfrak{X})$, and is compatible with the contravariant functoriality for $A^d(-)$ and $H^{2d}(-,\mathbb{Q}_\ell(d))$.  The same proof works for stacks.
\end{pg}
\begin{pg}
Now let $D \subset \mathfrak{X}$ be any reduced closed subscheme, and $\alpha: D \hookrightarrow \mathfrak{X}$ the inclusion.  Let $\mathrm{dim}(D) = e$ and $\mathrm{dim}(\mathfrak{X}) = N$.  We can use the above cycle class map to give a cycle-theoretic description of the Poincar\'{e} dual of the restriction map 
\begin{equation*}
\alpha^*: H_c^{2e}(\mathfrak{X},\mathbb{Q}_\ell(e)) \rightarrow H_c^{2e}(D,\mathbb{Q}_\ell(e)).
\end{equation*}
Let $D = \cup_i D_i$ be the decomposition of $D$ into its irreducible components, and let $I_e(D)$ be the set of $e$-dimensional irreducible components of $D$.  Let $Tr_i: H_c^{2e}(D_i,\mathbb{Q}_\ell(e)) \rightarrow \mathbb{Q}_\ell$ be the trace isomorphism for any $e$-dimensional irreducible component of $D$.  
\end{pg}
\begin{prop}
With notation as above, we have a commutative square
\begin{equation*}
\begin{CD}
H_c^{2e}(D,\mathbb{Q}_\ell(e))^\vee @> (\alpha^*)^\vee >> H_c^{2e}(\mathfrak{X},\mathbb{Q}_\ell(e))^\vee \\
@V \sim VV @V \sim VV \\
\mathbb{Q}_\ell^{I_e(d)} @>>> H^{2N - 2e}(\mathfrak{X},\mathbb{Q}_\ell(N-e))
\end{CD}
\end{equation*}
where the vertical arrows are the isomorphisms induced by Poincar\'{e} duality, and the lower arrow sends $\sum a_i[D_i]$ to $\sum a_icl_n(D_i)$.
\end{prop}
\begin{proof}
This is immediate from the above description of the cycle class map.
\end{proof}
\begin{pg}\label{divisors1}
In the case of divisors, the cycle class map can be described as follows.  Let $\mathfrak{X}$ be a smooth Deligne-Mumford stack over $k$, and $D$ a closed subscheme of $\mathfrak{X}$; let $i: D \hookrightarrow \mathfrak{X}$ be the inclusion, and $j: U \hookrightarrow \mathfrak{X}$ the inclusion of the open complement $U = \mathfrak{X} - D$.  Recall that there is a natural bijection between $H_D^1(\mathfrak{X},\mathbb{G}_m)$ and the group of Cartier divisors supported on $D$.  To see this, first define
\begin{equation*}
\mathscr{H}_D^j(\mathfrak{X},\mathbb{G}_m) := i_*R^ji^!\mathbb{G}_m.
\end{equation*}
Then we have an exact sequence on sheaves on $\mathfrak{X}_{et}$
\begin{equation*}
0 \rightarrow \mathscr{H}_D^0(\mathfrak{X},\mathbb{G}_m) \rightarrow \mathbb{G}_m \rightarrow j_*\mathbb{G}_{m,U} \rightarrow \mathscr{H}_D^1(\mathfrak{X},\mathbb{G}_m) \rightarrow 0,
\end{equation*}
so that
\begin{equation*}
\mathscr{H}_D^1(\mathfrak{X},\mathbb{G}_m) \cong \mathrm{coker}(\mathbb{G}_m \rightarrow j_*\mathbb{G}_{m,U}).
\end{equation*}
By its definition, giving a global section of the latter sheaf is the same as giving a Cartier divisor supported on $D$.  Moreover, it is clear that $\mathscr{H}_D^0(\mathfrak{X},\mathbb{G}_m) = 0$.  Therefore the spectral sequence
\begin{equation*}
H^p(\mathfrak{X},\mathscr{H}_D^q(\mathfrak{X},\mathbb{G}_m)) \Rightarrow H_D^{p+q}(\mathfrak{X},\mathbb{G}_m)
\end{equation*}
shows that $$H_D^1(\mathfrak{X},\mathbb{G}_m) \cong H^0(\mathfrak{X},\mathscr{H}_D^1(\mathfrak{X},\mathbb{G}_m)),$$ and hence that $H_D^1(\mathfrak{X},\mathbb{G}_m)$ is isomorphic to the group of Cartier divisors supported on $D$.

Now suppose $D$ is a Cartier divisor on $\mathfrak{X}$.  Then we have a canonical class $cl'(D) \in H_D^1(\mathfrak{X},\mathbb{G}_m)$ corresponding to $D$.  Then for any $n$ prime to $p = \mathrm{char \ }k$, we produce a class $cl''(D) \in H_D^2(\mathfrak{X},\mathbb{Q}_\ell(1))$ using the Kummer exact sequence of sheaves
\begin{equation*}
0 \rightarrow \mu_n \rightarrow \mathbb{G}_m \rightarrow \mathbb{G}_m \rightarrow 0
\end{equation*}
to induce a map $H_D^1(\mathfrak{X},\mathbb{G}_m) \rightarrow H_D^2(\mathfrak{X},\mu_n)$ and taking the limit over $n = \ell^m$.
\end{pg}
\begin{prop}\label{divisors}
The class $cl''(D) \in H_D^2(\mathfrak{X},\mathbb{Q}_\ell(1))$ agrees with the class $cl(D)$ defined earlier.
\end{prop}
\begin{proof}
In the case of schemes this is \cite[Cycle, 2.3.6]{SGA4h}, and one easily reduces to this case since the definition of $cl(D) \in H_D^2(\mathfrak{X},\mathbb{Q}_\ell(1))$ is compatible with \'{e}tale localization.
\end{proof}
Using this fact, and finite generation of the Neron-Severi group (\ref{ns}) we get the following:
\begin{prop}\label{factors}
Let $\mathfrak{X}$ be a smooth scheme over $k$ as above, and choose a prime $\ell \not= \mathrm{char \ }k$.   Then the cycle class map
\begin{equation*}
cl: \mathrm{Pic}(\mathfrak{X}) \rightarrow H^2(\mathfrak{X},\mathbb{Q}_\ell(1))
\end{equation*}
factors as 
\begin{equation*}
\mathrm{Pic}(\mathfrak{X}) \rightarrow NS(\mathfrak{X}) \otimes_{\mathbb{Z}} \mathbb{Q}_\ell \hookrightarrow H^2(\mathfrak{X},\mathbb{Q}_\ell(1)),
\end{equation*}
where the  map $\mathrm{Pic}(\mathfrak{X}) \rightarrow NS(\mathfrak{X}) \otimes_{\mathbb{Z}} \mathbb{Q}_\ell$ is the obvious one, and the map on the right is an injection.
\end{prop}
\begin{proof}
From the previous proposition we get that $cl: A^1(\mathfrak{X}) \rightarrow H^2(\mathfrak{X},\mathbb{Q}_\ell(1))$ factors through the injection (induced by the Kummer exact sequence)
\begin{equation*}
\bigg( \ilim_n \frac{\mathrm{Pic}(\mathfrak{X})}{\ell^n \mathrm{Pic}(\mathfrak{X})}\bigg) \otimes \mathbb{Q} \hookrightarrow H^2(\mathfrak{X},\mathbb{Q}_\ell(1)).
\end{equation*} 
By Lemma \ref{ns} we have an exact sequence
\begin{equation*}
0 \rightarrow \mathrm{Pic}^0(\mathfrak{X}) \rightarrow \mathrm{Pic}(\mathfrak{X}) \rightarrow NS(\mathfrak{X}) \rightarrow 0
\end{equation*}
with $\mathrm{Pic}^0(\mathfrak{X})$ divisible (at least up to $p$-torsion), and therefore
\begin{equation*}
\frac{\mathrm{Pic}(\mathfrak{X})}{\ell^n\mathrm{Pic}(\mathfrak{X})} = \frac{NS(\mathfrak{X})}{\ell^n NS(\mathfrak{X})}.
\end{equation*}
Therefore the map $cl$ factors as 
\begin{equation*}
\mathrm{Pic}(\mathfrak{X}) \rightarrow \ilim_n \frac{NS(\mathfrak{X})}{\ell^n NS(\mathfrak{X})} \otimes \mathbb{Q} \hookrightarrow H^2(\mathfrak{X},\mathbb{Q}_\ell(1)).
\end{equation*}
Therefore we are reduced to showing that
\begin{equation*}
\ilim_n \frac{NS(\mathfrak{X})}{\ell^n NS(\mathfrak{X})} \cong NS(\mathfrak{X}) \otimes_{\mathbb{Z}} \mathbb{Z}_\ell,
\end{equation*}
which is clear since $NS(\mathfrak{X})$ is finitely generated.
\end{proof}

\subsection{A variant of the cycle class map for compactly supported cohomology}  Let $\overline{\mathfrak{X}}$ be a smooth proper Deligne-Mumford stack over a field $k = \overline{k}$, and let $i: D \hookrightarrow \overline{\mathfrak{X}}$ be a reduced closed subscheme of $\overline{\mathfrak{X}}$.  Let $\mathfrak{X} = \overline{\mathfrak{X}} - D$, and let $j: \mathfrak{X} \hookrightarrow \overline{\mathfrak{X}}$ be the inclusion.  Finally, let $\mathrm{Div}_{\mathfrak{X}}(\overline{\mathfrak{X}})$ be the free abelian group of divisors on $\overline{\mathfrak{X}}$ whose support is contained in $\mathfrak{X}$, i.e., disjoint from $D$.  In the rest of this section we describe a cycle class map to compactly supported cohomology
\begin{equation}\label{cpctcycle}
cl_c :  \mathrm{Div}_{\mathfrak{X}}(\overline{\mathfrak{X}}) \longrightarrow H_c^2(\mathfrak{X},\mathbb{Q}_{\ell}(1)) := H^2(\overline{\mathfrak{X}},j_!\mathbb{Q}_{\ell}(1))
\end{equation}
and prove some compatibilities regarding this cycle class map.

\subsection{First definition of \ref{cpctcycle}} First recall the following well-known proposition:
\begin{prop}
Let $i: D \hookrightarrow \overline{\mathfrak{X}}$ be the inclusion, and define $$\mathbb{G}_{m,\overline{\mathfrak{X}},D} := \mathrm{Ker}(\mathbb{G}_{m,\overline{\mathfrak{X}}} \rightarrow i_*\mathbb{G}_{m,D}).$$ Then $H^1(\overline{\mathfrak{X}},\mathbb{G}_{m,\overline{\mathfrak{X}},D})$ is in bijection with isomorphism classes of pairs $(\mathscr{L},\sigma)$, where $\mathscr{L}$ is a line bundle on $\overline{\mathfrak{X}}$ and $\sigma: \mathscr{L}\vert_D \stackrel{\sim}{\longrightarrow} \mathcal{O}_{D}$ is a 
trivialization of $\mathscr{L}$ on $D$.  
\end{prop}
\begin{proof}
The statement when $\overline{\mathfrak{X}}$ is a scheme is well-known (see, e.g., \cite[App. A]{albpic}) and the same proof applies to Deligne-Mumford stacks.
\end{proof}
We denote this group (following standard notation) by $\mathrm{Pic}(\overline{\mathfrak{X}},D)$.  

Given $E \in \mathrm{Div}_{\mathfrak{X}}(\overline{\mathfrak{X}})$, there is a natural class $$cl_{c}'(E) \in \mathrm{Pic}(\overline{\mathfrak{X}},D)$$ consisting of the pair $(\mathcal{O}(E),s)$ where $\mathcal{O}(E)$ is the line bundle of the divisor $E$ and $s: \mathcal{O}_X \rightarrow \mathcal{O}(E)$ is the canonical meromorphic section of $E$, which is an isomorphism restricted to $D$ since $D \cap E = \emptyset$.  To define a class in $H_c^2(\mathfrak{X},\mu_n)$ note that we have an exact sequence of sheaves
\begin{equation*}
0 \longrightarrow j_!\mu_{n,\mathfrak{X}} \longrightarrow \mathbb{G}_{m,\overline{\mathfrak{X}},D} \stackrel{\cdot n}{\longrightarrow} \mathbb{G}_{m,\overline{\mathfrak{X}},D} \longrightarrow 0.
\end{equation*}
Then we temporarily write $$cl_{c,1}(E) \in H_c^2(\mathfrak{X},\mathbb{Q}_{\ell}(1))$$ for the image of $cl_{c}'(E) \in H^1(\overline{\mathfrak{X}},\mathbb{G}_{m,\overline{\mathfrak{X}},D})$ under the boundary map, taking the limit over $n = \ell^m$.

\subsection{Second definition of \ref{cpctcycle}}  Given $E \in \mathrm{Div}_{\mathfrak{X}}(\overline{\mathfrak{X}})$, let $v: E \hookrightarrow \overline{\mathfrak{X}}$ be the inclusion.  Consider the canonical map in $D_c^b(\overline{\mathfrak{X}})$ 
\begin{equation*}
f: v_*Rv^!\mathbb{Q}_{\ell,\overline{\mathfrak{X}}}(1) \longrightarrow \mathbb{Q}_{\ell,\overline{\mathfrak{X}}}(1).
\end{equation*}
By definition, this map sends the local cycle class $cl_{loc}(E) \in H_E^2(\overline{\mathfrak{X}},\mathbb{Q}_\ell(1))$ to the global cycle class $cl_{gl}(E) \in H^2(\overline{\mathfrak{X}},\mathbb{Q}_\ell(1))$.  Notice, moreover, that the composition
\begin{equation*}
v_*Rv^!\mathbb{Q}_{\ell,\overline{\mathfrak{X}}}(1) \stackrel{f}{\longrightarrow} \mathbb{Q}_{\ell,\overline{\mathfrak{X}}}(1) \longrightarrow i_*\mathbb{Q}_{\ell,D}(1)
\end{equation*}
is zero, since $D$ and $E$ are disjoint.  Therefore $f$ factors through a unique map
\begin{equation*}
\tilde{f}: v_*Rv^!\mathbb{Q}_{\ell,\overline{\mathfrak{X}}}(1) \longrightarrow j_!\mathbb{Q}_{\ell,\mathfrak{X}} = \mathrm{Ker}(\mathbb{Q}_{\ell,\overline{\mathfrak{X}}}(1) \rightarrow i_*\mathbb{Q}_{\ell,D}(1))
\end{equation*}
($\tilde{f}$ is unique because any two maps $f_1,f_2$ define a map $$f_1 - f_2: v_*Rv^!\mathbb{Q}_{\ell,\overline{\mathfrak{X}}}(1) \rightarrow i_*\mathbb{Q}_{\ell,D}[-1]$$ which must be the zero map, since $v_*Rv^!\mathbb{Q}_{\ell,\overline{\mathfrak{X}}}(1)$ is supported on $E$ and $i_*\mathbb{Q}_{\ell,D}(1)$ is supported on $D$).  

Taking global sections of the map $\tilde{f}$ induces a map $$H_E^2(\overline{\mathfrak{X}},\mathbb{Q}_\ell(1)) \longrightarrow H_c^2(\mathfrak{X},\mathbb{Q}_\ell(1))$$ and we define $cl_{c,2}(E) \in H_c^2(\mathfrak{X},\mathbb{Q}_\ell(1))$ to be the image of the local cycle class $cl_{loc}(E)$ under this map.

\begin{prop}\label{cpctagrees}
The two cycle classes defined above agree, i.e., $cl_{c,1}(E) = cl_{c,2}(E)$.
\end{prop}
\begin{proof} 
We can follow the same steps as in the definition of $cl_{c,2}(E)$ to show that the canonical map
\begin{equation*}
g: v_*Rv^!\mathbb{G}_{m,\overline{\mathfrak{X}}} \longrightarrow \mathbb{G}_{m,\overline{\mathfrak{X}}}
\end{equation*}
factors uniquely through a map
\begin{equation*}
\tilde{g}:  v_*Rv^!\mathbb{G}_{m,\overline{\mathfrak{X}}} \longrightarrow \mathbb{G}_{m,\overline{\mathfrak{X}},D} = \mathrm{Ker}(\mathbb{G}_{m,\overline{\mathfrak{X}}} \rightarrow \mathbb{G}_{m,D}).
\end{equation*}
Taking global sections of $\tilde{g}$ yields a map $$\tilde{g}: H_E^1(\overline{\mathfrak{X}},\mathbb{G}_m) \rightarrow H^1(\overline{\mathfrak{X}},\mathbb{G}_{m,\overline{\mathfrak{X}},D}) = \mathrm{Pic}(\overline{\mathfrak{X}},D).$$  In terms of Cech cohomology, this map is described as follows: let $\mathcal{O}(E)$ be the line bundle of $E$.  Viewing $E$ as a Cartier divisor, we get a transition function $\alpha \in \mathbb{G}_m(U \rightarrow \overline{\mathfrak{X}})$ (where $U$ is some \'{e}tale cover of $\overline{\mathfrak{X}}$) defining $\mathcal{O}(E)$ such that $\alpha$ does not vanish along $D$.  Therefore $\alpha$ also defines a transition function for an element of $H^1(\overline{\mathfrak{X}},\mathbb{G}_{m,\overline{\mathfrak{X}},D})$, and this transition function is precisely the image of $cl(E)$ under $\tilde{g}$.  

From this description it is clear that $\tilde{g}(cl_{loc}(E)) =  cl_{c}'(E) \in \mathrm{Pic}(\overline{\mathfrak{X}},D)$, where $cl_c'(E)$ is the class defined earlier.  Using the Kummer exact sequence, we conclude that the two classes
\begin{equation*}
cl_{c,1}(E), cl_{c,2}(E) \in H_c^2(\mathfrak{X},\mathbb{Q}_\ell(1))
\end{equation*}
are the same.
\end{proof}
We write $cl_c(E)$ for the element $cl_{c,1}(E) = cl_{c,2}(E)$.
\begin{cor}\label{cpctfactors}
Let $\overline{\mathfrak{X}}$, $\mathfrak{X}$, and $D$ be as above.  Then the cycle class map
\begin{equation*}
cl_c: \mathrm{Div}_{\mathfrak{X}}(\overline{\mathfrak{X}}) \longrightarrow H_c^2(\mathfrak{X},\mathbb{Q}_\ell(1))
\end{equation*}
factors as
\begin{equation*}
\mathrm{Div}_{\mathfrak{X}}(\overline{\mathfrak{X}}) \rightarrow NS(\overline{\mathfrak{X}},D) \otimes_{\mathbb{Z}} \mathbb{Q}_\ell \hookrightarrow H_c^2(\mathfrak{X},\mathbb{Q}_\ell(1))
\end{equation*}
where the map $\mathrm{Div}_{\mathfrak{X}}(\overline{\mathfrak{X}}) \rightarrow NS(\overline{\mathfrak{X}},D)$ is the natural map and the map on the right is an injection.
\end{cor}
\begin{proof}
The proof of Proposition \ref{factors} works here as well, using the fact that $\mathrm{Pic}^0(\overline{\mathfrak{X}},D)$ is divisible and $NS(\overline{\mathfrak{X}},D)$ is finitely generated.
\end{proof}

\section{Review of 1-motives}

In this section we review the theory of 1-motives as introduced in \cite[\S 10]{hodge3}.  We work over a perfect base field $k$ of characteristic $p \geq 0$, and choose an algebraic closure $k \hookrightarrow \overline{k}$.  All of this material (and much more) appears in \cite[App. C]{BVK}.

\begin{defn}
A 1-motive over $k$ is a 2-term complex
\begin{equation*}
[L \rightarrow G]
\end{equation*}
of abelian sheaves on $(Sch/k)_{fppf}$, where
\begin{itemize}
\item $L$ is an \'{e}tale-locally constant sheaf, and $L(\overline{k})$ is a finitely generated free abelian group.
\item $G$ is (represented by) a semi-abelian variety; that is, there is an extension
\begin{equation*}
0 \rightarrow T \rightarrow G \rightarrow A \rightarrow 0,
\end{equation*}
where $T$ is a torus and $A$ is an abelian variety.
\end{itemize}
Our convention is that $L$ is placed in degree -1 and $G$ is placed in degree 0.  Since $k$ is perfect, $L$ is fully described by the lattice $L(\overline{k})$ together with its action by $Gal(\overline{k}/k)$.
\end{defn}
\begin{pg}
If $M = [L \rightarrow G]$ and $M' = [L' \rightarrow G']$ are 1-motives, then a morphism of 1-motives $F = (f,g): M \rightarrow M'$ is a commuting diagram
\begin{equation*}
\begin{CD}
L @> f >> L' \\
@VVV @VVV \\
G @> g >> G'.
\end{CD}
\end{equation*}
A useful fact about homomorphisms of 1-motives is the following.
\end{pg}
\begin{prop}
Let $\mathscr{M}^1(k)$ be the category of 1-motives, and $D^b(Sch/k)_{fppf}$ the derived category of fppf sheaves over $k$.  Then the natural functor $$\mathscr{M}^1(k) \rightarrow D^b(Sch/k)_{fppf}$$ is fully faithful.
\end{prop}
\begin{proof}
This is \cite[Prop 2.3.1]{raynaud}.
\end{proof}
\begin{pg}
Let $R$ be a subring of $\mathbb{Q}$ containing $\mathbb{Z}$.  We will always take either $R = \mathbb{Z}[p^{-1}]$ or $R = \mathbb{Q}$.  The category of $R$-\emph{isogeny} 1-motives over $k$, denoted $\mathscr{M}^1_R(k)$ has the same objects as $\mathscr{M}^1(k)$, and for 1-motives $M$, $M'$, we set
\begin{equation*}
\mathrm{Hom}_{\mathscr{M}^1_R(k)}(M,M') := \mathrm{Hom}_{\mathscr{M}^1(k)(M,M')} \otimes R.
\end{equation*}
We also write $\mathscr{M}^1(k)[p^{-1}]$ when $R = \mathbb{Z}[p^{-1}]$, and $\mathscr{M}^1(k) \otimes \mathbb{Q}$ when $R = \mathbb{Q}$. 
\end{pg}
\begin{prop}
The category of $\mathbb{Q}$-isogeny 1-motives over $k$ is abelian.
\end{prop}
\begin{proof}
This is \cite[C.7.3]{BVK}, but we give a more elementary argument here.  First notice that $\mathscr{M}^1_{\mathbb{Q}}(k)$ is clearly additive.   Let $M = [L \stackrel{\alpha}{\rightarrow} G]$ and $M' = [L' \stackrel{\alpha'}{\rightarrow} G']$ be 1-motives, and let $F: M \rightarrow M'$ be a morphism of 1-motives given by the diagram
\begin{equation*}
\begin{CD}
L @> f >> L' \\
@V \alpha VV @V \alpha' VV \\
G @> g >> G'.
\end{CD}
\end{equation*}
We first describe the kernel $F$. Let $\mathrm{Ker}^0(g)$ be the connected component of the identity of $\mathrm{Ker}(g)$, and let $$\mathrm{Ker}^0(f):= \mathrm{Ker}(f) \cap \alpha^{-1}(\mathrm{Ker}^0(g)).$$  Then we set $$K := [\mathrm{Ker}^0(f) \rightarrow \mathrm{Ker}^0(g)],$$ and claim that $K$ is the kernel of $F: M \rightarrow M'$.  Let $M'' = [L'' \rightarrow G'']$ be another 1-motive, and let
\begin{equation*}
\begin{CD}
L'' @> u >> L \\
@VVV @VVV \\
G'' @> v >> G
\end{CD}
\end{equation*}
be a morphism such that the composition with $F$ is 0 in $1-Mot_{\mathbb{Q}}(k)$.  Then for some $n \in \mathbb{N}$, $nF = 0$ in $\mathscr{M}^1(k)$, and hence $f \circ nu = g \circ nv = 0$.  Therefore $nu$ and $nv$ factor through $\mathrm{Ker}(f)$ and $\mathrm{Ker}(g)$, respectively.  Since $\mathrm{Ker}(f)/\mathrm{Ker}^0(f)$ and $\mathrm{Ker}(g)/\mathrm{Ker}^0(g)$ are finite groups, we get that for some $m \in \mathbb{N}$, $mnu$ and $mnv$ factor through $\mathrm{Ker}^0(f)$ and $\mathrm{Ker}^0(g)$ respectively.  Then $$\frac{1}{mn}(mnu,mnv): [L'' \rightarrow G''] \rightarrow [\mathrm{Ker}^0(f) \rightarrow \mathrm{Ker}^0(g)]$$ is a morphism in $\mathscr{M}^1_\mathbb{Q}(k)$ factoring $(u,v): [L'' \rightarrow G''] \rightarrow [L \rightarrow G]$.

Now we describe the cokernel of $F$.  Let $T$ be the torsion subgroup of $\mathrm{coker}(f)$.  We set
\begin{equation*}
C := [\mathrm{coker}(f)/T \rightarrow \mathrm{coker}(g)/\alpha'(T)].
\end{equation*}
A check similar to that for $K$ shows that $C$ is the cokernel of $F$, and that the axioms of an abelian category are satisfied.
\end{proof}
\begin{rem}
The category of 1-motives over $k$ is not abelian, nor is the category of $p$-isogeny 1-motives where $p = \mathrm{char \ }k$.  However, there is an abelian category $^t\mathscr{M}^1(k)[p^{-1}]$ of torsion 1-motives over $k$ which is abelian (described in \cite[App. C]{BVK}), and a fully faithful functor $$\mathscr{M}^1(k)[p^{-1}] \hookrightarrow {^t}\mathscr{M}^1(k)[p^{-1}]$$ \cite[C.5.3]{BVK}.  This provides $\mathscr{M}^1(k)[p^{-1}]$ with an exact structure with respect to which the $\ell$-adic realization functors described below are exact \cite[C.6.2]{BVK}.
\end{rem}
\begin{pg}
Let $\ell$ be a prime distinct from the characteristic of $k$.  We review the $\ell$-adic realization functor from 1-motives over $k$ to $\ell$-adic representations of $\mathrm{Gal}(\overline{k}/k)$.  Let $M = [L \stackrel{\alpha}{\rightarrow} G]$ be a 1-motive.  For any integer $n$ prime to $\mathrm{char \ }k$, we set $M/n$ to be the cone of $\cdot n: M \rightarrow M$.  We then set $T_{\mathbb{Z}/n}(M) = H^{-1}(M/n)$.  More concretely, the $\overline{k}$-points of $T_{\mathbb{Z}/n}(M)$ can be written
\begin{equation*}
T_{\mathbb{Z}/n}(M) = \frac{\{(x,g) \in L \times G(\overline{k}) \mid u(x) = -mg\}}{\{(mx,-u(x)) \mid x \in L\} }.
\end{equation*}
From the exact triangle
\begin{equation*}
G \rightarrow M \rightarrow L \rightarrow G[1],
\end{equation*}
we get an exact triangle
\begin{equation*}
G/n \rightarrow M/n \rightarrow L/n \rightarrow G/n[1].
\end{equation*}
Here $G/n$ and $L/n$ are defined as cones of multiplication by $n$, in the same manner as $M/n$.  Taking cohomology sheaves, we get an exact sequence
\begin{equation}\label{1motiveexact}
0 \rightarrow {_n}G \rightarrow T_{\mathbb{Z}/n}(M) \rightarrow L/nL \rightarrow 0.
\end{equation}
Therefore $T_{\mathbb{Z}/n}(M)$ is an etale sheaf, and can be fully described by its $\overline{k}$-points together with the action by $\mathrm{Gal}(\overline{k}/k)$.  Now set $$T_{\ell}M := \ilim_n T_{\mathbb{Z}/\ell^n}(M).$$  Because the collection $( {_{\ell^n}}G)_{n}$ satisfies the Mittag-Leffler condition, we get an exact sequence
\begin{equation*}
0 \rightarrow T_{\ell} G \rightarrow T_{\ell}M \rightarrow \ilim_n L/\ell^n L \rightarrow 0.
\end{equation*}
where (as usual) $T_\ell G$ is the Tate module of the $\overline{k}$-points of $G$.  Finally, we set
\begin{equation*}
V_\ell M := T_\ell M \otimes_\mathbb{Z_\ell} \mathbb{Q}_\ell.
\end{equation*} 
Therefore we have defined functors
\begin{equation*}
\hat{T}_p := \prod_{\ell \not= p} T_\ell: \ \ \mathscr{M}^1(k)[p^{-1}] \longrightarrow \prod_{\ell \not= p} \mathrm{Rep}_{\mathbb{Z}_\ell}(\mathrm{Gal}(\overline{k}/k))
\end{equation*}
and
\begin{equation*}
V_\ell: \mathscr{M}^1(k) \otimes \mathbb{Q} \rightarrow \mathrm{Rep}_{\mathbb{Q}_\ell}(\mathrm{Gal}(\overline{k}/k)).
\end{equation*}
\end{pg}
\begin{prop}\label{faithful}
The functors $\hat{T}_p$ and $V_\ell$ defined above are exact, faithful, and reflect isomorphisms.
\end{prop}
\begin{proof}
We show the statement for $V_\ell$; the statement for $\hat{T}_p$ is no harder. Let $0 \rightarrow M' \rightarrow M \rightarrow M'' \rightarrow 0$ be an exact sequence of 1-motives for $\mathbb{Q}$-isogeny.  Letting $M = [L \rightarrow G]$, $M' = [L' \rightarrow G']$, $M'' = [L'' \rightarrow G'']$, we have sequences
\begin{align*}
0 \rightarrow G' &\rightarrow G \rightarrow G'' \rightarrow 0 \ \ \ \mathrm{and} \\
0 \rightarrow L' &\rightarrow L \rightarrow L'' \rightarrow 0
\end{align*}
which are exact up to isogeny.  Therefore, we have exact sequences
\begin{align*}
0 \rightarrow V_\ell G' &\rightarrow V_\ell G \rightarrow V_\ell G'' \rightarrow 0 \ \ \ \mathrm{and} \\
0 \rightarrow V_\ell L' &\rightarrow V_\ell L \rightarrow V_\ell L'' \rightarrow 0.
\end{align*}
Using the functoriality of $V_\ell$, we get a commuting diagram
\begin{equation*}
\begin{CD}
@. 0 @. 0 @. 0 \\
@. @VVV @VVV @VVV @. \\
0 @>>> V_\ell G' @>>> V_\ell M' @>>> V_\ell L' @>>> 0 \\
@. @VVV @VVV @VVV @. \\
0 @>>> V_\ell G @>>> V_\ell M @>>> V_\ell L @>>> 0 \\
@. @VVV @VVV @VVV @. \\
0 @>>> V_\ell G'' @>>> V_\ell M'' @>>> V_\ell L'' @>>> 0 \\
@. @VVV @VVV @VVV @. \\
@. 0 @. 0 @. 0 
\end{CD}
\end{equation*}
with exact rows, and such that the left and right columns are exact.  By standard homological algebra, this implies that the middle column is also exact, proving that $V_\ell$ is an exact functor.

To show that $V_\ell$ is faithful, it suffices to show: given a morphism $F = M \rightarrow M'$ of 1-motives, if $V_\ell F = 0$ then $F = 0$.  Suppose that $F$ is given by a commuting diagram
\begin{equation*}
\begin{CD}
L @> f >> L' \\
@VVV @VVV \\
G @> g >> G'.
\end{CD}
\end{equation*}
Then we get that $V_{\ell} f = V_\ell g = 0$.  This immediately implies that $f = 0$ since $L(\overline{k})$ and $L'(\overline{k})$ are finitely generated free abelian groups.   To show that $g = 0$, note that by exactness of $V_\ell$ we have $$V_\ell(\mathrm{Ker}(g)) = \mathrm{Ker}(V_\ell g) = V_\ell G.$$  This implies that $\mathrm{Ker}(g)$ is a closed subvariety of $G$ of equal dimension to $G$.  Therefore $\mathrm{Ker}(g) = G$, i.e., $g = 0$.  The fact that $V_\ell$ reflects isomorphisms follows formally from being exact and faithful.
\end{proof}
\subsection{Cartier duals of 1-motives}\label{cartierduals} In \cite[10.2]{hodge3} there is defined a notion of Cartier duality for 1-motives.  We briefly recall this definition below.

Suppose given a 1-motive $M = [L \rightarrow G]$, which we write as a commutative diagram
\begin{equation*}
\begin{CD}
@. @. L @. @. \\
@. @. @VV f V @. \\
0 @>>> T @> g >> G @> h >> A @>>> 0.
\end{CD}
\end{equation*}
The Cartier dual 1-motive is then defined by a diagram
\begin{equation*}
\begin{CD}
@. @. T^{\vee} @. @. \\
@. @. @VV g^{\vee} V @. \\
0 @>>> L^{\vee} @> f^{\vee} >> G^u @>h^{\vee} >> A^\vee @>>> 0.
\end{CD}
\end{equation*}
where $T^\vee$, $L^\vee$ and $A^\vee$ are the usual duals, while $G^u$ is defined as follows:  consider the 1-motive $M/W_{-2}M = [L \rightarrow A]$.  We have an exact sequence of 1-motives
\begin{equation*}
0 \rightarrow A \rightarrow M/W_{-2}M \rightarrow L[1] \rightarrow 0.
\end{equation*}
Applying $\mathscr{RH}om(-,\mathbb{G}_m)$ to this sequence of complexes of $fppf$-sheaves and taking cohomology yields an exact sequence
\begin{equation*}
0 \rightarrow L^\vee \rightarrow \mathscr{E}xt^1(M/W_{-2}M,\mathbb{G}_m) \rightarrow A^{\vee} \rightarrow 0.
\end{equation*}
We define $G^u$ to be the group scheme $\mathscr{E}xt^1(M/W_{-2}M,\mathbb{G}_m)$.

It remains to define the map $g^\vee: T^\vee \rightarrow G^u$.  By the definition of $G^u$, this means that for every $x \in T^\vee$, we must give
\begin{enumerate}
\item an extension $\tilde{x}$ of $A$ by $\mathbb{G}_m$, and
\item a trivialization of the pullback of $\tilde{x}$ via $h \circ f: L \rightarrow A$. 
\end{enumerate} 
Since $T^\vee = \mathrm{Hom}(T,\mathbb{G}_m)$, we can let $\tilde{x}$ be the pushforward extension $x_*G \in \mathscr{E}xt^1(A,\mathbb{G}_m)$ (where $x \in \mathrm{Hom}(T,\mathbb{G}_m)$), defining part (1) of our desired map $g^\vee: T^\vee \rightarrow G^u$.  The trivialization of part (2) is determined by the fact that $h \circ f: L \rightarrow A$ lifts (trivially) to $f: L \rightarrow G$.

\begin{pg}
The key property about Cartier duals we will use is the following:
\end{pg}
\begin{prop}
Let $M = [L \rightarrow G]$ be a 1-motive and $M^\vee$ its Cartier dual.  Then for every $n$ prime to the characteristic of $k$ there is a canonical perfect pairing
\begin{equation*}
T_{\mathbb{Z}/n}M \otimes T_{\mathbb{Z}/n}M^\vee \rightarrow \mathbb{Z}/n(1)
\end{equation*}
which is functorial in $M$.
\end{prop}
\begin{proof}
This is \cite[10.2.5]{hodge3}.
\end{proof}

\section{Construction of $M_{D,E}^1(\overline{X})$}
\begin{pg}\label{61}
Fix a perfect field $k$ of characteristic $p \geq 0$ and an algebraic closure $\overline{k}$.  Let $\overline{X}$ be a proper reduced $k$-scheme with two reduced closed subschemes $D$ and $E$ such that $D \cap E = \emptyset$.  Let $X = \overline{X} - E$, $\tilde{U} = \overline{X} - D$ and $U = X - D = \tilde{U} - E$.  We label our various maps as follows:
\begin{equation*}
\xymatrix{
U \ar@{^{(}->}[r]^{\tilde{j}} \ar@{^{(}->}[d]^{\tilde{u}} & X \ar@{^{(}->}[d]^{u} & D \ar@{_{(}->}[l]_{\tilde{i}} \ar@{=}[d]  \\
{\tilde{U}} \ar@{^{(}->}[r]^{j} & {\overline{X}} & D \ar@{_{(}->}[l]_{i} \\
E \ar@{^{(}->}[u]^{\tilde{v}} \ar@{=}[r] & E \ar@{^{(}->}[u]^{v}
}
\end{equation*}
Here in each row and column, the term in the middle is the union of the outer two terms.  We then define $$H_{D,E}^i(\overline{X},\mathscr{F}) := H^i(X_{\overline{k}},\tilde{j}_! \mathscr{F})$$ for an \'{e}tale sheaf $\mathscr{F}$ on $U_{et}$.  In this section we define a 1-motive $M_{D,E}^1(\overline{X})$ such that we have isomorphisms for $\ell \not= p$ $$T_\ell M_{D,E}^1(\overline{X}) \stackrel{\sim}{\longrightarrow} H_{D,E}^1(\overline{X},\mathbb{Z}_\ell(1)),$$ functorial in the triple $(\overline{X},D,E)$.  
\end{pg}
\begin{pg}\label{62}
To give the construction of $M_{D,E}^1(\overline{X})$, we start by choosing (via \cite[3.1]{dejong}) a simplicial scheme $\overline{\pi}_\bullet: \overline{X}_\bullet \rightarrow \overline{X}$ with each $\overline{X}_n$ proper and smooth, such that the inverse images $D_\bullet := \overline{\pi}_\bullet^{-1}(D)_{red}$ and $E_\bullet := \overline{\pi}_\bullet^{-1}(E)_{red}$ are simple normal crossings divisors.  We let $X_\bullet = \overline{X}_\bullet - E_\bullet$, $\tilde{U}_{\bullet} = \overline{X}_\bullet - D_\bullet$ and $U_\bullet = \overline{X}_\bullet - (D_\bullet \cup E_\bullet)$, and label our maps in the same way as before (with a subscript for simplicial index):
\begin{equation*}
\xymatrix{
U_{\bullet} \ar@{^{(}->}[r]^{\tilde{j}_{\bullet}} \ar@{^{(}->}[d]^{\tilde{u}_{\bullet}} & X_{\bullet} \ar@{^{(}->}[d]^{u_{\bullet}} & D_{\bullet} \ar@{_{(}->}[l]_{\tilde{i}_{\bullet}} \ar@{=}[d]  \\
{\tilde{U}_{\bullet}} \ar@{^{(}->}[r]^{j_{\bullet}} & {\overline{X}_{\bullet}} & D_{\bullet} \ar@{_{(}->}[l]_{i_{\bullet}} \\
E_{\bullet} \ar@{^{(}->}[u]^{\tilde{v}_{\bullet}} \ar@{=}[r] & E_{\bullet} \ar@{^{(}->}[u]^{v_{\bullet}}
}
\end{equation*}
\end{pg}
\begin{pg}
 Before constructing $M_{D,E}^1(\overline{X})$, we must review the relative Picard group of the simplicial pair $(\overline{X}_\bullet,D_\bullet)$, which is defined as
\begin{equation*}
\mathrm{Pic}(\overline{X}_\bullet,D_\bullet) := H^1(\overline{X}_\bullet, \mathrm{Ker}(\mathbb{G}_{m,\overline{X}_\bullet} \rightarrow (i_\bullet)_*\mathbb{G}_{m,D_\bullet})).
\end{equation*}
Intuitively, $\mathrm{Pic}(\overline{X}_\bullet,D_\bullet)$ classifies isomorphism classes of pairs $(\mathscr{L}^\bullet,\sigma: \mathcal{O}_{D_\bullet} \stackrel{\sim}{\rightarrow} \mathscr{L}^\bullet\vert_{D_\bullet})$, where $\mathscr{L}^\bullet$ is an invertible sheaf on $\overline{X}_\bullet$ and $\sigma$ is a trivialization of $\mathscr{L}^\bullet$ on $D_\bullet$.  We can also define an associated sheaf $\mathbf{Pic}_{\overline{X}_\bullet,D_\bullet}$, defined as the fppf-sheafification of the functor on $(Sch/k)_{fppf}$,
\begin{equation*}
Y \mapsto \mathrm{Pic}(\overline{X}_\bullet \times Y,D_\bullet \times Y).
\end{equation*}

\end{pg}
\begin{prop}\label{simprelpic}
The sheaf $\mathbf{Pic}_{\overline{X}_\bullet,D_\bullet}$ defined above is representable.  Moreover, let $\mathbf{Pic}^{0}_{\overline{X}_\bullet,D_\bullet}$ denote the connected component of the identity of this group scheme.  Then the reduction $\mathbf{Pic}^{0,red}_{\overline{X}_\bullet,D_\bullet}$ is a semiabelian variety, and we have an exact sequence
\begin{equation*}
0 \rightarrow \mathbf{Pic}^{0,red}_{\overline{X}_\bullet,D_\bullet} \rightarrow \mathbf{Pic}^{red}_{\overline{X}_\bullet,D_\bullet} \rightarrow \mathbf{NS}_{\overline{X}_\bullet,D_\bullet} \rightarrow 0,
\end{equation*}
where $\mathbf{NS}_{\overline{X}_\bullet,D_\bullet}$ is an \'{e}tale group scheme over $k$ whose $\overline{k}$-points form a finitely generated abelian group.
\end{prop}
\begin{proof}
Let $\mathbb{G}_{m,\overline{X}_\bullet,D_\bullet} = \mathrm{ker}(\mathbb{G}_{m,\overline{X}_\bullet} \rightarrow (i_{\bullet})_*\mathbb{G}_{m,D_\bullet})$, and let $\overline{p}_\bullet: \overline{X} \rightarrow \mathrm{Spec \ }k$ and $p_{D,\bullet}: D \rightarrow \mathrm{Spec \ }k$ be the structure maps.  Then we have
\begin{equation*}
\mathbf{Pic}_{\overline{X}_\bullet,D_\bullet} = R^1(\overline{p}_\bullet)_*\mathbb{G}_{m,\overline{X}_\bullet,D_\bullet},
\end{equation*}
and an exact sequence of fppf sheaves on $Sch/k$
\begin{equation*}
R^0(\overline{p}_\bullet)_* \mathbb{G}_{m,\overline{X}_\bullet} \stackrel{a}{\longrightarrow} R^0(p_{D,\bullet})_*\mathbb{G}_{m,D_\bullet} \rightarrow \mathbf{Pic}_{\overline{X}_\bullet,D_\bullet} \rightarrow \mathbf{Pic}_{\overline{X}_\bullet} \stackrel{b}{\longrightarrow} \mathbf{Pic}_{D_\bullet}.
\end{equation*}
The two sheaves on the right are representable by \cite[Appendix A.2]{albpic}, and moreover (by the same reference) the reduction $\mathbf{Pic}^{0,red}_{\overline{X}_\bullet}$ of the connected component of the identity is a semiabelian variety.  It is clear that the the two left-hand sheaves are representable by tori.  This gives us a short exact sequence 
\begin{equation*}
0 \rightarrow \mathrm{coker}(a) \rightarrow \mathbf{Pic}_{\overline{X}_\bullet,D_\bullet} \rightarrow \mathrm{ker}(b) \rightarrow 0
\end{equation*}
where $\mathrm{coker}(a)$ is a torus and $\mathrm{ker}(b)$ is (after taking the reduced part) an extension of a semi-abelian variety by a finitely generated \'{e}tale group scheme.  This implies the proposition statement.
\end{proof}
\begin{pg}\label{pg65}
Now consider the divisor $E_\bullet$.  Let $$\mathrm{Div}_{E_\bullet}(\overline{X}_\bullet) := \mathrm{ker}(p_1^* - p_2^*: \mathrm{Div}_{E_0}(\overline{X}_0) \rightarrow \mathrm{Div}_{E_1}(\overline{X}_1)),$$ where $p_1,p_2: \overline{X}_1 \rightarrow \overline{X}_0$ are the simplicial projections.  Because $D_n \cap E_n = \emptyset$ in each $\overline{X}_n$, there is a well-defined map
\begin{equation*}
cl: \mathrm{Div}_{E_\bullet}(\overline{X}_\bullet) \longrightarrow \mathrm{Pic}(\overline{X}_\bullet,D_\bullet).
\end{equation*}
Concretely, given a divisor $A$ supported on $E_0$, we have the associated line bundle $\mathcal{O}(A)$ on $X_0$ and meromorphic section $s: \mathcal{O}_{X_0} \rightarrow \mathcal{O}(A)$.  Since $D_0 \cap E_0 = \emptyset$, this induces an isomorphism $\mathcal{O}_{D_0} \stackrel{\sim}{\rightarrow} \mathcal{O}(A)\vert_{D_0}$.  Moreover, since $p_1^*A = p_2^*A$ as divisors, the meromorphic sections of $\mathcal{O}(p_1^*A)$ and $\mathcal{O}(p_2^*A)$ yield a canonical isomorphism $\rho: p_1^*\mathcal{O}(A) \stackrel{\sim}{\rightarrow} p_2^*\mathcal{O}(A)$ verifying a cocycle condition, so $\mathcal{O}(A)$ defines a line bundle on $X_\bullet$ with a trivialization on $D_\bullet$. 

With these preliminaries in hand, we can make the following definition:
\end{pg}
\begin{defn}
Let $(\overline{X},D,E)$ be as above.  Choose a proper hypercover $\overline{\pi}_\bullet: \overline{X}_\bullet \rightarrow \overline{X}$ such that each $\overline{X}_n$ is proper smooth and the inverse images of $D$ and $E$ are simple normal crossings divisors in each $X_n$.
Let $\mathrm{Div}_{E_\bullet}^0(\overline{X}_\bullet)$ be the subgroup of divisors which map to 0 in $NS(\overline{X}_\bullet,D_\bullet)$.  Let $\mathbf{Div}_{E_\bullet}^0(\overline{X}_\bullet)$ be the natural extension of $\mathrm{Div}_{E_\bullet}^0(\overline{X}_\bullet)$ to an \'{e}tale sheaf by the rule
\begin{equation*}
U \mapsto \mathrm{Div}_{E_\bullet \times U}(\overline{X}_\bullet \times U)
\end{equation*}
for smooth $U$. We then define
\begin{equation*}
M_{D,E}^1(\overline{X}) := [\mathbf{Div}_{E_\bullet}^0(\overline{X}_\bullet) \rightarrow \mathbf{Pic}^{0,red}_{\overline{X}_\bullet,D_\bullet}].
\end{equation*}
\end{defn}
Of course, we must show that $M_{D,E}^1(\overline{X})$ is independent of the choices we made in its construction, and that it is contravariantly functorial for morphisms of triples $(\overline{X},D,E)$.  To do this, we must first discuss the $\ell$-adic realizations of $M_{D,E}^1(\overline{X})$.

\subsection{$\ell$-adic realization of $M_{D,E}^1(\overline{X})$} 
Our goal in this section is to show that there is a natural isomorphism (for $\ell \not= p$)
\begin{equation*}
T_\ell M_{D,E}^1(\overline{X}) \stackrel{\sim}{\longrightarrow} H_{D,E}^1(\overline{X}_{\overline{k}},\mathbb{Z}_\ell(1)).
\end{equation*}
To show this, we may assume $k = \overline{k}$ (this is only to reduce the number of subscripts).  Continuing with the notation of \ref{61} and \ref{62} (so in particular we have chosen a simplicial resolution $\overline{X}_\bullet \rightarrow \overline{X}$ which appears in the definition of $M_{D,E}^1(\overline{X})$), we have
\begin{equation*}
H_{D,E}^1(\overline{X},\mathbb{Z}_\ell(1)) \stackrel{\sim}{\longrightarrow} H_{D_\bullet,E_\bullet}^1(\overline{X}_\bullet,\mathbb{Z}_\ell(1)) := H^1(X_\bullet,\tilde{j}_{\bullet,!} \mathbb{Z}_\ell(1)),
\end{equation*}
where the left-hand isomorphism is due to cohomological descent.  Note that the simplicial extension by zero functor $\tilde{j}_{\bullet,!}$ is defined by the same rule as in the usual case of schemes; in particular $(\tilde{j}_{\bullet,!}\mathscr{F}^\bullet)^n = {\tilde{j}_{n,!}}\mathscr{F}^n$ for any simplicial sheaf $\mathscr{F}^\bullet$ on $U_\bullet$.  
\begin{pg}
Recall that we have a commutative diagram of inclusions
\begin{equation*}
\xymatrix{
U_\bullet \ar@{^{(}->}[d]^{\tilde{u}_\bullet} \ar@{^{(}->}[r]^{\tilde{j}_\bullet} & X_\bullet \ar@{^{(}->}[d]^{u_\bullet} \\
\tilde{U}_\bullet \ar@{^{(}->}[r]^{j_\bullet} & \overline{X}_\bullet. }
\end{equation*}
Because $D_\bullet$ and $E_\bullet$ are disjoint, we have 
\begin{equation*}
R{u_{\bullet,*}}{\tilde{j}_{\bullet,!}} \cong {j_{\bullet,!}} R{\tilde{u}_{\bullet, *}}
\end{equation*}
as functors on $D^b(U_\bullet)$.  Let $\tilde{v}_\bullet: E_\bullet \hookrightarrow \tilde{U}_\bullet$ be the inclusion, and consider the exact triangle in $D_c^b(\tilde{U}_\bullet)$
\begin{equation*}
\tilde{v}_{\bullet *}R{\tilde{v}_\bullet}^!\mu_{n,\tilde{U}_\bullet} \rightarrow \mu_{n,\tilde{U}_\bullet} \rightarrow R\tilde{u}_{\bullet *} \mu_{n,U_\bullet} \rightarrow \tilde{v}_{\bullet *}R{\tilde{v}_\bullet}^!\mu_{n,\tilde{U}_\bullet}[1]
\end{equation*}
(where $n$ is prime to the characteristic $p$).  Applying ${j}_{\bullet,!}$ to this triangle gives a triangle (note that $v_{\bullet ,*} = j_{\bullet ,!} \circ \tilde{v}_{\bullet ,*}$)
\begin{equation}\label{eqn59}
v_{\bullet,*}Rv_\bullet^!\mu_{n,\overline{X}_\bullet} \rightarrow {j}_{\bullet,!}\mu_{n,\tilde{U}_\bullet} \rightarrow {j}_{\bullet,!}R{\tilde{u}}_{\bullet *}\mu_{n,U_\bullet} \rightarrow v_{\bullet,*}Rv_\bullet^!\mu_{n,\overline{X}_\bullet}[1].
\end{equation}
Taking cohomology of this triangle, and using the fact that $R{u_{\bullet}}_*{\tilde{j}_{\bullet,!}} \cong {{j}_{\bullet,!}} R{\tilde{u}_{\bullet *}}$, we get an exact sequence
\begin{equation}\label{eqn60}
0 \rightarrow H_c^1(\tilde{U}_\bullet,\mu_n) \rightarrow H_{D_\bullet,E_\bullet}^1(\overline{X}_\bullet,\mu_n) \rightarrow H_{E_\bullet}^2(\overline{X}_\bullet,\mu_n) \rightarrow H_c^2(\tilde{U}_\bullet,\mu_n).
\end{equation}
The following propositions give motivic interpretations of the groups appearing in this sequence:
\end{pg}
\begin{prop}\label{lefthand}
We have a natural bijection
\begin{equation*}
H_c^1(\tilde{U}_\bullet,\mu_n) \longleftrightarrow \mathbf{Pic}^0(\overline{X}_\bullet,D_\bullet)[n]
\end{equation*}
of $H_c^1(\tilde{U}_\bullet,\mu_n)$ with the $n$-torsion in $\mathbf{Pic}^0(\overline{X}_\bullet,D_\bullet)$.
\end{prop}
\begin{proof}
Let $\mathbb{G}_{m,\overline{X}_\bullet,D_\bullet} := \mathrm{Ker}(\mathbb{G}_{m,\overline{X}_\bullet} \rightarrow i_{\bullet,*}\mathbb{G}_{m,D_\bullet})$.  Then we have a commutative diagram of sheaves on $\overline{X}_\bullet$, where the rows and columns are exact:
\begin{equation*}
\begin{CD}
@. 0 @. 0 @. 0 @. \\
@. @VVV @VVV @VVV \\
0 @>>> {{j}_{\bullet,!}\mu_{n,\tilde{U}_\bullet}} @>>> {\mathbb{G}_{m,\overline{X}_\bullet,D_\bullet}} @> {\cdot n} >> {\mathbb{G}_{m,\overline{X}_\bullet,D_\bullet}} @>>> 0 \\
@. @VVV @VVV @VVV \\
0 @>>> \mu_{n,\overline{X}_\bullet} @>>> {\mathbb{G}_{m,\overline{X}_\bullet}} @> {\cdot n} >> {\mathbb{G}_{m,\overline{X}_\bullet}} @>>> 0 \\
@. @VVV @VVV @VVV \\
0 @>>> {i_{\bullet,_*}\mu_{n,D_\bullet}} @>>> {i_{\bullet,*}\mathbb{G}_{m,D_\bullet}} @> {\cdot n} >> {i_{\bullet,*}\mathbb{G}_{m,D_\bullet}} @>>> 0 \\
@. @VVV @VVV @VVV \\
@. 0 @. 0 @. 0 
\end{CD}
\end{equation*}
Taking cohomology along the top row, we have an exact sequence
\begin{equation}\label{eq61}
H_c^0(\tilde{U}_\bullet,\mathbb{G}_m) \stackrel{\cdot n}{\longrightarrow} H_c^0(\tilde{U}_\bullet,\mathbb{G}_m) \longrightarrow H_c^1(\tilde{U}_\bullet,\mu_n) \longrightarrow \mathrm{Pic}(\overline{X}_\bullet,D_\bullet) \stackrel{\cdot n}{\longrightarrow} \mathrm{Pic}(\overline{X}_\bullet,D_\bullet).
\end{equation}
But $H_c^0(\tilde{U}_\bullet,\mathbb{G}_m) = \mathrm{Ker}(H^0(X_\bullet,\mathbb{G}_m) \rightarrow H^0(D_\bullet,\mathbb{G}_m))$ is a torus; since the $\overline{k}$-points of a torus are divisible, we have $$H_c^1(\tilde{U}_\bullet,\mu_n) = \mathrm{Pic}(\overline{X}_\bullet,D_\bullet)[n] = \mathbf{Pic}^0(\overline{X}_\bullet,D_\bullet)[n]$$ as desired.
\end{proof}
Notice that if we extend the long exact sequence \ref{eq61} a little, we have an injection
\begin{equation}\label{extended}
\mathrm{Pic}(\overline{X}_\bullet,D_\bullet)/n\mathrm{Pic}(\overline{X}_\bullet,D_\bullet) \hookrightarrow H_c^2(\tilde{U}_\bullet,\mu_n).
\end{equation}
Because $\mathrm{Pic}^0(\overline{X}_\bullet,D_\bullet)$ is divisible, we have
\begin{equation*}
\mathrm{Pic}(\overline{X}_\bullet,D_\bullet)/n\mathrm{Pic}(\overline{X}_\bullet,D_\bullet) \cong NS(\overline{X}_\bullet,D_\bullet)/nNS(\overline{X}_\bullet,D_\bullet) = NS(\overline{X}_\bullet,D_\bullet) \otimes \mathbb {Z}/(n).
\end{equation*}
\begin{prop}\label{righthand}
We have a canonical isomorphism
\begin{equation*}
H_{E_\bullet}^2(\overline{X}_\bullet,\mu_n) \cong \mathrm{Div}_{E_\bullet}(\overline{X}_\bullet) \otimes \mathbb{Z}/(n),
\end{equation*}
and the map $H_{E_\bullet}^2(\overline{X}_\bullet,\mu_n) \rightarrow H_c^2(\tilde{U}_\bullet,\mu_n)$ obtained from sequence \ref{eqn60} factors as 
\begin{equation*}
\mathrm{Div}_{E_\bullet}(\overline{X}_{\bullet}) \otimes \mathbb{Z}/(n) \rightarrow NS(\overline{X}_\bullet,D_\bullet) \otimes \mathbb{Z}/(n) \hookrightarrow H_c^2(\tilde{U}_\bullet,\mu_n),
\end{equation*}
where the map $\mathrm{Div}_{E_\bullet}(\overline{X}_\bullet) \rightarrow NS(\overline{X}_\bullet,D_\bullet)$ is the cycle class map of \ref{pg65}, and the injection $NS(\overline{X}_\bullet,D_\bullet) \otimes \mathbb{Z}/(n) \hookrightarrow H_c^2(\tilde{U}_\bullet,\mu_n)$ is induced by \ref{extended}.
\end{prop}
\begin{proof}
The statement that $H_{E}^2(\overline{X},\mu_n) = \mathrm{Div}_E(\overline{X}) \otimes \mathbb{Z}/(n)$ for a closed subscheme $E$ of a proper smooth scheme $X$ over $k$ is well known.  The case of simplicial schemes follows by considering the spectral sequence $H^q_{E_p}(\overline{X}_p,\mu_n) \Rightarrow H_{E_\bullet}^{p+q}(\overline{X}_\bullet,\mu_n)$.  The claim regarding the map $H_{E_\bullet}^2(\overline{X}_\bullet,\mu_n) \rightarrow H_c^2(\tilde{U}_\bullet,\mu_n)$ is essentially a simplicial variant of \ref{cpctagrees} and \ref{cpctfactors}: first consider the canonical map in $D_c^b(\overline{X}_\bullet)$ 
\begin{equation*}
\alpha: v_{\bullet,*}v_{\bullet}^!\mathbb{G}_{m,\overline{X}_\bullet} \longrightarrow \mathbb{G}_{m,\overline{X}_{\bullet}}.
\end{equation*}
Because $D_\bullet$ and $E_\bullet$ are disjoint, the composition
\begin{equation*}
v_{\bullet,*}v_{\bullet}^!\mathbb{G}_{m,\overline{X}_\bullet} \stackrel{\alpha}{\longrightarrow} \mathbb{G}_{m,\overline{X}_{\bullet}} \longrightarrow i_{\bullet,*}\mathbb{G}_{m,D_\bullet}
\end{equation*}

is the zero map, implying that $\alpha$ factors through a unique map
\begin{equation*}
\tilde{\alpha}: v_{\bullet,*}v_{\bullet}^!\mathbb{G}_{m,\overline{X}_\bullet} \longrightarrow \mathbb{G}_{m,\overline{X}_\bullet,D_\bullet}.
\end{equation*}
Taking global sections induces a map
\begin{equation*}
H_{E_\bullet}^1(\overline{X}_\bullet,\mathbb{G}_m) \rightarrow \mathrm{Pic}(\overline{X}_\bullet,D_\bullet).
\end{equation*}
The group on the left is canonically isomorphic to $\mathrm{Div}_{E_\bullet}(X_\bullet)$, and we leave it to the reader to check that this map sends a divisor $W_\bullet$ to $(\mathcal{O}(W_\bullet),s)$ where $s$ is the canonical meromorphic section of $\mathcal{O}(W_\bullet)$ (the verification of this fact is the same as in the proof of \ref{cpctagrees}).  Using the Kummer exact sequence, we conclude that the map $H_{E_\bullet}^2(\overline{X}_\bullet,\mu_n) \rightarrow H_c^2(\tilde{U}_\bullet,\mu_n)$ can be described as claimed in the proposition statement.
\end{proof}
\begin{pg}
Summarizing the last two propositions, we see that we have a diagram
\begin{equation*}
\begin{CD}
0 @>>> \mathrm{Pic}^0(\overline{X}_\bullet,D_\bullet)[n] @>>> T_{\mathbb{Z}/n}(M_{D,E}^1(\overline{X})) @>>> \mathrm{Div}_{E_\bullet}^0(\overline{X}_\bullet) \otimes \mathbb{Z}/n @>>> 0 \\
@. @V \sim VV @.@V \sim VV \\
0 @>>> H_c^1(\tilde{U}_\bullet,\mu_n) @>>> H_{D,E}^1(\overline{X},\mu_n) @>>> \mathrm{Ker}(H_{E_\bullet}^2(\overline{X}_\bullet,\mu_n) \rightarrow H_c^2(\tilde{U}_\bullet,\mu_n)) @>>> 0
\end{CD}
\end{equation*}
where the left-hand and right-hand maps are isomorphisms.  Therefore we need only define a map $$f: T_{\mathbb{Z}/n}(M_{D,E}^1(\overline{X})) \longrightarrow H_{D,E}^1(\overline{X},\mu_n)$$ fitting into the middle of the diagram, and by the five lemma it will be an isomorphism.  By definition, $T_{\mathbb{Z}/n}(M_{D,E}^1(\overline{X}))$ consists of data $(C,\mathscr{L}^\bullet,\varphi)$, where 
\begin{itemize}
\item $C \in \mathrm{Div}_{E_\bullet}(\overline{X}_\bullet)$, 
\item $\mathscr{L}^\bullet$ is a line bundle on $\overline{X}_\bullet$, and
\item $\varphi: \mathcal{O}_{D_\bullet} \stackrel{\sim}{\longrightarrow} \mathscr{L}^\bullet\vert_{D_\bullet}$ is an isomorphism.  We also require that
\item There exists at least one isomorphism $\eta: (\mathscr{L}^{\bullet})^{\otimes n} \stackrel{\sim}{\rightarrow} \mathcal{O}(-C)$ identifying $\varphi^{\otimes n}$ with the canonical meromorphic section of $\mathcal{O}(-C)$ restricted to $D_\bullet$ (which is an isomorphism since $C$ is disjoint from $D_\bullet$).   
\end{itemize}
We then mod out by elements of the form $(-nC,\mathcal{O}(C),s)$ where $s: \mathcal{O}_{\overline{X}_\bullet} \rightarrow \mathcal{O}(C)$ is the canonical meromorphic section.
\end{pg}
\begin{pg}
By the general machinery of sites \cite[tag 03AJ]{stacksproject} the group $H_{D,E}^1(\overline{X},\mu_n) = H_{D_\bullet,E_\bullet}^1(\overline{X}_\bullet,\mu_n) =  H^1(X_\bullet,j_{\bullet,!}\mu_n)$ is in bijection with isomorphism classes of $j_{\bullet,!}\mu_n$-torsors on $X_\bullet$.  Our map $f: T_{\mathbb{Z}/n}M_{D,E}^1(\overline{X}) \rightarrow H_{D,E}^1(\overline{X},\mu_n)$ is defined as follows: given an object $(C,\mathscr{L}^\bullet,\varphi) \in T_{\mathbb{Z}/n}M_{D,E}^1(\overline{X})$, choose an isomorphism $\eta$ as in bullet point (4) above.  Since $C$ is disjoint from $X_\bullet = \overline{X}_\bullet - E_\bullet$, $\eta$ defines a trivialization of $(\mathscr{L}^\bullet\vert_{X_\bullet})^{\otimes n}$ carrying $\varphi^{\otimes n}$ to the identity morphism of $\mathcal{O}_{D_\bullet}$.  We then set $f(C_\bullet,\mathscr{L}^\bullet,\varphi)$ to be the $j_{\bullet,!}\mu_n$-torsor of local isomorphisms $\mathcal{O}_{X_\bullet} \stackrel{\sim}{\rightarrow} \mathscr{L}^\bullet\vert_{X_\bullet}$ compatible with $\eta$ on $n$th tensor products and reducing to $\varphi$ on $D_\bullet$.
\end{pg}
\begin{pg}\label{welldefined}
We still must show that this map is well-defined.  First, suppose we chose a different isomorphism $\eta': (\mathscr{L}^\bullet)^{\otimes n} \stackrel{\sim}{\rightarrow} \mathcal{O}(-C)$.  Then $\eta$ and $\eta'$ differ by an element
\begin{equation*}
\alpha \in H^0(\overline{X}_\bullet,\tilde{j}_{\bullet,!}\mathbb{G}_m) = \mathrm{Ker}(H^0(\overline{X}_\bullet,\mathbb{G}_m) \rightarrow H^0(D_\bullet,\mathbb{G}_m)).
\end{equation*}
This group is a torus, which implies that we can choose an $n$th root $\sqrt[n]{\alpha}$.  Then if $\psi: \mathcal{O}_{X_\bullet} \stackrel{\sim}{\rightarrow} \mathscr{L}^\bullet\vert_{X_\bullet}$ is an isomorphism compatible with $\eta$, then $\sqrt[n]{\alpha}\psi$ is an isomorphism which is compatible with $\eta'$.  Therefore multiplication by $\sqrt[n]{\alpha}$ defines an isomorphism between the $j_{\bullet,!}\mu_n$-torsors defined by choosing $\eta$ or $\eta'$.  Moreover, it is clear that elements of the form $(-nC,\mathcal{O}(C),can)$ map to the trivial torsor, which shows that $f: T_{\mathbb{Z}/n}(M_{D,E}^1(\overline{X})) \rightarrow H_{D,E}^1(\overline{X},\mu_n)$ is well-defined.
\end{pg}
\begin{pg}
It remains to show that $f$ fits into the diagram 
\begin{equation*}
\begin{CD}
0 @>>> \mathrm{Pic}^0(\overline{X}_\bullet,D_\bullet)[n] @>>> T_{\mathbb{Z}/n}(M_{D,E}^1(\overline{X})) @>>> \mathrm{Div}_{E_\bullet}^0(\overline{X}_\bullet) \otimes \mathbb{Z}/n @>>> 0 \\
@. @VVV @V f VV @VVV \\
0 @>>> H_c^1(\tilde{U}_\bullet,\mu_n) @>>> H_{D,E}^1(\overline{X},\mu_n) @>>> \mathrm{Ker}(H_{E_\bullet}^2(\overline{X}_\bullet,\mu_n) \rightarrow H_c^2(\tilde{U}_\bullet,\mu_n)) @>>> 0.
\end{CD}
\end{equation*}
To check that we have such a diagram, we must show that $\mathrm{Pic}^0(\overline{X}_\bullet,D_\bullet)[n]$ is precisely the inverse image of $H_c^1(\tilde{U}_\bullet,\mu_n)$ under the mapping $f$.  As a subgroup of $T_{\mathbb{Z}/n}(M_{D,E}^1(\overline{X}))$, $\mathrm{Pic}^0(\overline{X}_\bullet,D_\bullet)$ consists of the elements of the form $(0,\mathscr{L}^\bullet,\varphi)$, i.e., the divisor is empty.  This means that $(\mathscr{L}^{\otimes n},\varphi^{\otimes n})$ is isomorphic to $(\mathcal{O}_{\overline{X}_\bullet},1)$ (rather than simply isomorphic when restricted to $X_\bullet$).  This is precisely the condition for an element to factor through the subgroup $H_c^1(\tilde{U}_\bullet,\mu_n)$ of $H_{D,E}^1(\overline{X}_\bullet,\mu_n)$.  Therefore we have the diagram above, and by Propositions \ref{lefthand} and \ref{righthand}, the left-hand and right-hand vertical arrows are isomorphisms.  This implies that $f$ is an isomorphism, so we have proved the following:
\end{pg}
\begin{prop}\label{mdexelladic}
For any triple $(\overline{X},D,E)$ as above, and any choice of simplicial cover $\overline{X}_\bullet \rightarrow \overline{X}$ defining $M_{D,E}^1(\overline{X})$, there is a natural isomorphism
\begin{equation*}
T_\ell M_{D,E}^1(\overline{X}) \stackrel{\sim}{\longrightarrow} H_{D,E}^1(\overline{X},\mathbb{Z}_\ell(1)).
\end{equation*}
\end{prop}
\subsection{Functoriality of $M_{D,E}^1(\overline{X})$}\label{pg63}
Let $(\overline{X},D,E)$ and $(\overline{Y},A,B)$ be triples consisting of a proper scheme $\overline{X}$ (resp. $\overline{Y}$), and two disjoint reduced closed subschemes $D$ and $E$ (resp. $A$ and $B$).  We let $X = \overline{X} - E$, $U = \overline{X} - (D \cup E)$, and $\tilde{U} = \overline{X} - D$.  We similarly let $Y = \overline{Y} - B$, $V = \overline{Y} - (A \cup B)$, and $\tilde{V} = \overline{Y} - A$.  We label the various maps between these spaces as follows:
\begin{equation*}
\xymatrix{
U \ar@{^{(}->}[r]^{\tilde{j}} \ar@{^{(}->}[d]^{\tilde{u}} & X \ar@{^{(}->}[d]^{u} & D \ar@{_{(}->}[l]_{\tilde{i}} \ar@{=}[d]  \\
{\tilde{U}} \ar@{^{(}->}[r]^{j} & {\overline{X}} & D \ar@{_{(}->}[l]_{i} \\
E \ar@{^{(}->}[u]^{\tilde{v}} \ar@{=}[r] & E \ar@{^{(}->}[u]^{v}
}
\end{equation*}
and
\begin{equation*}
\xymatrix{
V \ar@{^{(}->}[r]^{\tilde{b}} \ar@{^{(}->}[d]^{\tilde{r}} & Y \ar@{^{(}->}[d]^{r} & A \ar@{_{(}->}[l]_{\tilde{a}} \ar@{=}[d]  \\
{\tilde{V}} \ar@{^{(}->}[r]^{b} & {\overline{Y}} & A \ar@{_{(}->}[l]_{a} \\
B \ar@{^{(}->}[u]^{\tilde{s}} \ar@{=}[r] & B. \ar@{^{(}->}[u]^{s}
}
\end{equation*}
 We will describe the contravariant functoriality of the 1-motives $M_{D,E}^1(\overline{X})$ and $M_{A,B}(\overline{Y})$.  For this, consider a map $f: \overline{X} \rightarrow \overline{Y}$ such that
\begin{enumerate}
\item $f^{-1}(B) \subseteq E$, and
\item $f^{-1}(A) = D$.
\end{enumerate}
Notice that (2) implies that $f$ restricts to a proper map $\tilde{U} \rightarrow \tilde{V}$.  In this situation it will turn out that there is a well-defined map $M_{A,B}^1(\overline{Y}) \rightarrow M_{D,E}^1(\overline{X})$.

\begin{pg}
Now choose proper hypercovers $\overline{\pi}_\bullet: \overline{X}_\bullet \rightarrow \overline{X}$ and $\overline{\rho}_\bullet: \overline{Y}_\bullet \rightarrow \overline{Y}$ such that there is a map $\overline{f}_\bullet: \overline{X}_\bullet \rightarrow \overline{Y}_\bullet$ compatible with $\overline{f}: \overline{X} \rightarrow \overline{Y}$.  Set $X_\bullet = \overline{X}_\bullet \times_{\overline{X}} X$ and $Y_\bullet = \overline{Y}_\bullet \times_{\overline{Y}} Y$, and similarly for the other subschemes of $\overline{X}$ and $\overline{Y}$.  Using these simplicial covers, we can define
\begin{align*}
M_{D,E}^1(\overline{X}) &:= [\mathbf{Div}_{E_\bullet}^0(\overline{X}_\bullet) \rightarrow \mathbf{Pic}_{\overline{X}_\bullet,D_\bullet}^{0,red}] \ \ \mathrm{and} \\
M_{A,B}^1(\overline{Y}) &:= [\mathbf{Div}_{B_\bullet}^0(\overline{Y}_\bullet) \rightarrow \mathbf{Pic}_{\overline{Y}_\bullet,A_\bullet}^{0,red}].
\end{align*}
Since we have not yet shown that $M_{D,E}^1(X)$ is independent of the choice of hypercover, this is an abuse of notation.
We get a natural map
\begin{equation*}
\hat{f}^*: \mathbf{Pic}^{0,red}_{\overline{Y}_\bullet,A_\bullet} \rightarrow \mathbf{Pic}^{0,red}_{\overline{X}_\bullet,D_\bullet}
\end{equation*}
by pulling back a line bundle $\mathscr{L}$ to $\overline{X}_\bullet$; the trivialization of $\mathscr{L}^\bullet$ on $A_\bullet$ pulls back to a trivialization on $D_\bullet$ because $f^{-1}(A) = D$.  Similarly, because $f^{-1}(B) \subseteq E$, there is an induced pullback map on divisors
\begin{equation*}
\mathrm{Div}_{B_\bullet}^0(\overline{Y}_\bullet) \rightarrow \mathrm{Div}_{E_\bullet}^0(\overline{X}_\bullet).
\end{equation*}
Putting these maps together, we get a map of 1-motives $\hat{f}^*: M_{A,B}^1(\overline{Y}) \rightarrow M_{D,E}^1(\overline{X})$.
\end{pg}
\begin{pg}
In the situation of \ref{pg63}, because $f^{-1}(A) = D$ there is a morphism of pairs $(X,D) \rightarrow (Y,A)$.  Therefore there is an induced morphism of cohomology groups 
\begin{equation*}
H_{A,B}^1(\overline{Y},\mathbb{Z}_\ell(1)) = H^1(Y,b_!\mathbb{Z}_\ell(1)) \rightarrow H^1(X,j_!\mathbb{Z}_\ell(1)) = H_{D,E}^1(\overline{X},\mathbb{Z}_\ell(1)).
\end{equation*}
This morphism is ultimately induced by pullback of line bundles and hence agrees with the $\ell$-adic realization of the map on cohomology groups $\hat{f}^*: M_{A,B}^1(\overline{Y}) \rightarrow M_{D,E}^1(\overline{X})$.  More precisely, we have the following:
\end{pg}
\begin{prop}\label{m1compatibilityladic}
In the notation above, we have a commutative diagram (for $\ell \not= p)$
\begin{equation*}
\begin{CD}
T_\ell M_{D,E}^1(\overline{X}) @> \alpha_{\overline{X}} >> H_{D,E}^1(\overline{X},\mathbb{Z}_\ell(1)) \\
@V T_\ell \hat{f}^* VV @V f^* VV \\
T_\ell M_{A,B}^1(\overline{Y}) @> \alpha_{\overline{Y}} >> H_{A,B}^1(\overline{X},\mathbb{Z}_\ell(1)),
\end{CD}
\end{equation*}
where $\alpha_{\overline{X}}$ and $\alpha_{\overline{Y}}$ are the comparison isomorphisms of Proposition \ref{mdexelladic}.
\end{prop}
\begin{pg}\label{610}
We are now ready to show that the above definition of $M_{D,E}^1(\overline{X})$ is independent of the choice of simplicial cover.  Suppose that $\pi_\bullet : \overline{X}_\bullet \rightarrow \overline{X}$ and $\rho_\bullet : \overline{X}'_\bullet \rightarrow \overline{X}$ are two simplicial covers by proper smooth schemes, and set $E_\bullet = \pi_\bullet^{-1}(E)$, $E'_\bullet = \rho_\bullet^{-1}(E)$, etc.  Then as explained in \cite[pp. 26-31]{desc}, we can choose a third simplicial cover $\tau: \overline{X}''_\bullet \rightarrow \overline{X}$ with the property that there are simplicial maps
\begin{equation*}
\overline{X}_\bullet \stackrel{f_\bullet}{\longleftarrow} \overline{X}''_\bullet \stackrel{g_\bullet}{\longrightarrow} \overline{X}'_\bullet
\end{equation*}
lying over the identity on $\overline{X}$.
\end{pg}
\begin{prop}
Let $M_{D,E}^1(\overline{X})$, $M_{D,E}^1(\overline{X})'$, and $M_{D,E}^1(\overline{X})''$ be the 1-motives constructed by using the simplicial covers $\overline{X}_\bullet$, $\overline{X}'_\bullet$, and $\overline{X}''_\bullet$ respectively.  Then the induced morphisms of 1-motives
\begin{equation*}
M_{D,E}^1(\overline{X}) \stackrel{\hat{f}}{\longrightarrow} M_{D,E}^1(\overline{X})'' \stackrel{\hat{g}}{\longleftarrow} M_{D,E}^1(\overline{X})'
\end{equation*}
are isomorphisms in $\mathscr{M}^1(k)[1/p]$, and the composite isomorphism
\begin{equation*}
\hat{g}^{-1} \circ \hat{f}: M_{D,E}^1(\overline{X}) \stackrel{\sim}{\longrightarrow} M_{D,E}^1(\overline{X})'
\end{equation*}
is the unique isomorphism in $\mathscr{M}^1(k)[1/p]$ inducing the identity map $H_{D,E}^1(\overline{X},\mathbb{Z}_\ell(1)) \rightarrow H_{D,E}^1(\overline{X},\mathbb{Z}_\ell(1))$ for every $\ell \not= p$.
\end{prop}
Hence the 1-motive $M_{D,E}^1(\overline{X}) \in \mathscr{M}^1(k)[1/p]$ is (up to unique isomorphism) indpendent of the choice of simplicial cover.
\begin{proof}
By Proposition \ref{m1compatibilityladic}, for each $\ell \not= p$, the $\ell$-adic realizations $T_\ell \hat{f}$ and $T_\ell \hat{g}$ induce the identity map on $H_{D,E}^1(\overline{X},\mathbb{Z}_\ell(1))$ (since they lie over the identity on $X$).  Hence by Prop. \ref{faithful}, the maps $f$ and $g$ are isomorphisms in $\mathscr{M}^1(k)[1/p]$.  It is also clear that $f$ and $g$ are the only isomorphisms of 1-motives in $\mathscr{M}^1(k)[1/p]$ inducing the identity on $H_{D,E}^1(\overline{X},\mathbb{Z}_\ell(1))$ for each $\ell \not= p$, because the realizations $T_\ell f$, $T_\ell g$ are uniquely determined by this condition.
\end{proof}
The results of this section can be summarized as follows:
\begin{thm}
Let $(\overline{X},D,E)$ be a triple consisting of a proper finite type $k$-scheme $\overline{X}$, and two disjoint closed subschemes $D,E \subset \overline{X}$.  Then there exists a 1-motive $M_{D,E}^1(\overline{X})$ defined up to unique isomorphism, with a natural isomorphism
\begin{equation*}
T_\ell M_{D,E}^1(\overline{X}) \cong H_{D,E}^1(\overline{X}_{\overline{k}},\mathbb{Z}_\ell(1)) := H^1(X_{\overline{k}},j_!\mathbb{Z}_\ell(1))
\end{equation*}
for all $\ell \not= p$.  $M_{D,E}^1(\overline{X})$ is functorial for morphisms of triples $f: (\overline{X},D,E) \rightarrow (\overline{Y},A,B)$ such that $f^{-1}(D) = A$ and $f^{-1}(E) \subseteq B$.
\end{thm}
\subsection{The 1-motives $M^1(X)$ and $M_c^1(X)$}
As special cases of the above constuction, for any separated finite type $k$-scheme $X$ we can define 1-motives $M^1(X)$ and $M_c^1(X)$ which realize the cohomology groups $H^1(X,\mathbb{Z}_\ell(1))$ and $H_c^1(X,\mathbb{Z}_\ell(1))$ respectively.  To do this, choose a compactification $j: X \hookrightarrow \overline{X}$, and let $i: C \hookrightarrow \overline{X}$ be the closed complement.  Then we set
\begin{align*}
M^1(X) = M_{\emptyset,C}^1(\overline{X}) = [\mathbf{Div}_{C_\bullet}^0(\overline{X}_\bullet) \rightarrow \mathbf{Pic}^{0,red}_{\overline{X}_\bullet}]
\end{align*}
and
\begin{align*}
M_c^1(X) = M_{C,\emptyset}^1(\overline{X}) = [0 \rightarrow \mathbf{Pic}^{0,red}_{\overline{X}_\bullet,C_\bullet}],
\end{align*}
where $\pi_\bullet: \overline{X}_\bullet \rightarrow \overline{X}$ is any proper smooth simplicial hypercover of $\overline{X}$, and $C_\bullet = \pi_\bullet^{-1}(C)$.  
\begin{prop}\label{dimension1motives}
The above definitions are independent (up to unique isomorphism) of the choice of the compactification $\overline{X}$.  They define contravariant functors
\begin{equation*}
M^1(-), M_c^1(-): Sch/k \longrightarrow \mathscr{M}^1(k)[1/p],
\end{equation*}
where $M^1(-)$ is functorial for arbitrary morphisms, while $M_c^1(-)$ is functorial for \emph{proper} morphisms.
\end{prop}
\begin{proof}
The proof that $M^1(X)$ and $M_c^1(X)$ are independent of the choice of compactification is essentially the same argument as in \ref{610}; any two compactifications are dominated by a third, and the induced morphisms of 1-motives are seen to be isomorphisms by looking at $\ell$-adic realizations.  It is clear that $M^1(-)$ is functorial for arbitrary morphisms.  To check that $M_c^1(-)$ is functorial for proper morphisms, it suffices to show the following lemma:
\begin{lem}\label{cartesianstuff}
Let $f: X \rightarrow Y$ be a proper morphisms of schemes over $k$, and choose compactifications $j: X \hookrightarrow \overline{X}$, $v: Y \hookrightarrow \overline{Y}$, and suppose that there exists a morphism $\overline{f}: \overline{X} \rightarrow \overline{Y}$ compactifying $f$.  Let $C = \overline{X} - X$, and let $D = \overline{Y} - Y$.  Then $\overline{f}^{-1}(D) = C$.  
\end{lem}
\begin{proof}(of lemma) It is equivalent to show that $X = \overline{f}^{-1}(Y)$.  We may assume that $X$, $Y$, $\overline{X}$, and $\overline{Y}$ are connected.   Let $Z = \overline{f}^{-1}(Y)$, which is an open subset of $\overline{X}$ containing $X$. Considering the commutative diagrams
\begin{equation*}
\xymatrix{
X \ar[r]^\alpha \ar[rd]_k & Z \ar[d]^{p_1} & X \ar[r]^\alpha \ar[rd]_f & Z \ar[d]^{p_2} \\
& \overline{X}, & & Y,}
\end{equation*}
we get that $\alpha$ is an open immersion by the first diagram, and that $\alpha$ is proper by the second diagram (since $k$ and $p_1$ are open immersions, while $f$ and $p_2$ are proper).  Therefore $\alpha$ is an isomorphism onto a connected component of $Z$.  Because $X$ and $\overline{X}$ were assumed to be connected and $X \hookrightarrow \overline{X}$ is dense, it follows that $Z$ is connected as well; therefore $\alpha: X \rightarrow Z$ is an isomorphism.
\end{proof}
This completes the proof of Prop. \ref{dimension1motives}.
\end{proof}
\subsection{An alternate construction of $M_{D,E}^1(\overline{X})$ for smooth stacks}
In the next section it will be useful to have a construction of $M_{D,E}^1(\overline{X})$ in case $\overline{X}$ is smooth which does not use simplicial schemes.  More generally, we will need this when $\overline{X} = \overline{\mathfrak{X}}$ is a smooth proper Deligne-Mumford stack, and $\mathcal{E}$ and $\mathcal{D}$ are disjoint closed substacks of $\overline{\mathfrak{X}}$ (not necessarily normal crossings divisors).

Let $\overline{\mathfrak{X}}$ be a smooth proper Deligne-Mumford stack, and $\mathcal{D}$ and $\mathcal{E}$ disjoint closed subschemes of $\mathfrak{X}$.  Consider the relative Picard group of the pair $(\overline{\mathfrak{X}},\mathcal{D})$, defined by the formula
\begin{equation*}
\mathrm{Pic}(\overline{\mathfrak{X}},\mathcal{D}) = H^1(\overline{\mathfrak{X}},\mathrm{Ker}(\mathbb{G}_{m,\overline{\mathfrak{X}}} \rightarrow i_*\mathbb{G}_{m,\mathcal{D}})),
\end{equation*}
where $i: \mathcal{D} \hookrightarrow \overline{\mathfrak{X}}$ is the inclusion.  The elements of $\mathrm{Pic}(\overline{\mathfrak{X}},\mathcal{D})$ are pairs $(\mathscr{L},\varphi)$, where $\mathscr{L}$ is a line bundle on $\overline{\mathfrak{X}}$ and $\varphi: \mathcal{O}_\mathcal{D} \stackrel{\sim}{\longrightarrow} \mathscr{L}\vert_{\mathcal{D}}$ is an isomorphism.  We have the following analogue of Proposition \ref{simprelpic}:
\begin{prop}\label{stackrelpic}
Let $\mathbf{Pic}_{\overline{\mathfrak{X}},\mathcal{D}}$ be the fppf-sheafification of the functor on $Sch/k$,
\begin{equation*}
Y \mapsto \mathrm{Pic}(\overline{\mathfrak{X}} \times Y,\mathcal{D} \times Y).
\end{equation*}
Then $\mathbf{Pic}_{\overline{\mathfrak{X}},\mathcal{D}}$ is representable.  Moreover, let $\mathbf{Pic}_{\overline{\mathfrak{X}},\mathcal{D}}^{0,red}$ be the reduced connected component of the identity.  Then $\mathbf{Pic}_{\overline{\mathfrak{X}},\mathcal{D}}^{0,red}$ is a semi-abelian variety, and there is an exact sequence
\begin{equation*}
0 \rightarrow \mathbf{Pic}_{\overline{\mathfrak{X}},\mathcal{D}}^{0,red} \rightarrow \mathbf{Pic}_{\overline{\mathfrak{X}},\mathcal{D}}^{red} \rightarrow \mathbf{NS}_{\overline{\mathfrak{X}},\mathcal{D}} \rightarrow 0,
\end{equation*}
where $\mathbf{NS}_{\overline{\mathfrak{X}},\mathcal{D}}$ is a finitely generated \'{e}tale-locally constant $k$-group scheme.
\end{prop}
\begin{proof}
The proof of Proposition \ref{simprelpic} goes through word-for-word, replacing $\overline{X}_\bullet$ by $\overline{\mathfrak{X}}$ and $D_\bullet$ by $\mathcal{D}$, and using Theorem \ref{stackrepres} to show that $\mathbf{Pic}_{\overline{\mathfrak{X}}}$ and $\mathbf{Pic}_{\mathcal{D}}$ are representable.
\end{proof}
We then define the group $\mathrm{Div}_{\mathcal{E}}(\overline{\mathfrak{X}})$ to be the group of Weil divisors on $\overline{\mathfrak{X}}$ supported on $\mathcal{E}$.  Because $\mathcal{E}$ and $\mathcal{D}$ are disjoint, there is a natural map
\begin{equation*}
cl: \mathrm{Div}_{\mathcal{E}}(\overline{\mathfrak{X}}) \longrightarrow \mathrm{Pic}(\overline{\mathfrak{X}},\mathcal{D}).
\end{equation*}
Let $\mathrm{Div}^0_{\mathcal{E}}(\overline{\mathfrak{X}})$ be the subgroup of $\mathrm{Div}_{\mathcal{E}}(\overline{\mathfrak{X}})$ mapping to 0 in $NS(\overline{\mathfrak{X}},\mathcal{D})$, and let $\mathbf{Div}^0_{\mathcal{E}}(\overline{\mathfrak{X}})$ be the natural extension of $\mathrm{Div}^0_{\mathcal{E}}(\overline{\mathfrak{X}})$ to an \'{e}tale group scheme over $k$.
\begin{defn}
With the pair $(\overline{\mathfrak{X}},\mathcal{D},\mathcal{E})$ as above, we define
\begin{equation*}
M_{\mathcal{D},\mathcal{E}}^1(\overline{\mathfrak{X}}) := [\mathbf{Div}^0_{\mathcal{E}}(\overline{\mathfrak{X}}) \rightarrow \mathbf{Pic}^{0,red}_{\overline{\mathfrak{X}},\mathcal{D}}].
\end{equation*}
\end{defn}
\begin{prop}\label{stackdim1}
With $M_{\mathcal{D},\mathcal{E}}^1(\overline{\mathfrak{X}})$ defined as above, there is a natural isomorphism (for $\ell \not= p$)
\begin{equation*}
T_\ell M_{\mathcal{D},\mathcal{E}}^1(\overline{\mathfrak{X}}) \cong H_{\mathcal{D},\mathcal{E}}^1(\overline{\mathfrak{X}},\mathbb{Z}_\ell(1)) := H^1(\mathfrak{X},j_!\mathbb{Z}_\ell(1)),
\end{equation*}
where $\mathfrak{X} := \overline{\mathfrak{X}} - \mathcal{E}$, $\mathcal{U} := \overline{\mathfrak{X}} - (\mathcal{D} \cup \mathcal{E})$, and $j: \mathcal{U} \hookrightarrow \mathcal{X}$ is the inclusion.
\end{prop}
\begin{proof}
The proof is entirely analogous to that of Proposition \ref{mdexelladic}.  
\end{proof}

\section{Construction of $M_c^{2d-1}(X)$}

Let $X$ be a $d$-dimensional separated finite type $k$-scheme.  In this section we define a 1-motive $M_c^{2d-1}(X)$ realizing the cohomology group $H_c^{2d-1}(X_{\overline{k}},\mathbb{Q}_{\ell}(d))$.  Because of restrictions related to resolution of singularities, we make the following assumption:
\begin{assumption}
Let the base field $k$ be algebraically closed.
\end{assumption}
In this case, we have the following:
\begin{thm}\label{resolution}
There exists a sequence of maps
\begin{equation*}
\mathfrak{X} \stackrel{p}{\longrightarrow} X'' \stackrel{q}{\longrightarrow} X' \stackrel{r}{\longrightarrow} X,
\end{equation*}
satisfying the following conditions:
\begin{enumerate}
\item $r$ is purely inseparable and surjective, therefore a universal homeomorphism;
\item $q$ is proper and birational;
\item $\mathfrak{X}$ is a smooth Deligne-Mumford stack (in fact a global quotient $[U/G]$ of a smooth $k$-scheme $U$ by a finite group $G$) and $p$ identifies $X''$ with the coarse moduli space of $\mathfrak{X}$.
\end{enumerate}
\end{thm}
\begin{proof}
This is \cite[7.4]{dejong}.
\end{proof}
\begin{defn}\label{resdefn}
Let $X$ be a separated finite type $k$-scheme, and let $\pi: \mathfrak{X} \rightarrow X$  be a map from a smooth proper Deligne-Mumford stack $\mathfrak{X}$ which factors as in Theoreom \ref{resolution}.  Then we call $\pi: \mathfrak{X} \rightarrow X$ a \emph{resolution} of $X$.
\end{defn}
The main fact we will use about such resolutions is the following:
\begin{prop}\label{ratcohthesame}
There exists an open dense subscheme $U \subset X$ such that $\pi\vert_U: \pi^{-1}(U) \rightarrow U$ induces an isomorphism $$\mathbb{Q}_{\ell,U} \stackrel{\sim}{\longrightarrow} R{\pi_U}_*\mathbb{Q}_{\ell,\pi^{-1}(U)}$$ in $D_c^b(U,\mathbb{Q}_\ell)$.
\end{prop}
\begin{proof}
This follows from the following two facts:
\begin{enumerate}
\item The morphism $r: X' \rightarrow X$ induces an isomorphism $\mathbb{Z}_{\ell,X} \stackrel{\sim}{\longrightarrow} Rr_*\mathbb{Z}_{\ell,X'}$.
\item The morphism $p: \mathfrak{X} \rightarrow X''$ induces an isomorphism $\mathbb{Q}_{\ell,X''} \stackrel{\sim}{\longrightarrow} Rp_*\mathbb{Q}_{\ell,\mathfrak{X}}$.
\end{enumerate}
The first fact is well-known \cite[3.12]{freitag}, while the second is Lemma \ref{cms}.
\end{proof}

\begin{pg}\label{setup}
We now construct the 1-motive $M_c^{2d-1}(X)$ as follows.  Choose a compactification $\alpha: X \hookrightarrow \overline{X}$, and a resolution $\overline{\pi}: \overline{\mathfrak{X}} \rightarrow \overline{X}$ in the sense of Definition \ref{resdefn}.  Therefore $\overline{\mathfrak{X}}$ is a smooth proper Deligne-Mumford stack and there exists an open dense subscheme $U \subset X$ such that $\overline{\pi}^{-1}(U) \rightarrow U$ is purely inseparable (in particular, $U$ satisfies the conclusion of Proposition \ref{ratcohthesame}).  We write $\mathcal{U}$ for $\overline{\pi}^{-1}(U)$, and let $\mathfrak{X} = \overline{\mathfrak{X}} \times_{\overline{X}} X$, $C = \overline{X} - X$, and $\mathcal{C}= \overline{\mathfrak{X}} - \mathfrak{X}$. Finally, let $Z = X - U$ and $\mathcal{Z} = \mathfrak{X} - \mathcal{U}$.  Summarizing, we have commutative diagrams  
\begin{equation*}
\xymatrix{
{\mathfrak{X}} \ar@{^{(}->}[r]^{\alpha'} \ar[d]^{\pi} & \overline{\mathfrak{X}} \ar[d]^{\overline{\pi}} & \mathcal{C} \ar@{_{(}->}[l]_{\beta'} \ar[d]^{\pi_C} \\
X \ar@{^{(}->}[r]^{\alpha} & \overline{X} & C \ar@{_{(}->}[l]_{\beta}}
\end{equation*}
and
\begin{equation*}
\xymatrix{
\mathcal{U} \ar@{^{(}->}[r]^{j'} \ar[d]^{\pi_U} & \mathfrak{X} \ar[d]^{\pi} & \mathcal{Z} \ar@{_{(}->}[l]_{i'} \ar[d]^{\pi_Z} \\
U \ar@{^{(}->}[r]^{j} & X & Z. \ar@{_{(}->}[l]_{i}}
\end{equation*}
Let $\overline{Z}$ be the closure of $Z$ in $\overline{X}$, and $\overline{\mathcal{Z}}$ the closure of $\mathcal{Z}$ in $\overline{\mathfrak{X}}$.  Let $\mathrm{Div}_{\mathcal{C} \cup \overline{\mathcal{Z}}}(\overline{\mathfrak{X}})$ be the free abelian group of divisors on $\overline{\mathfrak{X}}$ supported on the closed subscheme $\mathcal{C} \cup \overline{\mathcal{Z}}$.  
\end{pg}
\begin{rem}
Note that it is possible to choose $\overline{\mathfrak{X}}$ so that $\mathcal{C}$ is a (reduced) strict normal crossings divisor \cite[7.4]{dejong}.  However, it is not known whether we can arrange that $\mathcal{Z}$ be a strict normal crossings divisor.  
\end{rem}
\begin{defn}
We define a subgroup $$\mathrm{Div}_{\mathcal{C}\cup \overline{\mathcal{Z}}/\overline{Z}}^0(\overline{\mathfrak{X}}) \subset \mathrm{Div}_{\mathcal{C} \cup \overline{\mathcal{Z}}}(\overline{\mathfrak{X}})$$ to be the divisors $D \in \mathrm{Div}_{\mathcal{C} \cup \overline{\mathcal{Z}}}(\overline{\mathfrak{X}})$ satisfying the following two conditions:
\begin{enumerate}
\item The cycle class $cl(D) = 0$ in $NS(\overline{\mathfrak{X}})$.
\item Write $D$ as $D = D_1 + D_2$ with $D_1$ supported on $\mathcal{C}$ and $D_2$ supported on $\overline{\mathcal{Z}}$.  This decomposition is unique since $\mathcal{C}$ and $\overline{\mathcal{Z}}$ have no codimension-1 irreducible components in common.  Consider the proper pushforward map $$\pi_*: \mathrm{Div}_{\overline{\mathcal{Z}}}(\overline{\mathfrak{X}}) \otimes \mathbb{Q} \longrightarrow \mathrm{Div}_{\overline{Z}}(\overline{X}) \otimes \mathbb{Q}.$$ We then require that $\pi_*D_2 =0$.
\end{enumerate}
\end{defn}
There is a natural map $\mathrm{Div}_{\mathcal{C} \cup \overline{\mathcal{Z}}/\overline{Z}}^0(\overline{\mathfrak{X}}) \rightarrow \mathrm{Pic}^0(\overline{\mathfrak{X}})$ sending a divisor $D$ to its associated line bundle $\mathcal{O}(E)$.  Moreover, this map extends to a map of group schemes
\begin{equation*}
\mathbf{Div}_{\mathcal{C} \cup \overline{\mathcal{Z}}/\overline{Z}}^0(\overline{\mathfrak{X}}) \rightarrow \mathbf{Pic}_{\overline{\mathfrak{X}}}^{0,red},
\end{equation*}
where $\mathbf{Div}_{\mathcal{C} \cup \overline{\mathcal{Z}}/\overline{Z}}^0(\overline{\mathfrak{X}})$ is the \'{e}tale group scheme associated to $\mathrm{Div}_{\mathcal{C} \cup \overline{\mathcal{Z}}/\overline{Z}}^0(\overline{\mathfrak{X}})$.
 
We can now define the 1-motive $M_c^{2d-1}(X)$:
\begin{defn}\label{mc2d-1}
Let $X$ be a separated scheme of finite type, of dimension $d$ over $k$.  Choose a compactification $X \hookrightarrow \overline{X}$ and a resolution $\overline{\mathfrak{X}} \rightarrow \overline{X}$ as above.  We define $M_c^{2d-1}(X)$ to be the 1-motive
\begin{equation*}
M_c^{2d-1}(X) := [\mathbf{Div}^0_{\mathcal{C} \cup \overline{\mathcal{Z}}/\overline{Z}} \rightarrow \mathbf{Pic}^{0,red}_{\overline{\mathfrak{X}}}]^\vee,
\end{equation*}
where $\mathbf{Div}^0_{\mathcal{C} \cup \overline{\mathcal{Z}}/\overline{Z}}$ is the natural extension of $\divCZX$ to an \'{e}tale group scheme, and $(-)^\vee$ indicates taking the Cartier dual of a 1-motive (\ref{cartierduals}).
\end{defn}
We will show that this 1-motive is independent of the choice of $\overline{\mathfrak{X}}$ up to isogeny.
\subsection{$\ell$-adic realization of $M_c^{2d-1}(X)$}
To understand this definition of $M_c^{2d-1}(X)$ (and to show that it is, up to isogeny, independent of the choice of $\overline{X}$ and $\overline{\mathfrak{X}}$), we must discuss the $\ell$-adic realization of $M_c^{2d-1}(X)$.  Choose a prime $\ell \not= p$.  Then in $D_c^b(X,\mathbb{Q}_\ell)$ we have a commuting diagram with exact rows and columns
\begin{equation*}
\begin{CD}
j_!\mathbb{Q}_{\ell,U} @> \sim >> R\pi_*j'_!\mathbb{Q}_{\ell,\mathcal{U}} @>>> 0 @>>> \\
@VVV @VVV @VVV \\
\mathbb{Q}_{\ell,X} @>>> R\pi_*\mathbb{Q}_{\ell,\mathfrak{X}} @>>> A @>>> \\
@VVV @VVV @V \sim VV \\
i_*\mathbb{Q}_{\ell,Z} @>>> R\pi_*i'_*\mathbb{Q}_{\ell,\mathcal{Z}} @>>> i_*i^*A @>>> \\
@VVV @VVV @VVV
\end{CD}
\end{equation*}
The upper left-hand arrow is an isomorphism because it equals the composition
\begin{equation*}
j_!\mathbb{Q}_{\ell,U} \stackrel{\mathrm{ad}}{\longrightarrow} j_!R\pi_*\mathbb{Q}_{\ell,\mathcal{U}} \stackrel{\mathrm{=}}{\longrightarrow} R\pi_*j'_!\mathbb{Q}_{\ell,\mathcal{U}},
\end{equation*}
and the adjoint map $\mathbb{Q}_{\ell,U} \rightarrow R\pi_*\mathbb{Q}_{\ell,\mathcal{U}}$ is an isomorphism in $D_c^b(U,\mathbb{Q}_\ell)$ (Lemma \ref{cms}).
\begin{pg}
If we apply $\alpha_!$ to this diagram and take cohomology of the two left-hand vertical columns, we get a commuting diagram with exact rows
\begin{equation*}
\begin{CD}
H_c^{2d-2}(X,\mathbb{Q}_\ell) @>>> H_c^{2d-2}(Z,\mathbb{Q}_\ell) @>>> H_c^{2d-1}(U,\mathbb{Q}_\ell) @>>> H_c^{2d-1}(X,\mathbb{Q}_\ell) @>>> 0 \\
@VVV @VVV @VV \sim V @VVV \\
H_c^{2d-2}(\mathfrak{X},\mathbb{Q}_{\ell}) @>>> H_c^{2d-2}(\mathcal{Z},\mathbb{Q}_\ell) @>>> H_c^{2d-1}(\mathcal{U},\mathbb{Q}_\ell) @>>> H_c^{2d-1}(\mathfrak{X},\mathbb{Q}_\ell) @>>> 0.
\end{CD}
\end{equation*}
Twisting by $d-1$, taking duals, and applying Poincar\'{e} duality, we get a diagram \\
\begin{equation}\label{bigdiag}
\begin{CD}
0 @>>> H_c^{2d-1}(X,\mathbb{Q}_\ell(d\textrm{-1}))^\vee @>>> H_c^{2d-1}(U,\mathbb{Q}_\ell(d\textrm{-1}))^\vee @>>> \mathbb{Q}_\ell^{I_{d-1}(Z)} @>>> H_c^{2d-2}(X,\mathbb{Q}_\ell(d\textrm{-1}))^\vee \\
@. @AAA @AA \sim A @AAA @AAA \\
0 @>>> H^1(\mathfrak{X},\mathbb{Q}_\ell(1)) @>>> H^1(\mathcal{U},\mathbb{Q}_\ell(1)) @>>> \mathbb{Q}_{\ell}^{I_{d-1}(\mathcal{Z})} @>>> H^2(\mathfrak{X},\mathbb{Q}_\ell(1)).
\end{CD}
\end{equation}
(The isomorphisms $H_c^{2d-2}(\mathcal{Z},\mathbb{Q}_\ell(d-1))^\vee \cong \mathbb{Q}_\ell^{I_{d-1}(\mathcal{Z})}$ and $H_c^{2d-2}(Z,\mathbb{Q}_\ell(d-1))^{\vee} \cong \mathbb{Q}_\ell^{I_{d-1}(Z)}$ are induced by sums of trace maps as in \ref{tracemap}).  Recall that $I_{d-1}(Z)$ is the set of $(d-1)$-dimensional irreducible components of $Z$, so we have $$\mathbb{Q}_{\ell}^{I_{d-1}(Z)} = \mathrm{Div}_Z(X) \otimes \mathbb{Q}_\ell = \mathrm{Div}_{\overline{Z}}(\overline{X}) \otimes \mathbb{Q}_\ell.$$  Similarly, we have $$\mathbb{Q}_\ell^{I_{d-1}(\mathcal{Z})} = \mathrm{Div}_{\mathcal{Z}}(\mathfrak{X}) \otimes \mathbb{Q}_\ell = \mathrm{Div}_{\overline{\mathcal{Z}}}(\overline{\mathfrak{X}}) \otimes \mathbb{Q}_\ell.$$ Then the map $\mathbb{Q}_\ell^{I_{d-1}(\mathcal{Z})} \rightarrow \mathbb{Q}_\ell^{I_{d-1}(Z)}$ of diagram \ref{bigdiag} is induced by proper pushforward $\pi_*: \mathrm{Div}_{\overline{\mathcal{Z}}}(\overline{\mathfrak{X}}) \otimes \mathbb{Q} \rightarrow \mathrm{Div}_{\overline{Z}}(\overline{X}) \otimes \mathbb{Q}$, while the map $\mathbb{Q}_\ell^{I_{d-1}(\mathcal{Z})} \rightarrow H^2(\mathfrak{X},\mathbb{Q}_\ell(1))$ is induced by the divisor class map $\mathrm{Div}_{\mathcal{Z}}(\mathfrak{X}) \rightarrow NS(\mathfrak{X})$ (Proposition \ref{divisors}).  Then diagram \ref{bigdiag} shows that we have 
\begin{equation*}
H_c^{2d-1}(X,\mathbb{Q}_\ell(d-1))^\vee \cong \mathrm{Ker}(H^1(\mathcal{U},\mathbb{Q}_\ell(1)) \rightarrow \mathbb{Q}_\ell^{I_{d-1}(\mathcal{Z})} \stackrel{\pi_*}{\rightarrow} \mathbb{Q}_\ell^{I_{d-1}(Z)}).
\end{equation*}
\end{pg}
\begin{pg}\label{setuptocomp}
Now recall that $\mathcal{U} = \overline{\mathfrak{X}} - (\mathcal{C} \cup \overline{\mathcal{Z}})$.  By Proposition \ref{stackdim1} (with $\mathcal{D} := \emptyset$, $\mathcal{E} := \mathcal{C} \cup \mathcal{Z}$), if we set
\begin{equation*}
M^1(\mathcal{U}) := [\mathbf{Div}_{\mathcal{C} \cup \overline{\mathcal{Z}}}^0(\overline{\mathfrak{X}}) \rightarrow \mathbf{Pic}^0(\overline{\mathfrak{X}})],
\end{equation*}
we have a canonical isomorphism $V_\ell M^1(\mathcal{U}) \cong H^1(\mathcal{U},\mathbb{Q}_\ell(1))$.  Then from Definition \ref{mc2d-1} we have
\begin{equation*}
M_c^{2d-1}(\mathfrak{X})^\vee = \mathrm{Ker}(M^1(\mathcal{U}) \longrightarrow [\mathbf{Div}_{\overline{Z}}(\overline{X}) \rightarrow 0]),
\end{equation*}
where the map is induced by proper pushforward of divisors.  Applying the functor $V_\ell(-)$, we get a commuting diagram
\begin{equation*}
\begin{CD}
0 @>>> V_\ell(M_c^{2d-1}(X)^\vee) @>>> V_\ell(M^1(\mathcal{U})) @>>> \mathrm{Div}_{\overline{Z}}(\overline{X}) \otimes \mathbb{Q}_\ell \\
@. @VVV @VV \sim V @VV \sim V \\
0 @>>> H_c^{2d-1}(X,\mathbb{Q}_\ell(d-1))^\vee @>>> H^1(\mathcal{U},\mathbb{Q}_\ell(1)) @>>> \mathbb{Q}_\ell^{I_{d-1}(Z)}.
\end{CD}
\end{equation*}
By the five lemma, the map on the left is also an isomorphism.  Taking Cartier duals, we have the following:
\end{pg}
\begin{prop}\label{comparison2d-1}
For $X$ a separated scheme of finite type over $k$, of dimension $d$, set $M_c^{2d-1}(X) = [\divCZX \rightarrow \mathrm{Pic}^0(\overline{\mathfrak{X}})]^\vee$.  Then we have a canonical isomorphism
\begin{equation*}
V_\ell M_c^{2d-1}(X) \cong H_c^{2d-1}(X,\mathbb{Q}_\ell(d)).
\end{equation*}
\end{prop}

\subsection{Functoriality of $M_c^{2d-1}(X)$} We will prove that the 1-motive $M_c^{2d-1}(X)$ is contravariantly functorial for proper morphisms $f: X \rightarrow Y$ between varieties $X$ and $Y$ of equal dimension $d$.  Before this, we must prove a preliminary fact on functoriality for the resolutions of Theorem \ref{resolution}:
\begin{prop}\label{resolvemap}
Let $f: X \rightarrow Y$ be a morphism of separated finite type $k$-schemes, with $k = \overline{k}$.  Then there exist resolutions (in the sense of Definition \ref{resdefn}) $\pi: \mathfrak{X} \rightarrow X$, $\sigma: \mathcal{Y} \rightarrow Y$ with $\mathfrak{X}$ and $\mathcal{Y}$ smooth, and a \emph{representable} map $f': \mathfrak{X} \rightarrow \mathcal{Y}$ making the diagram
\begin{equation*}
\begin{CD}
\mathfrak{X} @> f' >> \mathcal{Y} \\
@V \pi VV @V \sigma VV \\
X @> f >> Y
\end{CD}
\end{equation*}
commute.  
\end{prop}
\begin{proof}
By Theorem \ref{resolution}, we can choose a resolution $\sigma: \mathcal{Y} \rightarrow Y$ with $\mathcal{Y} = [V/H]$.  Then let $$\mathfrak{X}_1 = \mathcal{Y} \times_Y X = [(V \times_Y X) /H].$$  By \cite[Thm. 7.3]{dejong}, there exists a quotient stack $\mathfrak{X} = [U/G]$ with $U$ smooth and $G$ finite, together with a proper map $\phi: \mathfrak{X} \rightarrow \mathfrak{X}_1$ such that the composition $\mathfrak{X} \rightarrow \mathfrak{X}_1 \rightarrow X$ is a resolution.  The induced map $f': \mathfrak{X} \rightarrow \mathcal{Y}$ satisfies all the conditions of Proposition \ref{resolvemap} except representability.  To make $f'$ representable, we replace $\mathcal{Y}$ by $\mathcal{Y} \times_k BG$ (note that $\mathcal{Y} \times BG$ has the same coarse moduli space as $\mathcal{Y}$, so it is still a resolution of $Y$).  
\end{proof}
\begin{pg}\label{2d-1functoriality}
Now let $X$ and $Y$ be separated finite type $k$-schemes of dimension $d$, and $f: X \rightarrow Y$ a proper morphism.  We assume $f$ is surjective; if $f$ is not surjective, we simply define $f^*: M_c^{2d-1}(Y) \rightarrow M_c^{2d-1}(X)$ to be the zero map.  Note that since $\mathrm{dim}(X) = \mathrm{dim}(Y)$, we have that $f$ is generically finite flat.  Choose compactifications $k: X \hookrightarrow \overline{X}$, $j: Y \hookrightarrow \overline{Y}$, and a map $\overline{f}: \overline{X} \rightarrow \overline{Y}$ such that the diagram
\begin{equation*}
\xymatrix{
X \ar@{^{(}->}[r]^{k} \ar[d]^{f} & \overline{X} \ar[d]^{\overline{f}} \\
Y \ar@{^{(}->}[r]^{j} & \overline{Y}
}
\end{equation*}
is cartesian (the diagram is cartesian because $f$ is proper; see Lemma \ref{cartesianstuff}).  Then using Proposition \ref{resolvemap}, choose resolutions $\overline{\pi}: \overline{\mathfrak{X}} \rightarrow \overline{X}$ and $\overline{\sigma}: \overline{\mathcal{Y}} \rightarrow \overline{Y}$ and a representable map $\overline{f}': \overline{\mathfrak{X}} \rightarrow \overline{\mathcal{Y}}$ lying over $\overline{f}$.  If we let $\mathfrak{X} = \overline{\mathfrak{X}} \times_{\overline{X}} X$, $\mathcal{Y} = \overline{\mathcal{Y}} \times_{\overline{Y}} Y$, $\mathcal{C} = \overline{\mathfrak{X}} - \mathfrak{X}$ and $\mathcal{D} = \overline{\mathcal{Y}} - \mathcal{Y}$, we then have a diagram
\begin{equation*}
\xymatrix{
\mathfrak{X} \ar@{^{(}->}[r]^{k'} \ar[d]^{f'} & \overline{\mathfrak{X}} \ar[d]^{\overline{f}'} & \mathcal{C} \ar@{_{(}->}[l]_{} \ar[d] \\
\mathcal{Y} \ar@{^{(}->}[r]^{j'} & \overline{\mathcal{Y}} & \mathcal{D}. \ar@{_{(}->}[l]}
\end{equation*}
Let $V \subset Y$ be an open subset of $Y$ such that $\overline{\sigma}^{-1}(V) \rightarrow V$ is purely inseparable, and define $\mathcal{V} := \overline{\sigma}^{-1}(V)$.  By possibly shrinking $V$, we can arrange that $$\mathfrak{X} \times_{\mathcal{Y}} \mathcal{V} \longrightarrow X \times_Y V$$ is purely inseparable, since $\mathfrak{X} \rightarrow X$ is purely inseparable on an open dense subset of $X$.  If we set $\mathcal{U} = \mathfrak{X} \times_{\mathcal{Y}} \mathcal{V}$ and $U = X \times_Y V$, and $\mathcal{Z} = \mathfrak{X} - \mathcal{U}$, $\mathcal{W} = \mathcal{Y} - \mathcal{V}$, we then have a commuting diagram
\begin{equation*}
\xymatrix{
\mathcal{U} \ar@{^{(}->}[r]^{a'} \ar[d]^{f'} & {\mathfrak{X}} \ar[d]^{{f}'} & \mathcal{Z} \ar@{_{(}->}[l]_{} \ar[d] \\
\mathcal{V} \ar@{^{(}->}[r]^{b'} & {\mathcal{Y}} & \mathcal{W} \ar@{_{(}->}[l]}
\end{equation*}
where $a': \mathcal{U} \hookrightarrow \mathfrak{X}$ and $b': \mathcal{V} \hookrightarrow \mathcal{Y}$ are open immersions $\mathcal{Z} \hookrightarrow \mathfrak{X}$ and $\mathcal{W} \hookrightarrow \mathcal{Y}$ are closed immersions, and $\overline{\pi}$ and $\overline{\sigma}$ restrict to pursely inseparable maps $\mathcal{U} \rightarrow U$ and $\mathcal{V} \rightarrow V$ respectively.  Let $\overline{\mathcal{Z}}$ and $\overline{\mathcal{W}}$ be the closures of $\mathcal{Z}$ and $\mathcal{W}$ in $\overline{\mathfrak{X}}$ and $\overline{\mathcal{Y}}$, respectively.  We then can set
\begin{align*}
M_c^{2d-1}(X) &= [\mathbf{Div}^0_{\mathcal{C} \cup \overline{\mathcal{Z}}/\overline{Z}}(\overline{\mathfrak{X}}) \rightarrow \mathbf{Pic}^{0,red}_{\overline{\mathfrak{X}}}]^\vee \ \ \mathrm{and} \\
M_c^{2d-1}(Y) &= [\mathbf{Div}^0_{\mathcal{D} \cup \overline{\mathcal{W}}/\overline{W}}(\overline{\mathcal{Y}}) \rightarrow \mathbf{Pic}^{0,red}_{\overline{\mathcal{Y}}}]^\vee.
\end{align*}
(Note that we have yet to show that $M_c^{2d-1}(X)$ and $M_c^{2d-1}(Y)$ are independent of choice of compactification.)  In order to define a morphism of 1-motives $\hat{f}^*: M_c^{2d-1}(Y) \rightarrow M_c^{2d-1}(X)$, we would like to define a covariant morphism of 1-motives
\begin{equation}\label{map}
\hat{f}_*: [\mathrm{Div}^0_{\mathcal{C} \cup \overline{\mathcal{Z}}/\overline{Z}}(\overline{\mathfrak{X}}) \rightarrow \mathbf{Pic}^{0,red}_{\overline{\mathfrak{X}}}] \longrightarrow [\mathrm{Div}^0_{\mathcal{D} \cup \overline{\mathcal{W}}/\overline{W}}(\overline{\mathcal{Y}}) \rightarrow \mathbf{Pic}^{0,red}_{\overline{\mathcal{Y}}}].
\end{equation}
If $\mathrm{dim}(f(X)) < d$, then we simply define this to be the 0 map.
\end{pg}
\begin{pg}
In the non-trivial case when $\mathrm{dim}(f(X)) = d$, we define $\hat{f}_*$ via proper pushforwards.  Because $\overline{f}'$ is representable, there is an integrally defined pushforward on cycle classes $\overline{f}'_*: \mathrm{Div}(\overline{\mathfrak{X}}) \rightarrow \mathrm{Div}(\overline{\mathcal{Y}})$.  Since taking the associated line bundle of a cycle class commutes with proper pushforward, this restricts to a map
\begin{equation*}
\overline{f}'_*: \mathrm{Div}^0_{\mathcal{C} \cup \overline{\mathcal{Z}}/\overline{Z}}(\overline{\mathfrak{X}}) \longrightarrow \mathrm{Div}^0_{\mathcal{D} \cup \overline{\mathcal{W}}/\overline{W}}(\overline{\mathcal{Y}})
\end{equation*}
which forms the lattice part of the required map \ref{map} of 1-motives.  It remains to define a pushforward map $\overline{f}'_*: \mathbf{Pic}^{0,red}_{\overline{\mathfrak{X}}} \rightarrow \mathbf{Pic}^{0,red}_{\overline{\mathcal{Y}}}$.  Since these are smooth group schemes, to define such a map it suffices to define a map $$\mathbf{Pic}^{0,red}_{\overline{\mathfrak{X}}}(A) \rightarrow \mathbf{Pic}^{0,red}_{\overline{\mathcal{Y}}}(A)$$ functorially in $A$, for \emph{smooth} $A$ (by the Yoneda lemma).  On the level of presheaves, this is just a map $$\mathrm{Pic}^0(\overline{\mathfrak{X}} \times A) \rightarrow \mathrm{Pic}^0(\overline{\mathcal{Y}} \times A)$$ which is functorial in $A$.  By Proposition \ref{cartpic}, every element of $\mathrm{Pic}^0(\overline{\mathfrak{X}} \times A)$ is represented by a divisor (since $\mathrm{Pic}^0(\overline{\mathfrak{X}} \times A)/\mathrm{Div}^0(\overline{\mathfrak{X}} \times A)$ is both finite and divisible).  Therefore we may use proper pushforward of divisors to define the desired map $\mathrm{Pic}^0(\overline{\mathfrak{X}} \times A) \rightarrow \mathrm{Pic}^0(\overline{\mathcal{Y}} \times A)$.  Thus we have defined the required map (\ref{map}).
\end{pg}
\begin{pg}
Let $\hat{f}^*: M_c^{2d-1}(Y) \rightarrow M_c^{2d-1}(X)$ be the morphism of 1-motives obtained by taking the Cartier dual of the map $\hat{f}_*$ of \ref{map}.  We wish to show that $V_\ell \hat{f}^*: V_\ell M_c^{2d-1}(Y) \rightarrow V_\ell M_c^{2d-1}(X)$ agrees with the pullback map on cohomology $f^*: H_c^{2d-1}(Y,\mathbb{Q}_\ell(d)) \rightarrow H_c^{2d-1}(X,\mathbb{Q}_\ell(d))$.  More precisely, we claim the following:
\end{pg}
\begin{prop}\label{prop716}
There is a commutative diagram
\begin{equation*}
\begin{CD}
V_\ell M_c^{2d-1}(Y) @> \alpha_Y >> H_c^{2d-1}(Y,\mathbb{Q}_\ell(d)) \\
@V {V_\ell \hat{f}^*} VV @V f^* VV \\
V_\ell M_c^{2d-1}(X) @> \alpha_X >> H_c^{2d-1}(X,\mathbb{Q}_\ell(d))
\end{CD}
\end{equation*}
where $\alpha_X$ and $\alpha_Y$ are the comparison isomorphisms of Proposition \ref{comparison2d-1} and $f^*$, $\hat{f}^*$ are as defined above.
\end{prop}
\begin{proof}
We continue with the notation of \ref{2d-1functoriality}.  If we define
\begin{equation*}
M^1(\mathcal{U}) := [\mathbf{Div}^0_{\mathcal{C} \cup \overline{\mathcal{Z}}} \rightarrow \mathbf{Pic}^0(\overline{\mathfrak{X}})]
\end{equation*}
and
\begin{equation*}
M^1(\mathcal{V}) := [\mathbf{Div}^0_{\mathcal{D} \cup \overline{\mathcal{W}}} \rightarrow \mathbf{Pic}^0(\overline{\mathcal{Y}})].
\end{equation*}
then by Proposition \ref{stackdim1} we have natural isomorphisms $V_\ell M^1(\mathcal{U}) \cong H^1(\mathcal{U},\mathbb{Q}_\ell(1))$ and $V_\ell M^1(\mathcal{V}) \cong H^1(\mathcal{V},\mathbb{Q}_\ell(1))$.  It is clear that we can define a map
\begin{equation*}
\hat{f}_*: M^1(\mathcal{U}) \rightarrow M^1(\mathcal{V})
\end{equation*}
by proper pushforward of divisors in the same way as we defined $\hat{f}^*: M_c^{2d-1}(Y) \rightarrow M_c^{2d-1}(Y)$.  Since we have commutative diagrams
\begin{equation*}
\xymatrix{
V_\ell(M_c^{2d-1}(X)^\vee) \ar@{^{(}->}[r] \ar[d]^{\alpha_X} & V_\ell(M^1(\mathcal{U})) \ar[d]^{\sim} \\
H_c^{2d-1}(X,\mathbb{Q}_\ell(d-1))^\vee \ar@{^{(}->}[r] & H^1(\mathcal{U},\mathbb{Q}_\ell(1))
} \ \ \mathrm{and} \ \ 
\xymatrix{
V_\ell(M_c^{2d-1}(Y)^\vee) \ar@{^{(}->}[r] \ar[d]^{\alpha_Y} & V_\ell(M^1(\mathcal{V})) \ar[d]^{\sim} \\
H_c^{2d-1}(Y,\mathbb{Q}_\ell(d-1))^\vee \ar@{^{(}->}[r] & H^1(\mathcal{V},\mathbb{Q}_\ell(1))
}
\end{equation*}
(see \ref{setuptocomp}), to show that we have a commuting diagram as in Proposition \ref{prop716} it suffices to show the following:

\begin{prop}\label{compatible3}
There is a commutative diagram
\begin{equation*}
\begin{CD}
V_\ell M^1(\mathcal{U}) @> \alpha_U >> H^1(\mathcal{U},\mathbb{Q}_\ell(1)) \\
@V {V_\ell \hat{f}_*} VV  @V f_* VV \\
V_\ell M^1(\mathcal{V}) @> \alpha_V >> H^1(\mathcal{V},\mathbb{Q}_\ell(1))
\end{CD}
\end{equation*}
where $\alpha_U$ and $\alpha_V$ are the comparison isomorphisms of Proposition \ref{stackdim1}, and $f_*$ is the pushforward map on cohomology, i.e., the Poincar\'{e} dual to the map $$f^*: H_c^{2d-1}(\mathcal{V},\mathbb{Q}_\ell(d-1)) \rightarrow H_c^{2d-1}(\mathcal{U},\mathbb{Q}_\ell(d-1)).$$
\end{prop}
\begin{proof}

 First consider the case where $\mathrm{dim}f(\mathcal{U}) < \mathrm{dim}f(\mathcal{V})$.  Then the proper pushforward on cohomology is clearly 0 since the map on $f^*$ on $H_c^{2d-1}$ must be 0.  On the other hand, the proper pushforward map $V_\ell \hat{f}_*: V_\ell M^1(\mathcal{U}) \rightarrow V_\ell M^1(\mathcal{V})$ is also 0 by definition.

Now assume $f: \mathcal{U} \rightarrow \mathcal{V}$ is finite and flat.  Then the pushforward map on cohomology is induced by a trace map
\begin{equation*}
tr_f: f_* f^*\mu_n \rightarrow \mu_n
\end{equation*}
\cite[Thm 4.1]{fujiwara}.  This trace map is shown in loc. cit. to be compatible with \'{e}tale localization, and to agree with the usual trace map in the case when $\mathcal{U}$ and $\mathcal{V}$ are schemes.  To check Proposition \ref{compatible3} we may work \'{e}tale-locally on $\mathcal{V}$ and hence may assume $\mathcal{V}$ is a scheme; since $f$ is representable, $\mathcal{U}$ is a scheme as well.  In this case, the proposition follows because the trace map agrees with the norm map on invertible sections \cite[p. 136]{freitag}, and it is clear that the norm map on invertible sections induces the proper pushforward on divisors.

In the general case $f$ is generically finite flat (since it is proper and representable and $\mathrm{dim}(\mathcal{U}) = \mathrm{dim}(\mathcal{V})$, $\mathrm{dim}f(\mathcal{U}) = \mathrm{dim}f(\mathcal{V})$).  Let $\mathcal{V}' \subset \mathcal{V}$, $\mathcal{U}':= \mathcal{U} \times_{\mathcal{V}} \mathcal{V}' \subset \mathcal{U}$ be open substacks such that $f: \mathcal{U}' \rightarrow \mathcal{V}'$ is finite flat.  Then we have commutative diagrams
\begin{equation*}
\xymatrix{
H^1(\mathcal{U},\mathbb{Q}_\ell(1)) \ar@{^{(}->}[r] \ar[d]^{f_*} & H^1(\mathcal{U}',\mathbb{Q}_\ell(1)) \ar[d]^{f_*} \\
H^1(\mathcal{V},\mathbb{Q}_\ell(1)) \ar@{^{(}->}[r] & H^1(\mathcal{V}',\mathbb{Q}_\ell(1)),} \ \ \ \ \ \ \ \ \ \ 
\xymatrix{
V_\ell M^1(\mathcal{U}) \ar@{^{(}->}[r] \ar[d]^{V_\ell \hat{f}_*} & V_\ell M^1(\mathcal{U}') \ar[d]^{V_\ell \hat{f}_*} \\
V_\ell M^1(\mathcal{V}) \ar@{^{(}->}[r] & V_\ell M^1(\mathcal{V}')}
\end{equation*}
where the left-hand diagram is induced from the restriction maps $\mathcal{U}' \hookrightarrow \mathcal{U}$ etc., and the right-hand diagram is essentially from the definition of the various 1-motives appearing and of the maps $\hat{f}_*$.  Therefore it suffices to prove Proposition \ref{compatible3} for the map $f: \mathcal{U}' \rightarrow \mathcal{V}'$, which is finite flat and hence has already been considered.
\end{proof}
This completes the proof of Proposition \ref{prop716}.
\end{proof}
\begin{pg}
We are finally ready to show that $M_c^{2d-1}(X)$ is independent of compactification (up to isogeny), and therefore induces a functor $$M_c^{2d-1}(-): (Sch_d/k)^{op} \rightarrow \mathscr{M}^1(k) \otimes \mathbb{Q}$$ which is functorial for proper morphisms.  

For a given $X \in Sch_d/k$, suppose that we choose two compactifications $\overline{X}$, $\overline{X}'$ of $X$ and resolutions $\overline{\mathfrak{X}} \rightarrow \overline{X}$, $\overline{\mathfrak{X}}' \rightarrow \overline{X}'$.  Then we aim to show the following:
\end{pg}
\begin{prop}\label{uniqueness7}
Let $M_c^{2d-1}(X)$ and $M_c^{2d-1}(X)'$ be the 1-motives of Definition \ref{mc2d-1} constructed using $\overline{\mathfrak{X}}$ and $\overline{\mathfrak{X}}'$ respectively.  Then there exists a unique isogeny of 1-motives $$f: M_c^{2d-1}(X) \rightarrow M_c^{2d-1}(X)'$$ fitting into a diagram
\begin{equation*}
\begin{CD}
V_\ell M_c^{2d-1}(X) @> \alpha_X >> H_c^{2d-1}(X,\mathbb{Q}_\ell(d)) \\
@V {V_\ell f} VV @V = VV \\
V_\ell M_c^{2d-1}(X)' @> {\alpha_X}' >> H_c^{2d-1}(X,\mathbb{Q}_\ell(d))
\end{CD}
\end{equation*}
for all $\ell \not= p$, where $\alpha_X$ and $\alpha_X'$ are the comparison isomorphisms of Proposition \ref{comparison2d-1}.
\end{prop} 
\begin{proof}
Let $\overline{X}'' = \overline{X} \times_X \overline{X}'$, a third compactification of $X$ which dominates $\overline{X}$ and $\overline{X}'$.  Recall that we can write $\overline{\mathfrak{X}}$ and $\overline{\mathfrak{X}}'$ as global quotient stacks, say $\overline{\mathfrak{X}} = [V/G]$ and $\overline{\mathfrak{X}}' = [V'/G']$.  Then set $$\mathcal{Y}  := \overline{\mathfrak{X}} \times_X \overline{\mathfrak{X}}' = [V \times_X V'/G \times G'].$$  Since this is a global quotient stack, an application of \cite[Thm 7.3]{dejong} gives a smooth proper stack $\overline{\mathfrak{X}}'' \rightarrow \mathcal{Y}$ which is purely inseparable on an open dense substack.  In fact, we can write $\overline{\mathfrak{X}}''$ as a global quotient stack $\overline{\mathcal{X}}'' = [W/H]$ , and then we have commutative diagrams
\begin{equation*}
\begin{CD}
\overline{\mathfrak{X}}'' @> f_1' >> \overline{\mathfrak{X}} \times BH \\
@VVV @VVV \\
\overline{X}'' @> f_1 >> \overline{X}, 
\end{CD} \ \ \ \ \ \ \ \ 
\begin{CD}
\overline{\mathfrak{X}}'' @> f_2 >> \overline{\mathfrak{X}}' \times BH \\
@VVV @VVV \\
\overline{X}'' @> f_2 >> \overline{X}'
\end{CD}
\end{equation*}
where $f_1'$ and $f_2'$ are representable, and $f_1$ and $f_2$ restrict to the identity on $X$.  Note that replacing $\overline{\mathfrak{X}}$ by $\overline{\mathfrak{X}} \times BH$ does not change the 1-motive $M_c^{2d-1}(X)$ constructed from $\overline{\mathfrak{X}}$ (similarly, replacing $\overline{\mathfrak{X}}'$ by $\overline{\mathfrak{X}}' \times BH$ leaves $M_c^{2d-1}(X)'$ unchanged).  Let $M_c^{2d-1}(X)''$ be the 1-motive of Definition \ref{mc2d-1} constructed from $\overline{\mathfrak{X}}'' \rightarrow \overline{X}''$.  Then $f_1$ and $f_2$ induce morphisms of 1-motives
\begin{equation*}
\hat{f}_1^*: M_c^{2d-1}(X) \rightarrow M_c^{2d-1}(X)''
\end{equation*}
and 
\begin{equation*}
\hat{f}_2^*: M_c^{2d-1}(X)' \rightarrow M_c^{2d-1}(X)''
\end{equation*}
which induce the identity on $H_c^{2d-1}(X,\mathbb{Q}_\ell)$ when one applies the functor $V_\ell(-)$.  Therefore, $\hat{f}_1^*$ and $\hat{f}_2^*$ are isogenies of 1-motives by Proposition \ref{faithful}, and 
\begin{equation*}
(\hat{f}_2^*)^{-1} \circ \hat{f}_1^*: M_c^{2d-1}(X) \rightarrow M_c^{2d-1}(X)'
\end{equation*}
is an isogeny of 1-motives fitting into the commuting diagram of Proposition \ref{uniqueness7}.  It is clear that this isomorphism is uniquely defined since $V_\ell$ is a faithful functor.
\end{proof}
Putting the results in this section together, we obtain the following:
\begin{thm}\label{cpctcodim1}
Let $k$ be an algebraically closed field, and let $Sch_d/k$ be the category of $d$-dimensional separated finite type $k$-schemes.  Then there exists a functor
\begin{equation*}
M_c^{2d-1}(-): (Sch_d/k)^{op} \rightarrow \mathscr{M}^1(k) \otimes \mathbb{Q},
\end{equation*}
functorial for proper morphisms and unique up to canonical isomorphism, such that we have
\begin{equation*}
V_\ell M_c^{2d-1}(X) \cong H_c^{2d-1}(X,\mathbb{Q}_\ell(d))
\end{equation*}
for all $\ell \not= p$.
\end{thm}

\section{Construction of $M^{2d-1}(X)$}

In this section we construct the 1-motive $M^{2d-1}(X)$ associated to a separated finite type $k$-scheme.  As in section 7, we make the following assumption:
\begin{assumption}
Let the base field $k$ be algebraically closed.
\end{assumption}
\begin{pg}\label{setupnoncompact}
We start with the same setup as in (\ref{setup}): choose a compactification $X \hookrightarrow \overline{X}$, and a resolution $\overline{\pi}: \overline{\mathfrak{X}} \rightarrow \overline{X}$, and commutative diagrams
\begin{equation*}
\xymatrix{
{\mathfrak{X}} \ar@{^{(}->}[r]^{\alpha'} \ar[d]^{\pi} & \overline{\mathfrak{X}} \ar[d]^{\overline{\pi}} & \mathcal{C} \ar@{_{(}->}[l]_{\beta'} \ar[d]^{\pi_C} \\
X \ar@{^{(}->}[r]^{\alpha} & \overline{X} & C \ar@{_{(}->}[l]_{\beta}}
\end{equation*}
and
\begin{equation*}
\xymatrix{
\mathcal{U} \ar@{^{(}->}[r]^{j'} \ar[d]^{\pi_U} & \mathfrak{X} \ar[d]^{\pi} & \mathcal{Z} \ar@{_{(}->}[l]_{i'} \ar[d]^{\pi_Z} \\
U \ar@{^{(}->}[r]^{j} & X & Z \ar@{_{(}->}[l]_{i}}
\end{equation*}
where $\mathcal{U} \rightarrow U$ is purely inseparable.  Let $\overline{Z}$ (resp. $\overline{\mathcal{Z}}$) be the closure of $Z$ in $\overline{X}$ (resp. of $\mathcal{Z}$ in $\overline{\mathfrak{X}}$).  
\end{pg}
\begin{pg}
Consider the relative Picard group of the pair $(\overline{\mathfrak{X}},\mathcal{C})$, defined by the formula
\begin{equation*}
\mathrm{Pic}(\overline{\mathfrak{X}},\mathcal{C}) = H^1(\overline{\mathfrak{X}},\mathrm{Ker}(\mathbb{G}_{m,\overline{\mathfrak{X}}} \rightarrow \beta'_*\mathbb{G}_{m,\mathcal{C}})).
\end{equation*}
The elements of $\mathrm{Pic}(\overline{\mathfrak{X}},\mathcal{C})$ are pairs $(\mathscr{L},\varphi)$, where $\mathscr{L}$ is a line bundle on $\overline{\mathfrak{X}}$ and $\varphi: \mathcal{O}_\mathcal{C} \stackrel{\sim}{\longrightarrow} \mathscr{L}\vert_{\mathcal{C}}$ is an isomorphism.  By Proposition \ref{stackrelpic}, the associated group scheme $\mathbf{Pic}_{\overline{\mathfrak{X}},\mathcal{C}}$ is representable, and we have an exact sequence
\begin{equation*}
0 \rightarrow \mathbf{Pic}_{\overline{\mathfrak{X}},\mathcal{C}}^{0,red} \rightarrow \mathbf{Pic}_{\overline{\mathfrak{X}},\mathcal{C}}^{red} \rightarrow \mathbf{NS}_{\overline{\mathfrak{X}},\mathcal{C}} \rightarrow 0
\end{equation*}
where $\mathbf{Pic}_{\overline{\mathfrak{X}},\mathcal{C}}^{0,red}$ is a semi-abelian variety and $\mathbf{NS}_{\overline{\mathfrak{X}},\mathcal{C}}$ is a finitely generated \'{e}tale-locally constant group scheme.
\end{pg}
\begin{pg}
Now consider the group $\mathrm{Div}_{\mathcal{Z}}(\overline{\mathfrak{X}})$ of divisors on $\overline{\mathfrak{X}}$ supported on $\mathcal{Z}$.  This is not the same as $\mathrm{Div}_{\overline{\mathcal{Z}}}(\overline{\mathfrak{X}})$; if $\mathcal{Z}$ has no proper irreducible components, then $\mathrm{Div}_{\mathcal{Z}}(\overline{\mathfrak{X}})$ is trivial.  Any divisor $D \in \mathrm{Div}_{\mathcal{Z}}(\overline{\mathfrak{X}})$ is disjoint from $\mathcal{C}$; therefore there is a cycle class map
\begin{equation*}
cl: \mathrm{Div}_{\mathcal{Z}}(\overline{\mathfrak{X}}) \rightarrow \mathrm{Pic}(\overline{\mathfrak{X}},\mathcal{C}),
\end{equation*}
sending $D$ to $(\mathcal{O}(D),s\vert_{\mathcal{C}}: \mathcal{O}_C \rightarrow \mathcal{O}(D)\vert_{\mathcal{C}})$ where $s: \mathcal{O}_X \rightarrow \mathcal{O}(D)$ is the meromorphic section associated to $D$.  We define a subgroup $\mathrm{Div}_{\mathcal{Z}/Z}^0(\overline{\mathfrak{X}})$ consisting of the divisors $D \in \mathrm{Div}_{\mathcal{Z}}(\overline{\mathfrak{X}})$ satisfying the following two conditions:
\begin{enumerate}
\item $cl(D) = 0$ in $NS(\overline{\mathfrak{X}},\mathcal{C})$, and
\item $\overline{\pi}_*(D) = 0$ under the proper pushforward map $\overline{\pi}_*: \mathrm{Div}_{\mathcal{Z}}(\overline{\mathfrak{X}}) \rightarrow \mathrm{Div}_Z(\overline{X})$.
\end{enumerate}
Furthermore, we can define an \'{e}tale-locally constant group scheme $\mathbf{Div}_{\mathcal{Z}/Z}^0(\overline{\mathfrak{X}})$ whose $k$-points equal $\mathrm{Div}_{\mathcal{Z}/Z}^0(\overline{\mathfrak{X}})$, and the above map $cl$ extends to a map
\begin{equation*}
cl: \mathbf{Div}_{\mathcal{Z}/Z}^0(\overline{\mathfrak{X}}) \rightarrow \mathbf{Pic}_{\overline{\mathfrak{X}},\mathcal{C}}^{0,red}.
\end{equation*}
\end{pg}
\begin{defn}
Let $X$ be a separated scheme of finite type, and choose a compactification $X \hookrightarrow \overline{X}$ and resolution $\overline{\pi}: \overline{\mathfrak{X}} \rightarrow \overline{X}$ as above.  Then we define
\begin{equation*}
M^{2d-1}(X) := [\mathbf{Div}_{\mathcal{Z}/Z}^0(\overline{\mathfrak{X}}) \rightarrow \mathbf{Pic}_{\overline{\mathfrak{X}},\mathcal{C}}^{0,red}]^\vee,
\end{equation*}
where the superscript $^\vee$ indicates taking the Cartier dual of a 1-motive.
\end{defn}
Of course, we have yet to show that $M^{2d-1}(X)$ is functorial and independent of choice of compactification.  First we discuss the $\ell$-adic realization of $M^{2d-1}(X)$.  
\subsection{$\ell$-adic realization of $M^{2d-1}(X)$.} 
Continuing with the notation of (\ref{setupnoncompact}), we have a commuting diagram in $D_c^b(X,\mathbb{Q}_\ell)$ with exact rows and columns
\begin{equation*}
\begin{CD}
j_!\mathbb{Q}_{\ell,U} @> \sim >> R\pi_*j'_!\mathbb{Q}_{\ell,\mathcal{U}} @>>> 0 @>>> \\
@VVV @VVV @VVV \\
\mathbb{Q}_{\ell,X} @>>> R\pi_*\mathbb{Q}_{\ell,\mathfrak{X}} @>>> A @>>> \\
@VVV @VVV @V \sim VV \\
i_*\mathbb{Q}_{\ell,Z} @>>> R\pi_*i'_*\mathbb{Q}_{\ell,\mathcal{Z}} @>>> i_*i^*A @>>> \\
@VVV @VVV @VVV
\end{CD}
\end{equation*}
where $A$ is simply defined to be $\mathrm{cone}(\mathbb{Q}_{\ell,X} \rightarrow R\pi_*\mathbb{Q}_{\ell,\mathfrak{X}}$).  Taking global sections of the two left-hand vertical columns, we get a diagram with exact rows \\
\begin{equation}\label{maindiagram1}
\begin{CD}
H^{2d-2}(X,\mathbb{Q}_\ell) @>>> H^{2d-2}(Z,\mathbb{Q}_\ell) @>>> H^{2d-1}(X,j_!\mathbb{Q}_\ell) @>>> H^{2d-1}(X,\mathbb{Q}_\ell) @>>> 0 \\
@VVV @VVV @VV \sim V @VVV \\
H^{2d-2}(\mathfrak{X},\mathbb{Q}_\ell) @>>> H^{2d-2}(\mathcal{Z},\mathbb{Q}_\ell) @>>> H^{2d-1}(\mathfrak{X},j'_!\mathbb{Q}_\ell) @>>> H^{2d-1}(\mathfrak{X},\mathbb{Q}_\ell) @>>> 0.
\end{CD}
\end{equation}
We will apply Poincar\'{e} duality to the terms in this diagram to interpret them in terms of divisors and cycle maps.  We start with some preliminary lemmas:
\begin{lem}\label{lem87}
Poincar\'{e} duality induces an isomorphism
\begin{equation*}
H^{2d-2}(Z,\mathbb{Q}_\ell(d-1))^\vee \stackrel{\sim}{\longrightarrow} \mathrm{Div}_{Z}(\overline{X}) \otimes_\mathbb{Z} \mathbb{Q}_\ell,
\end{equation*}
where $\mathrm{Div}_Z(\overline{X})$ is (as usual) the group of Weil divisors on $\overline{X}$ supported on $Z$ (note that $Z$ is not closed in $\overline{X}$).  Similarly, we have 
\begin{equation*}
H^{2d-2}(\mathcal{Z},\mathbb{Q}_\ell(d-1))^\vee \cong \mathrm{Div}_{\mathcal{Z}}(\overline{\mathfrak{X}}) \otimes_\mathbb{Z} \mathbb{Q}_\ell.
\end{equation*}
Finally, let
\begin{equation*}
(\pi^*)^\vee: H^{2d-2}(\mathcal{Z},\mathbb{Q}_\ell(d-1))^\vee \rightarrow H^{2d-2}(Z,\mathbb{Q}_\ell(d-1))^\vee
\end{equation*}
be the map induced by applying Poincar\'{e} duality to the map $\pi^*: H^{2d-2}(Z,\mathbb{Q}_\ell) \rightarrow H^{2d-2}(\mathcal{Z},\mathbb{Q}_\ell)$.  Then under the above isomorphisms, $(\pi^*)^\vee$ corresponds to the proper pushforward map on Weil divisors
\begin{equation*}
\pi_*: \mathrm{Div}_{\mathcal{Z}}(\overline{\mathfrak{X}}) \otimes_\mathbb{Z} \mathbb{Q}_\ell \rightarrow \mathrm{Div}_Z(\overline{X}) \otimes_\mathbb{Z} \mathbb{Q}_\ell.
\end{equation*}
\end{lem}
\begin{proof}
We prove the statement for $Z$; the proof for $\mathcal{Z}$ is the same after passing to the coarse moduli space of $\mathcal{Z}$.  Choose a resolution $Z' \rightarrow Z$ in the sense of Definition \ref{resdefn}; this induces an isomorphism $H^{2d-2}(Z,\mathbb{Q}_\ell(d-1))\cong H^{2d-2}(Z',\mathbb{Q}_\ell(d-1))$.  By Poincar\'{e} duality applied on the smooth stack $Z'$, we have that $H^{2d-2}(Z',\mathbb{Q}_\ell(d-1))^\vee$ is free on the proper $(d-1)$-dimensional connected components of $Z'$.  This can be identified with the set of $(d-1)$-dimensional proper irreducible components of $Z$; hence $H^{2d-2}(Z,\mathbb{Q}_\ell(d-1))^\vee \cong \mathrm{Div}_{Z}(\overline{X})$ as was to be shown.  The fact that $(\pi^*)^\vee$ corresponds to proper pushforward of divisors is then reduced to the case when $Z$ and $\mathcal{Z}$ are smooth, where it is standard.
\end{proof}
Next we give a concrete description of the Poincar\'{e} dual of the map
\begin{equation*}
H^{2d-2}(\mathfrak{X},\mathbb{Q}_\ell) \rightarrow H^{2d-2}(\mathcal{Z},\mathbb{Q}_\ell)
\end{equation*}
induced by the inclusion $\mathcal{Z} \hookrightarrow \mathfrak{X}$.  By the above lemma, this corresponds to a map
\begin{equation*}
g: \mathrm{Div}_{\mathcal{Z}}(\overline{\mathfrak{X}}) \otimes \mathbb{Q}_\ell \longrightarrow H_c^2(\mathfrak{X},\mathbb{Q}_\ell(1)).
\end{equation*}
\begin{lem}\label{thing}
The map $g$ above factors as
\begin{equation*}
\mathrm{Div}_{\mathcal{Z}}(\overline{\mathfrak{X}}) \otimes \mathbb{Q}_\ell \rightarrow NS(\overline{\mathfrak{X}},\mathcal{C}) \otimes \mathbb{Q}_\ell \hookrightarrow H_c^2(\mathfrak{X},\mathbb{Q}_\ell(1)),
\end{equation*}
where the map on the right is an injection and the map $\mathrm{Div}_{\mathcal{Z}}(\overline{\mathfrak{X}}) \rightarrow NS(\overline{\mathfrak{X}},\mathcal{C})$ is the class map defined earlier, sending a divisor $[D]$ to the class of $(\mathcal{O}(D),s: \mathcal{O}_X \rightarrow \mathcal{O}(D))$, where $s$ is the canonical meromorphic section of $D$ (restricted to $\mathcal{C}$, which is disjoint from the support of $D$).
\end{lem}
\begin{proof}
Another way of describing $g$ is as $H_c^2$ of the map on complexes
\begin{equation*}
i_*Ri'^!\mathbb{Q}_{\ell,\mathfrak{X}}(1) \rightarrow \mathbb{Q}_{\ell,\mathfrak{X}}(1).
\end{equation*}
Therefore $g$ is precisely the cycle class map \ref{cpctcycle} for compactly supported cohomology.  The lemma then follows from Proposition \ref{cpctfactors}.
\end{proof}
\begin{pg}
At this point, we apply Poincar\'{e} duality to the diagram \ref{maindiagram1}.  Taking into account the previous two lemmas, we get \\
\footnotesize
\begin{equation}\label{dualversion}
\begin{CD}
0 @>>> H^{2d-1}(X,\mathbb{Q}_\ell(d\textrm{-1}))^\vee @>>> H^{2d-1}(X,j_!\mathbb{Q}_\ell(d\textrm{-1}))^\vee @>>> \mathrm{Div}_Z(\overline{X}) \otimes \mathbb{Q}_\ell @>>> H^{2d-2}(X,\mathbb{Q}_\ell(d\textrm{-1}))^\vee \\
@. @AAA @AA \sim A @AAA @AAA \\
0 @>>> H_c^1(\mathfrak{X},\mathbb{Q}_\ell(1)) @>>> H^{2d-1}(\mathfrak{X},j'_!\mathbb{Q}_\ell(d\textrm{-1}))^\vee @>>> \mathrm{Div}_{\mathcal{Z}}(\overline{\mathfrak{X}}) \otimes \mathbb{Q}_\ell @>>> NS(\overline{\mathfrak{X}},\mathcal{C}) \otimes \mathbb{Q}_\ell.
\end{CD}
\end{equation} 
\normalsize

We give a concrete interpretation of the group $H^{2d-1}(\mathfrak{X},j'_!\mathbb{Q}_\ell(d\textrm{-}1))^\vee$.  By Poincar\'{e} duality, this group is isomorphic to $H_c^1(\mathfrak{X},Rj'_*\mathbb{Q}_\ell(1))$.  We will show the following:
\end{pg}
\begin{prop}\label{substep}
Let $M$ be the 1-motive
\begin{equation*}
M := [\mathbf{Div}_{\mathcal{Z}}^0(\overline{\mathfrak{X}}) \rightarrow \mathbf{Pic}^0(\overline{\mathfrak{X}},\mathcal{C})].
\end{equation*}
Then there is a canonical isomorphism $V_\ell M \stackrel{\sim}{\longrightarrow}H_c^1(\mathfrak{X},Rj'_*\mathbb{Q}_\ell(1)) $.
\end{prop}

\begin{proof}
 The argument that follows is essentially \cite[Sect. 2.5]{albpic}.  The main point is to define a map $V_\ell M \rightarrow H_c^1(\mathfrak{X},Rj'_*\mathbb{Q}_\ell(1))$.  It will then be easy (using the five lemma) to show that the map is an isomorphism.  For any $n$ prime to $p$, we define a map
\begin{equation*}
\varphi_n: T_{\mathbb{Z}/n}(M) \longrightarrow H_c^1(\mathfrak{X},Rj'_*\mu_n).
\end{equation*}
Recall that $T_{\mathbb{Z}/n}(M)$ is defined as
\begin{equation*}
T_{\mathbb{Z}/n}(M) = \frac{\{ (\mathscr{L},a,D) \in \mathrm{Pic}^0(\overline{\mathfrak{X}},\mathcal{C}) \times \mathrm{Div}^0_{\mathcal{Z}}(\overline{\mathfrak{X}}) \vert (\mathscr{L}^{\otimes n},a^{\otimes n}) \cong (\mathcal{O}(-D),s)\} }{\{(\mathcal{O}(D),s,-nD) \vert D \in \mathrm{Div}_{\mathcal{Z}}^0(\overline{\mathfrak{X}})\} }
\end{equation*}
where
\begin{enumerate}
\item $\mathscr{L}$ is a line bundle on $\overline{\mathfrak{X}}$,
\item $a: \mathcal{O}_{\mathcal{C}} \stackrel{\sim}{\rightarrow} \mathscr{L}\vert_{\mathcal{C}}$ is a trivialization of $\mathscr{L}$ on $\mathcal{C}$, and
\item $D \in \mathrm{Div}_{\mathcal{Z}}^0(\overline{\mathfrak{X}})$ is such that the class of $-D$ in $\mathrm{Pic}^0(\overline{\mathfrak{X}},\mathcal{C})$ is the same as the class of $(\mathscr{L}^{\otimes n},a^{\otimes n})$.
\end{enumerate}
Suppose given $(\mathscr{L},a,D) \in T_{\mathbb{Z}/n}(M)$, and let $D_{red}$ be the support of $D$ viewed as a closed subscheme of $\mathfrak{X}$ (it is also closed in $\overline{\mathfrak{X}})$.  Note that $D_{red}$ is disjoint from $\mathcal{C}$; let $\tilde{\mathcal{U}} = \mathfrak{X} - D_{red}$ and $\mathcal{U}' = \overline{\mathfrak{X}} - D_{red}$.  We then have the following diagram of inclusions:
\begin{equation*}
\xymatrix{
{\mathcal{U}} \ar@{^{(}->}[rd]^{g} \\
& {\tilde{\mathcal{U}}} \ar@{^{(}->}[r]^{u} \ar@{^{(}->}[d]^{\alpha} & {\mathfrak{X}} \ar@{^{(}->}[d]^{\alpha'} & {D_{red}} \ar@{_{(}->}[l]_{v} \ar@{=}[d]  \\
& {\mathcal{U}'} \ar@{^{(}->}[r]^{u'} & {\overline{\mathfrak{X}}} & {D_{red}} \ar@{_{(}->}[l]_{v'} \\
& {\mathcal{C}} \ar@{^{(}->}[u]^{\beta} \ar@{=}[r] & {}\mathcal{C} \ar@{^{(}->}[u]^{\beta'}
}
\end{equation*}
where along every row and column, the term in the middle is the union of the terms on the ends, and each square is cartesian.

Consider the cohomology group $H^1(\mathcal{U}',\alpha_!\mu_n)$, which by general nonsense \cite[Tag 03AJ]{stacksproject} is in bijection with $\alpha_!\mu_n$-torsors on $\mathcal{U}'$.  Given the class $(\mathscr{L},a,D) \in T_{\mathbb{Z}/n}(M)$ as above, choose an isomorphism $$\eta: \mathcal{O}(-D) \stackrel{\sim}{\longrightarrow} \mathscr{L}^{\otimes n}$$ such that $\eta\vert_{\mathcal{C}}: \mathcal{O}(-D)\vert_{\mathcal{C}} \cong \mathcal{O}_{\mathcal{C}} \rightarrow \mathscr{L}^{\otimes n}\vert_{\mathcal{C}}$ agrees with section $a^{\otimes n}: \mathcal{O}_{\mathcal{C}} \stackrel{\sim}{\rightarrow} \mathscr{L}^{\otimes n}\vert_{\mathcal{C}}$.  Such an isomorphism $\eta$ exists by bullet point (3) above.  Notice that $\eta$ restricts to an isomorphism on $\mathcal{U}'$.  Therefore we can define a class $$\psi_n(\mathscr{L},a,D) \in H^1(\mathcal{U}',\alpha_!\mu_n)$$ to be the $\alpha_!\mu_n$-torsor of local isomorphisms $\mathcal{O}_{\mathcal{U}}' \stackrel{\sim}{\longrightarrow} \mathscr{L}$ which are compatible with $\eta$ on $n$th tensor powers and reduce to $a$ on $D$.  By the same argument as in \ref{welldefined}, $\psi_n(\mathscr{L},a,D)$ does not depend on the choice of $\eta$.  

Next notice that because $\mathcal{C}$ and $D_{red}$ are disjoint, we have an isomorphism in $D_c^b(\overline{\mathfrak{X}})$ $$Ru'_* \alpha_! \cong \alpha'_! Ru_*,$$ so we have a sequence of maps
\begin{equation*}
H^1(\mathcal{U}',\alpha_!\mu_n) \stackrel{\sim}{\longrightarrow} H_c^1(\mathfrak{X},Ru_*\mu_n) \longrightarrow H_c^1(\mathfrak{X},Ru_*Rg_*\mu_n) = H_c^1(\mathfrak{X},Rj'_*\mu_n),
\end{equation*}
where we recall that $j' = u \circ g: \mathcal{U} \hookrightarrow \mathfrak{X}$.  We let $\varphi_n(\mathscr{L},a,d) \in H_c^1(\mathfrak{X},j'_!\mu_n)$ be the image of $\psi_n(\mathscr{L},a,D)$ under this sequence.  Taking the limit over $n = \ell^m$, we get an element
\begin{equation*}
\varphi(\mathscr{L},a,D) \in H_c^1(\mathfrak{X},\mathbb{Q}_\ell(1)).
\end{equation*}
It is not hard to show that the map $\varphi$ fits into a commutative diagram
\begin{equation*}
\begin{CD}
0 @>>> V_\ell \mathrm{Pic}^0(\overline{\mathfrak{X}},\mathcal{C}) @>>> V_\ell M @>>> \mathrm{Div}_{\mathcal{Z}}^0(\overline{\mathfrak{X}}) \otimes \mathbb{Q}_\ell @>>> 0 \\
@. @VVV @V \varphi VV @| \\
0 @>>> H_c^1(\mathfrak{X},\mathbb{Q}_\ell(1)) @>>> H_c^1(\mathfrak{X},Rj'_*\mathbb{Q}_\ell(1)) @>>> \mathrm{Div}_{\mathcal{Z}}^0(\overline{\mathfrak{X}}) \otimes \mathbb{Q}_\ell @>>> 0,
\end{CD}
\end{equation*}
where the lower row is given by the lower row of \ref{dualversion}.  Since the left-hand and right-hand vertical arrows are isomorphisms, the five lemma implies that $\varphi$ is an isomorphism.
\end{proof}
 Applying Proposition \ref{substep} to diagram \ref{dualversion}, we get an exact sequence
\begin{equation*}
0 \longrightarrow H^{2d-1}(X,\mathbb{Q}_{\ell}(1))^\vee \longrightarrow V_\ell M \longrightarrow \mathrm{Div}_Z(\overline{X}) \otimes \mathbb{Q}_\ell,
\end{equation*}
where $M = [\mathrm{Div}_{\mathcal{Z}}^0(\overline{\mathfrak{X}}) \rightarrow \mathbf{Pic}^0(\overline{\mathfrak{X}},\mathcal{C})]$.  

We leave it to the reader to check that the map $V_\ell M \rightarrow \mathrm{Div}_Z(\overline{X})$ is the obvious one, defined by the projection $M \rightarrow \mathrm{Div}_{\mathcal{Z}}^0(\overline{\mathfrak{X}})$ followed by the proper pushforward $\mathrm{Div}_{\mathcal{Z}}(\overline{\mathfrak{X}}) \rightarrow \mathrm{Div}_Z(\overline{X})$.  From this it is clear that we have  an isomorphism $V_\ell M^{2d-1}(X)^\vee \cong H^{2d-1}(X,\mathbb{Q}_\ell(1))^\vee$.  Dualizing this statement, we have shown the following:
\begin{prop}\label{m2d-1compatibility}
Let $M^{2d-1}(X)$ be defined as above.  Then for every $\ell \not= p$, there is a canonical  isomorphism
\begin{equation*}
V_\ell M^{2d-1}(X) \stackrel{\sim}{\longrightarrow} H^{2d-1}(X,\mathbb{Q}_\ell(d)).
\end{equation*}
\end{prop}

\subsection{Functoriality of $M^{2d-1}(X)$}
Our next goal is to show that $M^{2d-1}(X)$ is contravariantly functorial.  Before doing this, we prove the following:
\begin{prop}\label{pushforward2}
Let $f: X \rightarrow X'$ be a proper surjective, representable morphism between $d$-dimensional smooth proper Deligne-Mumford stacks over an algebraically closed field $k$.  Let $\partial X \subset X$ and $\partial X' \subset X'$ be reduced strict normal crossings divisors (i.e., the irreducible components of $\partial X$, $\partial X'$ are smooth) such that $f^{-1}(\partial X')_{red} \subseteq \partial X$.  Then there is pushforward morphism of algebraic groups
\begin{equation*}
f_*: \mathbf{Pic}_{X,\partial X}^{0,red} \rightarrow \mathbf{Pic}_{X',\partial X'}^{0,red}
\end{equation*}
satisfying the following conditions:
\begin{enumerate}
\item The map $f \mapsto f_*$ is compatible with composition of appropriate proper maps,
\item if $D$ is any divisor on $X$ with support disjoint from $C$, and $cl(D) \in \mathrm{Pic}(X,C)$ is the associated cycle class, then $f_*cl(D) = cl(f_*D)$, where $f_*D$ denotes the proper pushforward of $D$ to a divisor on $X'$.
\item The assignment $f \mapsto f_*$ is compatible with the proper pushforward of line bundles $\mathbf{Pic}^{0,red}_{X} \rightarrow \mathbf{Pic}^{0,red}_{X'}$ induced by proper pushforward of Weil divisors. 
\end{enumerate}
\end{prop}  
\begin{proof}
We remark that this is a refinement and generalization \cite[Lemma 6.2]{albpic} to the case when the base field $k$ has positive characteristic.  Because their proof uses resolution of singularities, we use a different approach.

To begin the proof, first note that by considering the obvious restriction functor $$\mathbf{Pic}_{X,\partial X}^{0,red} \rightarrow \mathbf{Pic}_{X,f^{-1}(\partial X')_{red}}^{0,red},$$ we may assume $\partial X = f^{-1}(\partial X')_{red}$.  Next let $U' \subset X$ be an open substack of $X$ such that $f^{-1}(U') \rightarrow U'$ is finite and flat and such that $Z' := X' - U'$ is of codimension 2.  To see that such a $U'$ exists, one easily reduces to the case of schemes since $f$ is representable, and in that case it follows since the dimension of fibers is an upper semi-continuous function \cite[Ex. 3.22]{hartshorne} and $f$ is flat over every codimension-1 point of $X'$.  Let $U = f^{-1}(U')$, $\partial U' = U' \cap \partial X'$, and $\partial U = f^{-1}(\partial U)_{red}$, so we have a commutative diagram of pairs
\begin{equation*}
\xymatrix{
{(U, \partial U)} \ar@{^{(}->}[r] \ar[d]^{f} & {(X,\partial X)} \ar[d]^{f} \\
{(U',\partial U')} \ar@{^{(}->}[r] & {(X',\partial X')}.
}
\end{equation*}

Now consider sheaves $\mathbf{Pic}_{U,\partial U}^{red}$ and $\mathbf{Pic}_{U',\partial U'}^{red}$ on $(Sm/k)_{et}$ defined as in the case when $U, U'$ are proper, namely $\mathbf{Pic}_{U,\partial U}^{red}$ is the sheafification of the functor
\begin{equation*}
W \mapsto \mathrm{Pic}(U \times W,\partial U \times W)
\end{equation*}
(with a similar definition for $\mathbf{Pic}_{U',\partial U'}^{red}$).  As in the case of proper schemes, we have
\begin{equation*}
\mathbf{Pic}_{U,\partial U}^{red} = R^1\pi_*(\mathbb{G}_{m,U,\partial U})
\end{equation*}
where $\pi: U \rightarrow \mathrm{Spec \ }k$ is the structure morphism and $\mathbb{G}_{m,U,\partial U} = \mathrm{Ker}(\mathbb{G}_{m,U} \rightarrow a_*\mathbb{G}_{m,\partial U})$ and $a: \partial U \hookrightarrow U$ is the inclusion.  A similar formula holds for $\mathbf{Pic}_{U,\partial U'}^{red}$.  There is a natural map of sheaves
\begin{equation*}
N: \mathbf{Pic}_{U,\partial U}^{red} \longrightarrow \mathbf{Pic}_{U',\partial U'}^{red}
\end{equation*}
induced by the norm map on sheaves
\begin{equation*}
N: f_* \mathbb{G}_{m,U} \rightarrow \mathbb{G}_{m,U'}
\end{equation*}
and then appliying $R^1\pi_*$ (see \cite[p. 61]{albpic} for a proof that $N$ restricts to a morphism of subsheaves $f_*\mathbb{G}_{m,U,\partial U} \rightarrow \mathbb{G}_{m,U',\partial U'}$).  We then have a map
\begin{equation*}
f_*: \mathbf{Pic}_{X,\partial X}^{0,red} \stackrel{restr}{\longrightarrow} \mathbf{Pic}_{U,\partial U}^{red} \stackrel{N}{\longrightarrow} \mathbf{Pic}_{U',\partial U'}^{red}.
\end{equation*}
We claim the following, which defines the required map $f_*: \mathbf{Pic}_{X,\partial X}^{0,red} \rightarrow \mathbf{Pic}_{X',\partial X'}^{0,red}$:
\begin{lem}\label{hardlemma}
The inclusion $U' \hookrightarrow X'$ induces an injection of sheaves $\mathbf{Pic}_{X',\partial X'}^{red} \hookrightarrow \mathbf{Pic}_{U',\partial U'}^{red}$, and the map $f_*$ defined above factors through this subsheaf.  Since $\mathbf{Pic}_{X,\partial X}^{0,red}$ is connected, this implies that $f_*$ factors through $\mathbf{Pic}_{X',\partial X'}^{0,red}$.
\end{lem}
\begin{rem}
It is clear that the resulting map $f_*$ is independent of the choice of $U' \subset X'$, since any two choices $U_1'$ and $U_2'$ are dominated by the third open subset $U_3' = U_1' \cap U_2'$, and $U_3'$ also has complement of codimension $\geq 2$.
\end{rem}
\begin{proof}(of lemma) It suffices to look at the maps on closed points, i.e., the map $\mathrm{Pic}^0(X',\partial X') \rightarrow \mathrm{Pic}(U',\partial U')$, etc.  Now consider the inclusion $U' \hookrightarrow X'$; it induces a commutative diagram
\begin{equation*}
\xymatrix{
\mathcal{O}^*(X') \ar[r] \ar[d]^{\sim} & \mathcal{O}^*(\partial X') \ar[r] \ar@{^{(}->}[d] & \mathrm{Pic}(X',\partial X') \ar[r] \ar[d] & \mathrm{Pic}(X') \ar[r] \ar[d]^{\sim} & \mathrm{Pic}(\partial X') \ar[d] \\
\mathcal{O}^*(U') \ar[r] & \mathcal{O}^*(\partial U') \ar[r] & \mathrm{Pic}(U',\partial U') \ar[r] & \mathrm{Pic}(U') \ar[r] & \mathrm{Pic}(\partial U'). }
\end{equation*}
The map $\mathrm{Pic}(X') \rightarrow \mathrm{Pic}(U')$ is an isomorphism because the complement $Z' \subset X'$ has codimension $\geq 2$.  This immediately implies that $\mathrm{Pic}(X',\partial X')$ injects into $\mathrm{Pic}(U',\partial U')$; moreover, if we let $C := \mathrm{Coker}(\mathrm{Pic}(X',\partial X') \rightarrow \mathrm{Pic}(U',\partial U'))$, we have an exact sequence
\begin{equation*}
0 \rightarrow \mathcal{O}^*(\partial U')/\mathcal{O}^*(\partial X') \rightarrow C \rightarrow \mathrm{Ker}(\mathrm{Pic}(\partial X') \rightarrow \mathrm{Pic}(\partial U')).
\end{equation*}
We need to show that if $L = (\mathscr{L},\sigma: \mathcal{O}_{\partial X} \stackrel{\sim}{\rightarrow} \mathscr{L}\vert_{\partial X}) \in \mathrm{Pic}^0(X,\partial X)$, then the image of $f_*L$ in $C$ is 0. 

First we show that the image of $L$ in $\mathrm{Ker}(\mathrm{Pic}(\partial X') \rightarrow \mathrm{Pic}(\partial U'))$ is 0.  Let $$K := \mathrm{Ker}(\mathrm{Pic}(\partial X') \rightarrow \mathrm{Pic}(\partial U')).$$  Concretely, the map $\mathrm{Pic}^0(X,\partial X) \rightarrow K$ is defined as follows: given $L = (\mathscr{L},\sigma) \in \mathrm{Pic}^0(X,\partial X)$, we have $(N(\mathscr{L}\vert_U),\mathrm{det}(\sigma\vert_U)) \in \mathrm{Pic}(U',\partial U')$.  Then there exists a line bundle $M \in \mathrm{Pic}(X')$ with $M\vert_{U'} = N(\mathscr{L}\vert_U)$, and the image of $L$ in $K$ is $M\vert_{\partial X'}$.  From this description it is clear that the map $\mathrm{Pic}^0(X,\partial X) \rightarrow \mathrm{Pic}(\partial X')$ factors through $\mathrm{Pic}^0(X)$ (i.e., the image in $K$ only depends on the line bundle $\mathscr{L}$ and not on the trivialization $\sigma$).  We therefore have a factorization
\begin{equation*}
\mathrm{Pic}^0(X,\partial X) \longrightarrow A \longrightarrow K,
\end{equation*}
where $A := \mathrm{Image}(\mathrm{Pic}^0(X,\partial X) \rightarrow \mathrm{Pic}^0(X))$ is an abelian variety (since $\mathrm{Pic}^0(X)$ is).  

We claim that $\mathrm{Ker}(\mathrm{Pic}(\partial X') \rightarrow \mathrm{Pic}(\partial U'))$ is a group variety whose connected component of the identity is a torus.  From this it will follow that the map $\mathrm{Pic}^0(X,\partial X) \rightarrow K$ is zero, since it factors through the abelian variety $A$.  To prove this claim, we first set up some notation: let $C_i$ be the (smooth) irreducible components of $\partial X'$, and for each increasing sequence $i_0 < ... < i_n$, let $C_{i_0...i_n} = C_{i_0} \cap ... \cap C_{i_n}$.  Then \cite[Lemma 3.2]{bakhtary} we have a resolution of sheaves on $(\partial X')_{et}$
\begin{equation*}
0 \longrightarrow \mathcal{O}_{\partial X'} \longrightarrow \bigoplus_i \mathcal{O}_{C_i} \longrightarrow \bigoplus_{i < j} \mathcal{O}_{C_{ij}} \longrightarrow ...,
\end{equation*}
where we have abused notation and written $\mathcal{O}_{C_i}$ instead of $\iota_* \mathcal{O}_{C_i}$ for $\iota: C_i \hookrightarrow \partial X$ the closed immersion  (the reference (loc. cit) only proves this for schemes, but the statement is local for the \'{e}tale topology and so immediately follows for stacks).  This sequence remains exact when one takes units: to see this, we can work locally.  Each morphism of local rings $\mathcal{O}_{C_{i_0...i_n}} \rightarrow \mathcal{O}_{C_{i_0...i_{n+1}}}$ is a surjection of local rings (whenever it is non-zero), and for a surjection of local rings $\pi: R \rightarrow S$, $r \in R$ is a unit if and only if $\pi(r)$ is a unit.  From this the exactness of the above sequence on units is immediate, giving an exact sequence of sheaves of abelian groups
\begin{equation}\label{sncresolution}
0 \longrightarrow \mathbb{G}_{m,\partial X'} \longrightarrow \bigoplus_i \mathbb{G}_{m,C_i} \longrightarrow \bigoplus_{i<j} \mathbb{G}_{m,C_{ij}} \longrightarrow ...
\end{equation}
From this resolution we get an exact sequence
\begin{equation*}
0 \rightarrow T \rightarrow \mathrm{Pic}(\partial X') \rightarrow \bigoplus_i \mathrm{Pic}(C_i),
\end{equation*}
where
\begin{equation*}
T:= \frac{\mathrm{Ker}(\oplus_{i < j} \mathcal{O}^*(C_{ij}) \rightarrow \oplus_{i < j < k} \mathcal{O}^*(C_{ijk}))}{\mathrm{Image}(\oplus_{i} \mathcal{O}^*(C_i) \rightarrow \oplus_{i < j} \mathcal{O}^*(C_{ij}))}
\end{equation*}
 is an extension of a finite abelian group by a torus.  Here the maps are induced by restriction.

We can restrict sequence \ref{sncresolution} to $\partial U'$ and get an exact sequence
\begin{equation*}
0 \rightarrow T_U \rightarrow \mathrm{Pic}(\partial U') \rightarrow \oplus_i \mathrm{Pic}(C_i\vert_U'),
\end{equation*}
where $T_U$ is defined by the same formula as for $\partial X'$ (replacing $C_i$ by $C_i\vert_{U'}$).  Moreover, the inclusion $\partial U' \hookrightarrow \partial X'$ induces a commutative diagram
\begin{equation*}
\begin{CD}
0 @>>> T @>>> \mathrm{Pic}(\partial X') @>>> \bigoplus_i \mathrm{Pic}(C_i) \\
@. @VVV @VVV @VVV \\
0 @>>> T_{U'} @>>> \mathrm{Pic}(\partial U') @>>> \bigoplus_i \mathrm{Pic}(C_i\vert_{U'}).
\end{CD}
\end{equation*}
But notice that since $C_i$ is smooth, the kernel of the map on the right is finitely generated.  Therefore we have an exact sequence
\begin{equation*}
0 \longrightarrow \tilde{T} \longrightarrow K \rightarrow F
\end{equation*}
where $\tilde{T} := \mathrm{Ker}(T \rightarrow T_U)$ is an extension of a torus by a finite group and $F$ is finitely generated.  Therefore the map $\mathrm{Pic}^0(X,\partial X) \rightarrow K$ must be zero, since it factors through the abelian variety $\mathrm{Image}(\mathrm{Pic}^0(X,\partial X) \rightarrow \mathrm{Pic}^0(X))$, and there are no non-zero maps from an abelian variety to a torus.

We have shown that the map $\mathrm{Pic}^0(X,\partial X) \rightarrow C$ factors through $\mathcal{O}^*(\partial U')/\mathcal{O}^*(\partial X')$, where we recall that $C = \mathrm{Coker}(\mathrm{Pic}(X',\partial X') \rightarrow \mathrm{Pic}(U',\partial U'))$.  We want to show that the resulting map 
\begin{equation}\label{lastthing}
f: \mathrm{Pic}^0(X,\partial X) \rightarrow \mathcal{O}^*(\partial U')/\mathcal{O}^*(\partial X')
\end{equation} is zero.  Recall that we have an exact sequence
\begin{equation*}
0 \rightarrow \mathcal{O}^*(\partial X)/\mathcal{O}^*(X) \rightarrow \mathrm{Pic}^0(X,\partial X) \rightarrow \mathrm{Pic}^0(X).
\end{equation*}
To show that the map in \ref{lastthing} is zero, we start by showing that the restriction 
\begin{equation}\label{restrictedmap}
f\vert_{\mathcal{O}^*(\partial X)}: \mathcal{O}^*(\partial X)/\mathcal{O}^*(X) \rightarrow \mathcal{O}^*(\partial U')/\mathcal{O}^*(\partial X')
\end{equation}
 is zero.  To do this, we first explicitly describe this map.  An element of $\mathcal{O}^*(\partial X)$ can be given as follows: first label the connected components of $\partial X$ as $C_1,...,C_n$.  Then for each $i$ we let let $a_i \in k^*$ be the unit which is multiplication by $a_i$ on $C_i$, and the identity on the other connected components.  Then we have a corresponding element $$L := (\mathcal{O}_X,\Pi_i a_i) \in \mathrm{Pic}^0(X,\partial X).$$  We claim that $f_*L \in \mathrm{Pic}(U',\partial U')$ can be described as follows:
\begin{lem}\label{definenorm}
Let $D_i \subset \partial U'$ be any connected component, and let $E_1,...,E_s$ be the connected components of $f^{-1}(D_i)_{red}$, and $d_1,...,d_s$ the degrees of these connected components under the map $\partial U \times_{\partial U'} D_i \rightarrow D_i$, and let $r_1,...,r_s$ be their ramification degrees (so $e_1r_1 + ... + e_sr_s = \mathrm{deg}(f)$).  For each $j, 1 \leq j \leq s$, let $C_{\phi(j)}$ be the connected component of $\partial X$ containing $E_j$.  Then let $$b_i = \prod_{j=1}^s (a_{\phi(j)})^{e_jr_j} \in k^*,$$ which we think of as a unit in $\mathcal{O}^*(\partial U')$ which is multiplication by $b_i$ on $D_i$, and the identity on the other connected components.  We then have $$f_*L = (\mathcal{O}_{U'},\prod_{i} b_i).$$
\end{lem}
\begin{proof}(of lemma \ref{definenorm}) In general, the map $\mathcal{O}^*(\partial U) \rightarrow \mathcal{O}^*(\partial U')$ is obtained by locally lifting an element of $\mathcal{O}^*(\partial U)$ to $\mathcal{O}^*(U)$, and then applying the norm map and restricting to $\partial U'$.  However, if we let $\overline{\partial U}  = U \times_{U'} \partial U'$ (so that $\partial U = \overline{\partial U}_{red})$, we have a commutative diagram \cite[6.4.8]{ega2} 
\begin{equation*}
\begin{CD}
\mathbb{G}_{m,U} @>>> \iota_*\mathbb{G}_{m,\overline{\partial U}} \\
@V N VV @V N VV \\
\mathbb{G}_{m,U'} @>>> \iota'_*\mathbb{G}_{m,\partial U'}
\end{CD}
\end{equation*}
where $\iota, \iota'$ are the inclusions.  This implies that we only have to lift a section to $\mathcal{O}^*(\overline{\partial U})$ and apply the norm map there.  For the section $\Pi_i a_i$ we are interested in, this can be done globally and the resulting formula for $N(\Pi_i a_i)$ given in the lemma statement is immediate.
\end{proof}
We return to showing that the map in \ref{restrictedmap} is zero.  From the description of $f_*L$ given in Lemma \ref{definenorm}, we see that if $D_i$ and $D_j$ are connected components of $\partial U$ which belong to the same connected component of $\partial X'$, then $b_i = b_j$.  This implies that the section $\Pi_i b_i$ extends to an section of $\mathcal{O}^*(\partial X')$, which in turn implies that the image of $f_*L$ under the map \ref{restrictedmap} is zero.

We have shown that the map $f$ of \ref{lastthing} factors through $\mathrm{Pic}^0(X,\partial X)/\mathcal{O}^*(\partial X)$, which is a subvariety of $\mathrm{Pic}^0(X)$ and hence an abelian variety.  But $\mathcal{O}^*(\partial U')/\mathcal{O}^*(\partial X')$ is an extension of finitely generated group by a torus: if $\partial U'$ is smooth then this is clear, while in the general case it follows from the commuting diagram (where the rows are equalizers)
\begin{equation*}
\xymatrix{
\mathcal{O}^*(\partial X') \ar[r] \ar[d] &  \bigoplus_i \mathcal{O}^*(\partial X'_i) \ar@<.7ex>[r] \ar@<-.7ex>[r] \ar[d] & \bigoplus_{i < j} \mathcal{O}^*(\partial X'_{ij}) \ar[d] \\
\mathcal{O}^*(\partial U') \ar[r]  &  \bigoplus_i \mathcal{O}^*(\partial U'_i)  \ar@<.7ex>[r] \ar@<-.7ex>[r] & \bigoplus_{i < j} \mathcal{O}^*(\partial U'_{ij}), }
\end{equation*}
where $\partial X_i'$ are the (smooth) irreducible components of $\partial X$, and similarly for $\partial U'$.  This implies that the resulting map $\mathrm{Pic}^0(X,\partial X)/\mathcal{O}^*(\partial X) \rightarrow \mathcal{O}^*(\partial U')/\mathcal{O}^*(\partial X')$ is zero (since there's no non-zero map from an abelian variety to a torus).  We have finally shown that the map $\mathrm{Pic}^0(X,\partial X) \rightarrow C$ is the zero map, completing the proof of Lemma \ref{hardlemma}.
\end{proof}
This defines the required map $f_*: \mathbf{Pic}_{X,\partial X}^{0,red} \rightarrow \mathbf{Pic}_{X',\partial X'}^{0,red}$.  It is clear that this map satisfies conditions (1)-(3) of the proposition statement, since in each case we can restrict to the case when $f$ is finite and flat (by the way $f_*$ was defined), where it follows from standard compatibility properties between pushforward of divisors and the norm map.
\end{proof}
\begin{pg}\label{m2d-1setup}
We can now define the functoriality of $M^{2d-1}(X)$ as follows.  Let $f: X \rightarrow Y$ be a morphism between $d$-dimensional separated finite type $k$-schemes, and choose a compactified morphism $\overline{f}: \overline{X} \rightarrow \overline{Y}$.  As in \ref{resolvemap} and \ref{2d-1functoriality}, we can choose a commutative diagram
\begin{equation*}
\begin{CD}
\overline{\mathfrak{X}} @> \overline{f}' >> \overline{\mathcal{Y}} \\
@V \pi VV  @V \sigma VV \\
\overline{X} @> \overline{f} >> \overline{Y}
\end{CD}
\end{equation*}
where $\pi: \overline{\mathfrak{X}} \rightarrow \overline{X}$ and $\sigma: \overline{\mathcal{Y}} \rightarrow \overline{Y}$ are resolutions and $\overline{f}'$ is representable.  Let $\mathfrak{X} = X \times_{\overline{X}} \overline{\mathfrak{X}}$ and $\mathcal{Y} := Y \times_{\overline{Y}} \overline{\mathcal{Y}}$.  Then we can arrange that $\mathcal{C} := \overline{\mathfrak{X}} - \mathfrak{X}$ and $\mathcal{D} := \overline{\mathcal{Y}} - \mathcal{Y}$ are reduced strict normal crossings divisors; we have $f^{-1}(\mathcal{D})_{red} \subseteq \mathcal{C}$.  Let $V \subset Y$ be an open subset of $V$ such that $\mathcal{V} := \sigma^{-1}(V) \rightarrow V$ is purely inseparable.  Moreover, by shrinking $V$ we can arrange that $\mathfrak{X} \times_ {\mathcal{Y}} \mathcal{V} \rightarrow X \times_{Y} V$ is puresly inseparable.  Set $\mathcal{U} = \mathfrak{X} - \mathfrak{X} \times_{\mathcal{Y}} \mathcal{V}$ and $U = X - X \times_Y V$.   Finally, let $Z = X - U$, $\mathcal{Z} = \mathfrak{X} - \mathcal{U}$, $W = Y - V$, $\mathcal{W} = \mathcal{Y} - \mathcal{V}$.  We get a commuting diagrams
\begin{equation*}
\xymatrix{
{\mathfrak{X}} \ar@{^{(}->}[r] \ar[d] & {\overline{\mathfrak{X}}} \ar[d] & \mathcal{C} \ar@{_{(}->}[l] \\
{\mathcal{Y}} \ar@{^{(}->}[r] & {\overline{\mathcal{Y}}} & {\mathcal{D}} \ar@{_{(}->}[l]}
\end{equation*}
and 
\begin{equation*}
\xymatrix{
\mathcal{U} \ar@{^{(}->}[r] \ar[d] & \mathfrak{X} \ar[d] & \mathcal{Z} \ar@{_{(}->}[l] \ar[d] \\
{\mathcal{V}} \ar@{^{(}->}[r] & {\mathcal{Y}} & \mathcal{W}. \ar@{_{(}->}[l]}
\end{equation*}

With this notation, to define a map $\hat{f}^*: M^{2d-1}(Y) \rightarrow M^{2d-1}(X)$, we want to define a map 
\begin{equation}\label{m2d-1map}
\hat{f}_*: [\mathbf{Div}^0_{\mathcal{Z}/Z}(\overline{\mathfrak{X}}) \rightarrow \mathbf{Pic}^{0,red}_{\overline{\mathfrak{X}},\mathcal{C}}] \longrightarrow [\mathbf{Div}^0_{\mathcal{W}/W}(\overline{\mathcal{Y}}) \rightarrow \mathbf{Pic}^{0,red}_{\overline{\mathcal{Y}},\mathcal{D}}].
\end{equation}
This is possible by using the proper pushforward map $\mathbf{Div}^0_{\mathcal{Z}/Z}(\overline{\mathfrak{X}}) \rightarrow \mathbf{Div}^0_{\mathcal{W}/W}(\overline{\mathcal{Y}})$ and the pushforward map $f_*: \mathbf{Pic}^{0,red}_{\overline{\mathfrak{X}},\mathcal{C}} \rightarrow \mathbf{Pic}^{0,red}_{\overline{\mathcal{Y}},\mathcal{D}}$ provided by Proposition \ref{pushforward2}.  The dual of the map $\hat{f}_*$ is our desired map $\hat{f}^*: M^{2d-1}(Y) \rightarrow M^{2d-1}(X)$. 
\end{pg}
\begin{pg}
Our next goal is to show that this pullback map $\hat{f}^*: M^{2d-1}(Y) \rightarrow M^{2d-1}(X)$ is compatible with $\ell$-adic realizations; i.e., we claim the following:
\end{pg}
\begin{prop}\label{prop86}
In the notation of \ref{m2d-1setup}, we have a commutative diagram
\begin{equation*}
\begin{CD}
V_\ell M^{2d-1}(Y) @> \alpha_{Y} >> H^{2d-1}(Y,\mathbb{Q}_\ell(d)) \\
@V {V_\ell \hat{f}^*} VV @V f^* VV \\
V_\ell M^{2d-1}(X) @> \alpha_{X} >> H^{2d-1}(X,\mathbb{Q}_\ell(d)),
\end{CD}
\end{equation*}
where $\alpha_Y$ and $\alpha_X$ are the comparison isomorphisms of \ref{m2d-1compatibility}.
\end{prop}
\begin{proof}
Consider the 1-motives
\begin{align*}
M &:= [\mathbf{Div}^0_{\mathcal{Z}}(\overline{\mathfrak{X}}) \rightarrow \mathbf{Pic}^{0,red}_{\overline{\mathfrak{X}},\mathcal{C}}] \ \ \mathrm{and} \\
N &:= [\mathbf{Div}^0_{\mathcal{W}}(\overline{\mathcal{Y}}) \rightarrow \mathbf{Pic}^{0,red}_{\overline{\mathcal{Y}},\mathcal{D}}].
\end{align*}
It is clear that the map $\hat{f}_*$ of \ref{m2d-1map} extends to a map of 1-motives $\hat{f}_*: M \rightarrow N$.  Moreover, in the notation of \ref{dualversion}, we have $V_\ell M \cong H_c^1(\mathfrak{X},Rj'_*\mathbb{Q}_\ell(1))$ and $V_\ell N \cong H_c^1(\mathcal{Y},Rk'_*\mathbb{Q}_\ell(1))$, where $j': \mathcal{U} \hookrightarrow \mathfrak{X}$ and $k': \mathcal{V} \hookrightarrow \mathcal{Y}$ are the inclusions.  Since $H^{2d-1}(X,\mathbb{Q}_\ell(d-1))^\vee$ injects into $H_c^1(\mathfrak{X},Rj'_*\mathbb{Q}_\ell(1))$ and $H^{2d-1}(Y,\mathbb{Q}_\ell(d-1))^\vee$ injects into $H_c^1(\mathcal{Y},Rk'_*\mathbb{Q}_\ell(1))$, it suffices to show that $V_\ell \hat{f}_*: V_\ell M \rightarrow V_\ell N$ is compatible with the proper pushforward $f_*: H_c^1(\mathfrak{X},\mathbb{Q}_\ell(1)) \rightarrow H_c^1(\mathcal{Y},\mathbb{Q}_\ell(1))$.  Note further that $f_*$ and $V_\ell \hat{f}_*$ both induce morphisms  of short exact sequences
\begin{equation*}
\begin{CD}
0 @>>> H_c^1(\mathfrak{X},\mathbb{Q}_\ell(1)) @>>> H_c^1(\mathfrak{X},Rj'_*\mathbb{Q}_\ell(1)) @>>> \mathrm{Div}_{\mathcal{Z}}^0(\overline{\mathfrak{X}}) \otimes \mathbb{Q}_\ell @>>> 0 \\
@. @V {f_*,V_\ell \hat{f}_*} VV  @V {f_*,V_\ell \hat{f}_*} VV @V {f_*,V_\ell \hat{f}_*} VV \\
0 @>>> H_c^1(\mathcal{Y},\mathbb{Q}_\ell(1)) @>>> H_c^1(\mathcal{Y},Rk'_*\mathbb{Q}_\ell(1)) @>>> \mathrm{Div}_{\mathcal{W}}^0(\overline{\mathcal{Y}}) \otimes \mathbb{Q}_\ell @>>> 0.
\end{CD}
\end{equation*}
Therefore it suffices to show that the maps induced by $f_*$ and $V_\ell \hat{f}_*$ agree on $H_c^1(\mathfrak{X},\mathbb{Q}_\ell(1))$ and on $\mathrm{Div}_{\mathcal{Z}}^0(\overline{\mathfrak{X}}) \otimes \mathbb{Q}_\ell$.  Since $f_*$ and $V_\ell \hat{f}_*$ are both defined by proper pushforward on $\mathrm{Div}_{\mathcal{Z}}^0(\overline{\mathfrak{X}}) \otimes \mathbb{Q}_\ell$, it is clear that the action of $V_\ell \hat{f}_*$ and $f_*$ on this group agree.  Therefore we are left with showing that $f_*$ and $V_\ell \hat{f}_*$ induce the same map on $H_c^1(\mathfrak{X},\mathbb{Q}_\ell(1))$.  We state this as the following lemma, which completes the proof of \ref{prop86}.
\end{proof}
\begin{lem}
In the notation above, we have a commutative diagram
\begin{equation*}
\begin{CD}
V_\ell \mathrm{Pic}^0(\overline{\mathfrak{X}},\mathcal{C}) @> \sim >> H_c^1(\mathfrak{X},\mathbb{Q}_\ell(1)) \\
@V {V_\ell \hat{f}_*} VV @V {f_*} VV \\
V_\ell \mathrm{Pic}^0(\overline{\mathcal{Y}},\mathcal{D}) @> \sim >> H_c^1(\mathcal{Y},\mathbb{Q}_\ell(1)),
\end{CD}
\end{equation*}
where the horizontal arrows are the canonical comparison isomorphisms.
\end{lem}
\begin{proof}
Let $B \subset \overline{\mathcal{Y}}$ be an open substack such that $f^{-1}(B) \rightarrow B$ is finite flat, and $\overline{\mathcal{Y}} - B$ has codimension 2.  Let $A = f^{-1}(B)$, and let $\alpha: A \cap \mathfrak{X} \hookrightarrow A$ and $\beta: B \cap \mathcal{Y} \hookrightarrow B$ be the inclusions.  Then the inclusions $A \hookrightarrow \overline{\mathfrak{X}}$ and $B \hookrightarrow \overline{\mathcal{Y}}$ induce a commutative diagram
\begin{equation*}
\xymatrix{ 
H_c^1(\mathfrak{X},\mathbb{Q}_\ell(1)) \ar@{^{(}->}[r] \ar[d]^{f_*} & H^1(A,\alpha_!\mathbb{Q}_\ell(1)) \ar[d]^{f_*} \\
H_c^1(\mathcal{Y},\mathbb{Q}_\ell(1)) \ar@{^{(}->}[r] & H^1(B,\beta_!\mathbb{Q}_\ell(1)) 
}
\end{equation*}
where the horizontal arrows are injections.   Since $f: A \rightarrow B$ is finite flat, $f_*$ is induced by the trace mapping $Tr: f_*f^*\mathbb{Q}_\ell(1) \rightarrow \mathbb{Q}_\ell(1)$.  By \cite[p. 136]{freitag}, we have a commutative diagram of sheaves
\begin{equation*}
\begin{CD}
0 @>>> f_*\alpha_!\mu_{\ell^n} @>>> f_*\mathbb{G}_{m,A,\mathcal{C} \cap A} @> \ell^n >> f_*\mathbb{G}_{m,A,\mathcal{C} \cap A} @>>> 0 \\
@. @V Tr VV @V N VV @V N VV \\
0 @>>> \beta_! \mu_{\ell^n} @>>> \mathbb{G}_{m,B,\mathcal{D} \cap B} @> \ell^n >> \mathbb{G}_{m,B,\mathcal{D} \cap B} @>>> 0.
\end{CD}
\end{equation*}
Here $N: f_*\mathbb{G}_{m,A} \rightarrow \mathbb{G}_{m,B}$ is the norm mapping.  Taking global sections and then inverse limits induces a commutative diagram
\begin{equation*}
\begin{CD}
0 @>>> G @>>> H^1(A,\alpha_!\mathbb{Q}_\ell(1)) @>>> V_\ell \mathrm{Pic}(A,\mathcal{C} \cap A) @>>> 0 \\
@. @VVV @V f_* VV @V N VV \\
0 @>>> 0 @>>> H^1(B,\beta_!\mathbb{Q}_\ell(1)) @>>> V_\ell \mathrm{Pic}(B,\mathcal{D} \cap B) @>>> 0,
\end{CD}
\end{equation*}
where 
\begin{equation*}
G := \ilim_n \frac{\mathrm{Ker}(\mathcal{O}^*(A) \rightarrow \mathcal{O}^*(A \cap \mathcal{C}))}{\ell^n \mathrm{Ker}(\mathcal{O}^*(A) \rightarrow \mathcal{O}^*(A \cap \mathcal{C}))}
\end{equation*}
(note that the corresponding group for $B$ is zero since $B$ is of codimension 2 in the smooth proper Deligne-Mumford stack $\overline{\mathcal{Y}}$).  In summary, we have a commuting diagram
\begin{equation*}
\xymatrix{
H_c^1(\mathfrak{X},\mathbb{Q}_\ell(1)) \ar@{^{(}->}[r] \ar[d]^{f_*} & H^1(A,\alpha_!\mathbb{Q}_\ell(1)) \ar[r] \ar[d]^{f_*} & {V_\ell \mathrm{Pic}(A,\mathcal{C} \cap A)} \ar[d]^N \\
H_c^1(\mathcal{Y},\mathbb{Q}_\ell(1)) \ar@{^{(}->}[r] & H^1(B,\beta_!\mathbb{Q}_\ell(1)) \ar[r]^{\sim} & {V_\ell \mathrm{Pic}(B,\mathcal{D} \cap B)}
}
\end{equation*}
showing that the cohomological pushforward $f_*$ is compatible with taking norms of line bundles.  On the other hand, the pushforward of 1-motives $\hat{f}_*: \mathrm{Pic}^0(\overline{\mathfrak{X}},\mathcal{C}) \rightarrow \mathrm{Pic}^0(\overline{\mathcal{Y}},\mathcal{D})$ is defined so that there is a commutative diagram
\begin{equation*}
\xymatrix{
\mathrm{Pic}^0(\overline{\mathfrak{X}},\mathcal{C}) \ar[d] \ar[r] & \mathrm{Pic}(A,\mathcal{C} \cap A) \ar[d] \\
\mathrm{Pic}^0(\overline{\mathcal{Y}},\mathcal{D}) \ar@{^{(}->}[r] & \mathrm{Pic}(B,\mathcal{D} \cap B).
}
\end{equation*}
Applying $V_\ell(-)$ to this diagram, and combining with the diagram above, shows that $f_*$ and $V_\ell \hat{f}_*$ are compatible in the sense of the proposition statement.
\end{proof}
\begin{pg}
We can finally show that $M^{2d-1}(X)$ is independent of choice of compactification and resolution.  In particular, we have the following: 
\begin{prop}\label{uniqueness8}
Let $X$ be a $d$-dimensional separated finite type $k$-scheme, and let $\overline{\mathfrak{X}}$, $\overline{\mathfrak{X}}'$ be two distinct resolutions of compactifications of $X$.  Let $M^{2d-1}(X)$ and $M^{2d-1}(X)'$ be the 1-motives constructed using $\overline{\mathfrak{X}}$ and $\overline{\mathfrak{X}}'$, respectively.  Then there exists a unique isogeny of 1-motives $$a: M^{2d-1}(X) \rightarrow M^{2d-1}(X)'$$ fitting into a diagram
\begin{equation*}
\begin{CD}
V_\ell M^{2d-1}(X) @> \alpha_X >> H^{2d-1}(X,\mathbb{Q}_\ell(d)) \\
@V {V_\ell a} VV @V = VV \\
V_\ell M^{2d-1}(X)' @> {\alpha_X}' >> H^{2d-1}(X,\mathbb{Q}_\ell(d))
\end{CD}
\end{equation*}
for all $\ell \not= p$, where $\alpha_X$ and $\alpha_X'$ are the comparison isomorphisms of Proposition \ref{m2d-1compatibility}.
\end{prop} 
\begin{proof}
Since we have shown that these 1-motives are contravariantly functorial in a way that is compatible with $\ell$-adic realizations, we can prove this proposition by the exact same argument as used for Proposition \ref{uniqueness7}.
\end{proof}
\end{pg}
The end result of our work in the section is the following:
\begin{thm}\label{noncpctcodim1}
Let $k$ be an algebraically closed field, and $Sch_d/k$ the category of $d$-dimensional separated finite type $k$-schemes.  Then there exists a functor
\begin{equation*}
M^{2d-1}(-): (Sch_d/k)^{op} \rightarrow \mathscr{M}^1(k) \otimes \mathbb{Q},
\end{equation*}
unique up to canonical isomorphism, such that we have a natural isomorphism
\begin{equation*}
V_\ell M^{2d-1}(X) \cong H^{2d-1}(X,\mathbb{Q}_\ell(d-1))
\end{equation*}
for all $\ell \not= p$.
\end{thm}
\begin{proof}
This is simply a summary of our work in this section.
\end{proof}
\section{Application to Independence of $\ell$}

Let $k$ be an algebraically closed field and $f: X \rightarrow X$ an endomorphism of a separated finite type $k$-scheme (we are primarily thinking of the case $k = \overline{\mathbb{F}}_q$ and $f$ is the geometric Frobenius endomorphism of a scheme $X$ defined over $\mathbb{F}_q$).  Then for any $i$ and $\ell \not= p$ we can define
\begin{equation*}
P_{\ell}^i(f,t) := \mathrm{det}(1 - tf \vert H^i(X,\mathbb{Q}_\ell))
\end{equation*}
and in case $f$ is proper,
\begin{equation*}
P_{\ell,c}^i(f,t) := \mathrm{det}(1 - tf \vert H_c^i(X,\mathbb{Q}_\ell)).
\end{equation*}
An old conjecture is that these polynomials have integer coefficients independent of $\ell$.  Based on our work on 1-motives, we can prove the following:
\begin{prop}\label{lindependence}
Let $f: X \rightarrow X$ be an endomorphism of a separated finite type $k$-scheme.  Then the polynomials $P_{\ell}^i(f,t)$ have integer coefficients independent of $\ell$ for $i = 0,1,2d- 1,2d$.  If $f$ is proper, then the same holds true for the polynomials $P_{\ell,c}^i(f,t)$ for the same values of $i$.
\end{prop}
\begin{proof}
First we handle the cases $i = 0$ and $i = 2d$.  Let $C(X)$ and $PC(X)$ be, respectively, the sets of connected components and proper connected components of $X$.  Also, let $I_d(X)$ and $PI_d(X)$ be, respectively, the sets of $d$-dimensional irreducible components and $d$-dimensional proper irreducible components of $X$.  It is clear that we have functorial isomorphisms $H^0(X,\mathbb{Q}_{\ell}) \cong \mathbb{Q}_{\ell}^{C(X)}$ and $H_c^0(X,\mathbb{Q}_\ell) \cong \mathbb{Q}_{\ell}^{PC(X)}$.  This proves the case $i =0$.   By Lemmas \ref{tracemap} and \ref{lem87}, we have functorial isomorphisms $H^{2d}(X,\mathbb{Q}_\ell(d)) \cong \mathbb{Q}_\ell^{PI_d(X)}$ and $H_c^{2d}(X,\mathbb{Q}_\ell(d)) \cong \mathbb{Q}_\ell^{I_d(X)}$.  This deals with the case $i = 2d$.  We have shown that $M_{(c)}^{i}(X)$ is the realization of a natural 1-motive for $i = 1,2d-1$, so we see that the following proposition completes the proof.
\end{proof}
\begin{prop}
Let $M = [L \rightarrow G]$ be a 1-motive over $k$, and let $f: M \rightarrow M$ be an endomorphism of $M$.  For any $\ell \not= p$ define the polynomial
\begin{equation*}
P_{\ell}^i(t) := \mathrm{det}(1 - tf \vert V_\ell M).
\end{equation*}
Then $P_\ell^i(t)$ has integer coefficients independent of $\ell$. 
\end{prop}
\begin{proof}
The endomorphism $f$ induces a commutative diagram
\begin{equation*}
\begin{CD}
0 @>>> V_\ell G @>>> V_\ell M @>>> V_\ell L @>>> 0 \\
@. @V f VV @V f VV @V f VV \\
0 @>>> V_\ell G @>>> V_\ell M @>>> V_\ell L @>>> 0,
\end{CD}
\end{equation*}
so it suffices to prove the proposition individually for $V_\ell G$ and $V_\ell L$.  Since $V_\ell L = L \otimes \mathbb{Q}_\ell$, the statement is clear for $V_\ell L$.  For $V_\ell G$, let $T$ be the torus part of $G$ and $A$ the abelian quotient.  Then $f$ induces a diagram
\begin{equation*}
\begin{CD}
0 @>>> V_\ell T @>>> V_\ell G @>>> V_\ell A @>>> 0 \\
@. @V f VV @V f VV @V f VV \\
0 @>>> V_\ell T @>>> V_\ell G @>>> V_\ell A @>>> 0
\end{CD}
\end{equation*}
So it suffices to prove the proposition individually for a torus $T$ and an abelian variety $A$, where both cases are well known \cite[p. 96]{Dem}.
\end{proof}

\begin{pg}
Now consider the case when $X$ is 2-dimensional.  Then for any endomorphism $f: X \rightarrow X$, we have proved $\ell$-independence for $P_{\ell,(c)}^i(f,t)$ for all $i$ except $i = 2$.  But this single remaining value of $i$ can be dealt with by the trace formula (for certain $f$).  We obtain the following:
\begin{cor}\label{surfaces}
Let $X$ be a 2-dimensional separated finite type $k$-scheme.  If $f: X \rightarrow X$ is any proper endomorphism, then for all values of $i$, the polynomial $P_{\ell,c}^i(t)$ has rational coefficients independent of $\ell$.  If $f: X \rightarrow X$ is any quasi-finite endomorphism, then the polynomial $P_{\ell}^i(t)$ has rational coefficients independent of $\ell$ for all $i$.
\end{cor}
\begin{proof}
The statement for $P_{\ell,c}^i(t)$ follows from the trace formula on compactly supported cohomology, known as Fujiwara's theorem \cite[5.4.5]{fuj}.  The statement for $P_\ell^i(t)$ follows from a trace formula for quasi-finite morphisms \cite[Thm 1.1]{chern}.
\end{proof}
\end{pg}
\providecommand{\bysame}{\leavevmode\hbox
to3em{\hrulefill}\thinspace}

\end{document}